\documentclass[12pt]{amsart}
\usepackage{mathrsfs}
\usepackage{cases}
\usepackage{epic}
\usepackage{amsfonts}
\usepackage{tikz}
\usepackage{graphicx}
\usepackage{amsmath}
\usepackage{amssymb, upgreek}
\usepackage{bm}
\usepackage{latexsym,todonotes}
\usepackage{pdflscape}
\usepackage[all]{xypic}
\usepackage[all]{xy}
\usepackage{color}
\usepackage{colordvi}
\usepackage{multicol}
\usepackage[normalem]{ulem}
\usepackage[linktocpage=true]{hyperref}
\hypersetup{colorlinks,linkcolor=blue,urlcolor=cyan,citecolor=blue}

\allowdisplaybreaks

\begin{document}
\input xy
\xyoption{all}
\newcommand{\iLa}{\Lambda^{\imath}}
\newcommand{\iadd}{\operatorname{iadd}\nolimits}

\renewcommand{\mod}{\operatorname{mod}\nolimits}
\newcommand{\fproj}{\operatorname{f.proj}\nolimits}
\newcommand{\Fac}{\operatorname{Fac}\nolimits}
\newcommand{\ci}{{\I}_{\btau}}
\newcommand{\proj}{\operatorname{proj}\nolimits}
\newcommand{\inj}{\operatorname{inj}\nolimits}
\newcommand{\rad}{\operatorname{rad}\nolimits}
\newcommand{\Span}{\operatorname{Span}\nolimits}
\newcommand{\soc}{\operatorname{soc}\nolimits}
\newcommand{\ind}{\operatorname{inj.dim}\nolimits}
\newcommand{\Ginj}{\operatorname{Ginj}\nolimits}
\newcommand{\res}{\operatorname{res}\nolimits}
\newcommand{\np}{\operatorname{np}\nolimits}
\newcommand{\Mor}{\operatorname{Mor}\nolimits}
\newcommand{\Mod}{\operatorname{Mod}\nolimits}
\newcommand{\End}{\operatorname{End}\nolimits}
\newcommand{\lf}{\operatorname{l.f.}\nolimits}
\newcommand{\Iso}{\operatorname{Iso}\nolimits}
\newcommand{\Aut}{\operatorname{Aut}\nolimits}
\newcommand{\Rep}{\operatorname{Rep}\nolimits}

\newcommand{\colim}{\operatorname{colim}\nolimits}
\newcommand{\gldim}{\operatorname{gl.dim}\nolimits}
\newcommand{\cone}{\operatorname{cone}\nolimits}
\newcommand{\rep}{\operatorname{rep}\nolimits}
\newcommand{\Ext}{\operatorname{Ext}\nolimits}
\newcommand{\Tor}{\operatorname{Tor}\nolimits}
\newcommand{\Hom}{\operatorname{Hom}\nolimits}
\newcommand{\Top}{\operatorname{top}\nolimits}
\newcommand{\Coker}{\operatorname{Coker}\nolimits}
\newcommand{\thick}{\operatorname{thick}\nolimits}
\newcommand{\rank}{\operatorname{rank}\nolimits}
\newcommand{\Gproj}{\operatorname{Gproj}\nolimits}
\newcommand{\Len}{\operatorname{Length}\nolimits}
\newcommand{\RHom}{\operatorname{RHom}\nolimits}
\renewcommand{\deg}{\operatorname{deg}\nolimits}
\renewcommand{\Im}{\operatorname{Im}\nolimits}
\newcommand{\Ker}{\operatorname{Ker}\nolimits}
\newcommand{\Coh}{\operatorname{Coh}\nolimits}
\newcommand{\Id}{\operatorname{Id}\nolimits}
\newcommand{\Qcoh}{\operatorname{Qch}\nolimits}
\newcommand{\CM}{\operatorname{CM}\nolimits}
\newcommand{\sgn}{\operatorname{sgn}\nolimits}
\newcommand{\utMH}{\operatorname{\cm\ch(\iLa)}\nolimits}

\newcommand{\bb}{\psi_*}
\newcommand{\bvs}{{\boldsymbol{\varsigma}}}
\newcommand{\vs}{\varsigma}

\newcommand{\e}{{\bf 1}}
\newcommand{\EE}{E^*}
\newcommand{\dbl}{\operatorname{dbl}\nolimits}
\newcommand{\ga}{\gamma}
\newcommand{\tM}{\cm\widetilde{\ch}}
\newcommand{\la}{\lambda}

\newcommand{\For}{\operatorname{{\bf F}or}\nolimits}
\newcommand{\coker}{\operatorname{Coker}\nolimits}
\newcommand{\rankv}{\operatorname{\underline{rank}}\nolimits}
\newcommand{\dimv}{{\operatorname{\underline{dim}\,}\nolimits}}
\newcommand{\diag}{{\operatorname{diag}\nolimits}}
\newcommand{\swa}{{\operatorname{swap}\nolimits}}

\renewcommand{\Vec}{{\operatorname{Vec}\nolimits}}
\newcommand{\pd}{\operatorname{proj.dim}\nolimits}
\newcommand{\gr}{\operatorname{gr}\nolimits}
\newcommand{\id}{\operatorname{id}\nolimits}
\newcommand{\aut}{\operatorname{Aut}\nolimits}
\newcommand{\Gr}{\operatorname{Gr}\nolimits}

\newcommand{\pdim}{\operatorname{proj.dim}\nolimits}
\newcommand{\idim}{\operatorname{inj.dim}\nolimits}
\newcommand{\Gd}{\operatorname{G.dim}\nolimits}
\newcommand{\Ind}{\operatorname{Ind}\nolimits}
\newcommand{\add}{\operatorname{add}\nolimits}
\newcommand{\pr}{\operatorname{pr}\nolimits}
\newcommand{\oR}{\operatorname{R}\nolimits}
\newcommand{\oL}{\operatorname{L}\nolimits}
\newcommand{\Perf}{{\mathfrak Perf}}
\newcommand{\cc}{{\mathcal C}}
\newcommand{\gc}{{\mathcal GC}}
\newcommand{\ce}{{\mathcal E}}
\newcommand{\calI}{{\mathcal I}}
\newcommand{\cs}{{\mathcal S}}
\newcommand{\cf}{{\mathcal F}}
\newcommand{\cx}{{\mathcal X}}
\newcommand{\cy}{{\mathcal Y}}
\newcommand{\ct}{{\mathcal T}}
\newcommand{\cu}{{\mathcal U}}
\newcommand{\cv}{{\mathcal V}}
\newcommand{\cn}{{\mathcal N}}
\newcommand{\mcr}{{\mathcal R}}
\newcommand{\ch}{{\mathcal H}}
\newcommand{\ca}{{\mathcal A}}
\newcommand{\cb}{{\mathcal B}}
\newcommand{\cj}{{\mathcal J}}
\newcommand{\cl}{{\mathcal L}}
\newcommand{\cm}{{\mathcal M}}
\newcommand{\cp}{{\mathcal P}}
\newcommand{\cg}{{\mathcal G}}
\newcommand{\cw}{{\mathcal W}}
\newcommand{\co}{{\mathcal O}}
\newcommand{\cq}{{\mathcal Q}}
\newcommand{\cd}{{\mathcal D}}
\newcommand{\ck}{{\mathcal K}}
\newcommand{\calr}{{\mathcal R}}
\newcommand{\cz}{{\mathcal Z}}
\newcommand{\ol}{\overline}
\newcommand{\ul}{\underline}
\newcommand{\st}{[1]}
\newcommand{\ow}{\widetilde}
\renewcommand{\P}{\mathbf{P}}
\newcommand{\pic}{\operatorname{Pic}\nolimits}
\newcommand{\Spec}{\operatorname{Spec}\nolimits}

%Theorem for the introduciton
\newtheorem{innercustomthm}{{\bf Theorem}}
\newenvironment{customthm}[1]
  {\renewcommand\theinnercustomthm{#1}\innercustomthm}
  {\endinnercustomthm}

  \newtheorem{innercustomcor}{{\bf Corollary}}
\newenvironment{customcor}[1]
  {\renewcommand\theinnercustomcor{#1}\innercustomcor}
  {\endinnercustomthm}

  \newtheorem{innercustomprop}{{\bf Proposition}}
\newenvironment{customprop}[1]
  {\renewcommand\theinnercustomprop{#1}\innercustomprop}
  {\endinnercustomthm}

\newtheorem{theorem}{Theorem}[section]
\newtheorem{acknowledgement}[theorem]{Acknowledgement}
\newtheorem{algorithm}[theorem]{Algorithm}
\newtheorem{axiom}[theorem]{Axiom}
\newtheorem{case}[theorem]{Case}
\newtheorem{claim}[theorem]{Claim}
\newtheorem{conclusion}[theorem]{Conclusion}
\newtheorem{condition}[theorem]{Condition}
\newtheorem{conjecture}[theorem]{Conjecture}
\newtheorem{construction}[theorem]{Construction}
\newtheorem{corollary}[theorem]{Corollary}
\newtheorem{criterion}[theorem]{Criterion}
\newtheorem{definition}[theorem]{Definition}
\newtheorem{example}[theorem]{Example}
\newtheorem{assumption}[theorem]{Assumption}
\newtheorem{lemma}[theorem]{Lemma}
\newtheorem{notation}[theorem]{Notation}
\newtheorem{problem}[theorem]{Problem}
\newtheorem{proposition}[theorem]{Proposition}
\newtheorem{solution}[theorem]{Solution}
\newtheorem{summary}[theorem]{Summary}
\newtheorem{hypothesis}[theorem]{Hypothesis}
\newtheorem*{thm}{Theorem}

\theoremstyle{remark}
\newtheorem{remark}[theorem]{Remark}

\newcommand{\tK}{\widetilde{K}}

\newcommand{\tk}{\widetilde{k}}
\newcommand{\tU}{\widetilde{{\mathbf U}}}
\newcommand{\Ui}{{\mathbf U}^\imath}
\newcommand{\tUi}{\widetilde{{\mathbf U}}^\imath}
\newcommand{\qbinom}[2]{\begin{bmatrix} #1\\#2 \end{bmatrix} }
\newcommand{\ov}{\overline}
\newcommand{\tMHg}{\operatorname{\widetilde{\ch}(Q,\btau)}\nolimits}
\newcommand{\tMHgop}{\operatorname{\widetilde{\ch}(Q^{op},\btau)}\nolimits}

\newcommand{\rMHg}{\operatorname{\ch_{\rm{red}}(Q,\btau)}\nolimits}
\newcommand{\dg}{\operatorname{dg}\nolimits}
\def \fu{{\mathfrak{u}}}
\def \fv{{\mathfrak{v}}}
\def \sqq{{\bf t}}
\def \bp{{\mathbf p}}
\def \bv{{\mathbf v}}
\def \bw{{\mathbf w}}
\def \bA{{\mathbf A}}
\def \bL{{\mathbf L}}
\def \bF{{\mathbf F}}
\def \bS{{\mathbf S}}
\def \bC{{\mathbf C}}
\def \bU{{\mathbf U}}
\def \U{{\mathbf U}}
\def \btau{\varpi}
\def \La{\Lambda}
\def \Res{\Delta}
\newcommand{\ev}{\bar{0}}
\newcommand{\odd}{\bar{1}}
\def \fk{\mathfrak{k}}
\def \ff{\mathfrak{f}}
\def \fp{{\mathfrak{P}}}
\def \fg{\mathfrak{g}}
\def \fn{\mathfrak{n}}
\def \gr{\mathfrak{gr}}
\def \Z{{\Bbb Z}}
\def \F{{\Bbb F}}
\def \D{{\Bbb D}}
\def \C{{\Bbb C}}
\def \N{{\Bbb N}}
\def \Q{{\Bbb Q}}
\def \G{{\Bbb G}}
\def \P{{\Bbb P}}
\def \K{{k}}
\def \E{{\Bbb E}}
\def \I{{\Bbb I}}

\def \BH{{\Bbb H}}
\def \btau{\varrho}
\def \cv{\varpi}

\def \tR{\widetilde{\bf R}}
\def \tRZ{\widetilde{\bf R}_\Z}
\def\tRi{\widetilde{\bf R}^\imath}
\def\tRiZ{\widetilde{\bf R}^\imath_\Z}

\def \bp{{\mathbf p}}
\def \ts{\texttt{S}}
\def \bq{{\bm q}}
\def \bvt{{v}}
\def \bs{{\bm s}}
\def \tt{{v}}

\newcommand{\browntext}[1]{\textcolor{brown}{#1}}
\newcommand{\greentext}[1]{\textcolor{green}{#1}}
\newcommand{\redtext}[1]{\textcolor{red}{#1}}
\newcommand{\bluetext}[1]{\textcolor{blue}{#1}}
\newcommand{\brown}[1]{\browntext{ #1}}
\newcommand{\green}[1]{\greentext{ #1}}
\newcommand{\red}[1]{\redtext{ #1}}
\newcommand{\blue}[1]{\bluetext{ #1}}
\numberwithin{equation}{section}
\renewcommand{\theequation}{\thesection.\arabic{equation}}

%todo
\newcommand{\wtodo}{\todo[inline,color=orange!20, caption={}]}
\newcommand{\lutodo}{\todo[inline,color=green!20, caption={}]}
\def \tT{\widetilde{\mathcal T}}

\def \tTL{\tT(\iLa)}

\title[Hall algebras and quantum symmetric pairs III]{Hall algebras and quantum symmetric pairs III: \\ Quiver varieties}

\author[Ming Lu]{Ming Lu}
\address{Department of Mathematics, Sichuan University, Chengdu 610064, P.R.China}
\email{luming@scu.edu.cn}

\author[Weiqiang Wang]{Weiqiang Wang}
\address{Department of Mathematics, University of Virginia, Charlottesville, VA 22904, USA}
\email{ww9c@virginia.edu}


\begin{abstract}
The $\imath$quiver algebras were introduced recently by the authors to provide a Hall algebra realization of universal $\imath$quantum groups, which is a generalization of Bridgeland's Hall algebra construction for (Drinfeld doubles of) quantum groups; here an $\imath$quantum group and a corresponding Drinfeld-Jimbo quantum group form a quantum symmetric pair. In this paper, the Dynkin $\imath$quiver algebras are shown to arise as new examples of singular Nakajima-Keller-Scherotzke categories. Then we provide a geometric construction of the universal $\imath$quantum groups and their ``dual canonical bases" with positivity, via the quantum Grothendieck rings of Nakajima-Keller-Scherotzke quiver varieties, generalizing Qin's geometric realization of quantum groups of type ADE.
\end{abstract}

\maketitle
 \tableofcontents

 %%%%%%%
 %%%%%%%
\section{Introduction}

\subsection{Background (H)} %Algebraic background}

Bridgeland \cite{Br13} provided a Hall algebra construction of the Drinfeld double $\tU$ of a quantum group $\U$; see Ringel \cite{Rin90}, Lusztig \cite{Lus90} and Green \cite{Gr95} on Hall algebra constructions of half a quantum group $\U^+$.  As further generalizations of Bridgeland's construction, Gorsky \cite{Gor13, Gor18} introduced {\em semi-derived Hall algebras} via $\Z/2$-graded complexes of an exact category while the first author and Peng \cite{LP21, Lu19} formulated the {\em semi-derived Ringel-Hall algebras} for 1-Gorenstein algebras.

An $\imath$quiver $(Q, \btau)$ by definition consists of an acyclic quiver $Q =(Q_0, Q_1)$ together with an involution $\btau$ on $Q$ (here $\btau =\id$ is allowed).
Associated to an $\imath$quiver $(Q, \btau)$, the authors  \cite{LW19} recently formulated a family of finite-dimensional 1-Gorenstein algebras, called $\imath$quiver algebras and denoted by $\iLa$. The semi-derived Ringel-Hall algebras for $\iLa$ (or simply the $\imath$Hall algebras) were then shown to provide a categorical realization of the quasi-split {\em universal} $\imath$quantum groups $\tUi$; the Drinfeld double is reproduced associated to the $\imath$quiver of diagonal type. The $\imath$Hall algebra approach also provides a conceptual construction of a braid group action on $\tUi$ \cite{LW19b}.% (compare \cite{KP11}).

The universal $\imath$quantum group $\tUi =\langle B_i, \tk_i \mid i\in \I \rangle$ is by definition a subalgebra of $\tU$, and $(\tU, \tUi)$ forms a quantum symmetric pair; cf. \cite{LW19}. A novelty of $\tUi$ is that it admits a Cartan subalgebra generated by $\tk_i$ $(i\in \I =Q_0)$, which contains various central elements. A central reduction of $\tUi$ reproduces the $\imath$quantum group $\Ui =\Ui_\bvs$ with parameters $\bvs$ which arises from the construction of quantum symmetric pairs $(\U, \Ui)$ of G.~Letzter \cite{Let99}. Note $\tUi$ is naturally $\N \I$-graded while $\Ui$ (with its inhomogeneous $\imath$Serre relations) is not.

We view Letzter's $\imath$quantum groups and our universal $\imath$quantum groups as a vast generalization of the Drinfeld-Jimbo quantum groups. The $\imath$-program as outlined by  Bao and the second author \cite{BW18} aims at generalizing various fundamental constructions from quantum groups to $\imath$quantum groups.

\subsection{Background (G)} %Geometric background}

Via quantum Grothendieck rings of cyclic quiver varieties, Qin \cite{Qin} provided a geometric construction of the Drinfeld doubles of quantum groups of type ADE. A bonus of the geometric approach is the construction of a positive basis on $\tU$, which contains as subsets the (mildly rescaled) dual canonical bases of Lusztig for halves of the quantum groups \cite{Lus90}. Qin's work was built on the construction of Hernandez-Leclerc \cite{HL15} (who realized half a quantum group using graded Nakajima quiver varieties) as well as the concept of quantum Grothendieck rings introduced by Nakajima \cite{Na04} and Varagnolo-Vasserot \cite{VV}.

%
%\subsection{}

Motivated by Nakajima \cite{Na01}, Hernander-Leclerc \cite[\S9]{HL15} and Leclerc-Plamondon \cite{LeP13}, Keller and Scherotzke \cite{KS16} formulated the notion of regular/singular Nakajima categories $\mcr^{\gr}_C$ and $\cs^{\gr}_C$ as a variant of mesh category of the repetition category $\Z Q$ with a configuration $C$, for an acyclic quiver $Q$; a copy of the subset $C$ of vertices of $\Z Q$ plays the role of framing in Nakajima's construction of quiver varieties. The representation varieties of $\mcr^{\gr}_C$ and $\cs^{\gr}_C$ (with suitable choices of $C$) were then used to realize graded Nakajima quiver varieties and their stratifications. Scherotzke \cite{Sch18} further generalized the main results of \cite{KS16} to an orbit quotient setting by introducing $\mcr=\mcr_{C,F}:=\mcr^{\gr}_C/F$ and $\cs=\cs_{C,F}$ associated to an admissible pair $(F, C)$, where $F$ is an automorphism and $C$ is a configuration stable under $F$. Note that $\mcr, \cs$ were called the generalized Nakajima categories {\em loc. cit.}, and they are called Nakajima-Keller-Scherotzke (NKS) categories in this paper.

It is difficult to describe explicitly these NKS categories and the corresponding NKS (quiver) varieties in such a great generality; the NKS varieties are by definition (cf. Definition~\ref{def:NKS}) the representation varieties $\cm(\bv,\bw)$ and $\cm_0(\bw)$ of $\mcr$ and $\cs$. The known examples of NKS varieties are those when taking $F$ to be the AR translation $\tau$ and the square $\Sigma^2$ of the shift functor, and they reproduce classical and cyclic Nakajima quiver varieties, respectively  (cf. \cite{Sch18, SS16}; see also Proposition~\ref{prop:2-complexes}). Qin's construction \cite{Qin} of the Drinfeld double $\tU$ uses such cyclic quiver varieties; in this case, the singular NKS category $\cs$ (when viewed as an algebra) can be identified with $\La = kQ \otimes R_2$ in \eqref{eq:La}, i.e., the $\imath$quiver algebra used to formulate Bridgeland's Hall algebra; cf. \cite{LW19}.

\subsection{Goal}

The goal of this paper is to provide a geometric setting for constructing the quasi-split universal $\imath$quantum groups $\tUi$ of ADE type, using the quantum Grothendieck rings of Nakajima-Keller-Scherotzke quiver varieties (associated to some distinguished choices of admissible pairs arising from Dynkin $\imath$quivers). In addition, the geometric construction provides a favorable basis for $\tUi$ with positive integral structure constants, which should be thought as ``dual $\imath$canonical basis".

The construction in this paper should be viewed as a geometric counterpart of the $\imath$Hall algebra construction in \cite{LW19}, just as the construction of Qin \cite{Qin} (also see \cite{SS16}) is a geometric counterpart of Bridgeland's Hall algebra construction \cite{Br13}.

\subsection{Main results}

Recall the $\imath$quiver algebra $\iLa$ associated to an $\imath$quiver $(Q, \btau)$, whose $\imath$Hall algebra produces the universal $\imath$quantum group $\tUi$ \cite{LW19} .
Let $\widehat{\btau}$ be the triangulated equivalence of the derived category $\cd_Q$ of $kQ$-modules induced by the automorphism $\btau$ of $Q$. With motivation from representation theory, we single out below a distinguished class of NKS categories and NKS varieties; see Definition~\ref{def:RS for iQG}. 

\begin{customthm}{{\bf A}}  [Theorem~\ref{thm:iNKS}]
   \label{thm:A}
  Let $(Q,\btau)$ be a Dynkin $\imath$quiver. Let $\mcr^\imath$ and $\cs^\imath$ be the regular and singular $\imath$NKS categories  as in Definition~\ref{def:RS for iQG}. Then the category $\Gproj(\Lambda^\imath)$ of Gorenstein projective $\iLa$-modules is equivalent to $\proj(\mcr^\imath)$, $\proj(\Lambda^\imath)$ is Morita equivalent to $\cs^\imath$, and there are equivalences of Frobenius categories
$\Gproj(\Lambda^\imath)\simeq \Gproj(\cs^\imath)\simeq \proj(\mcr^\imath).$
\end{customthm}
In other words, $\cs^\imath$ viewed as an algebra is isomorphic to the $\imath$quiver algebra $\iLa$.

%
%\subsection{}

Following the formalism in \cite{Sch18}, we define in Definition~\ref{def:NKS} and \eqref{eq:M0}--\eqref{eq:pi} the $\imath$NKS varieties $\cm (\bv,\bw,\mcr^\imath)$ and $\cm_0(\bv,\bw,\mcr^\imath)$, together with a morphism $\pi: \cm (\bv,\bw,\mcr^\imath) \rightarrow \cm_0(\bv,\bw,\mcr^\imath)$, for given dimension vectors $\bv \in \N^{\mcr^\imath_0-\cs^\imath_0}$ and $\bw \in \N^{\cs^\imath_0}$. Thanks to the inclusion relations among $\cm_0(\bv,\bw,\mcr^\imath)$ when $\bv$ are comparable, we can form $\cm_0(\bw,\mcr^\imath) =\cup_{\bv} \cm_0(\bv,\bw,\mcr^\imath)$, which can be identified with the representation variety $\rep(\bw, \cs^\imath)$; cf. Lemma~\ref{lem:M=rep}. There is a stratification (see \eqref{eqn:stratification})
\[
\cm_0(\bw, \mcr^\imath)=\bigsqcup_{\bv:\sigma^*\bw-{\mathcal C}_q\bv\geq0} \cm_0^{\text{reg}}(\bv,\bw, \mcr^\imath).
\]

A transversal slice result for graded/cyclic quiver varieties \cite{Na01} (also cf. \cite{Na04}) has played a fundamental role in a decomposition of the perverse sheaf \eqref{eq:piPS}--\eqref{eq:decomp} into a direct summand of IC sheaves with trivial local systems; this decomposition is more precise than what the BBD decomposition theorem offers. We show that an analogous transversal slice result holds for $\imath$NKS varieties; see Proposition~\ref{prop:Translice1}. However, we have been unable to imitate Nakajima's argument completely (which relies on the indecomposability of standard modules of quantum loop algebras) for all $\imath$NKS varieties, and thus we impose a technical Hypothesis~\ref{hypothesis} throughout the paper. This Hypothesis is known to hold for a substantial subclass of $\imath$NKS varieties; see Table~\ref{table:values}. Accordingly, Theorem~\ref{thm:B}, Theorem~\ref{thm:C} and Proposition~\ref{prop:D} hold unconditionally in cases listed in Table~\ref{table:values} and in the remaining cases conditional on the validity of Hypothesis~\ref{hypothesis}.

We apply the general machinery \cite{Na04, VV} of quantum Grothendieck rings to the $\imath$NKS varieties. A coalgebra $K^{gr}(\mod(\cs^\imath))$ in \eqref{eq:Kgr}  is endowed with a twisted comultiplication (which involves a square root $\tt^{1/2}$ of $\tt$), and its dual gives rise to the quantum Grothendieck ring  $(\tRiZ, \cdot)$ over $\Z[\tt^{1/2}, \tt^{-1/2}]$; cf. \eqref{eq:KSi}--\eqref{eq:tw}. By construction, $\tRiZ$ affords a distinguished basis $\{L(\bv, \bw)\}$ dual to a basis in $K^{gr}(\mod(\cs^\imath))$ of perverse sheaves associated to the strata $\cm_0^{\text{reg}}(\bv,\bw, \mcr^\imath)\subseteq \cm_0(\bw, \mcr^\imath)$ with respect to the trivial local systems, cf. \eqref{eq:bases dual}.

We have a good knowledge of the size of $\tRiZ$ thanks to the following theorem. The parametrization of this basis for $\tRiZ$ is fairly explicit; cf. \eqref{eq:Ri0Z} and Proposition~\ref{prop:dominant++}. See \eqref{eq:vfi} for notations $\bv^i, \bw^i$.
\begin{customthm}{\bf B}   [Theorem~\ref{thm:RRR}, Proposition~\ref{prop:poly}]
  \label{thm:B}
The subring $\tR^{\imath0}_\Z$ of $\,\tRiZ$ generated by $L(\bv^{i},\bw^{i})$, for $i\in \I$, is a polynomial ring, and it has a basis
\[
\{ L(\bv^0,\bw^0) \mid (\bv^0,\bw^0)\in (V^0,W^0),\sigma^*\bw^0-{\mathcal C}_q \bv^0=0 \}.
%\{ L(\bv,\bw) \mid (\bv,\bw)\in (V^0,W^0),\sigma^*\bw-{\mathcal C}_q \bv=0 \}.
\]
Moreover, as a left (or right) $\tR^{\imath0}_\Z$-module, $\tRiZ$ is free with basis
\[
\{L(\bv^+,\bw^+)\mid (\bv^+,\bw^+)\in (V^{+},W^{+})\text{ $l$-dominant} \}.
\]
\end{customthm}

Now we can state the main theorem of this paper, where $\tRi =\Q(\tt^{1/2}) \otimes \tRiZ$ and the field $\Q({\tt^{1/2}})$ is used to include $\tt^{1/2}$; note that $\tt^{1/2}$ arises in the twisting of (co)multiplication.
\begin{customthm}{{\bf C}}  [Theorem~\ref{thm:UiRi}]
  \label{thm:C}
Let $(Q, \btau)$ be a Dynkin $\imath$quiver. Then there exists an isomorphism of $\Q({\tt^{1/2}})$-algebras
$%\begin{align*}
\widetilde{\kappa}: \tUi \stackrel{\simeq}{\longrightarrow}  \tRi
$ %\end{align*}
which sends
\begin{align*}
  B_i\mapsto  \frac{\tt}{1-\tt^2} L(0,\e_{\sigma \ts_i}),\text{ if }i\in\I_\btau ,&\qquad B_i\mapsto \frac{\tt}{\tt^2-1}L(0,\e_{\sigma \ts_i}),\text{ if } i\in \I \backslash \I_\btau,
  \\
\tk_j\mapsto L(\bv^{j}, \bw^{j}),\text{ if } \btau j\neq j,&\qquad
\tk_j\mapsto-\frac{1}{\tt} L(\bv^{j},\bw^{j}),\text{ if } \btau j=j.
\end{align*}
\end{customthm}
It is instructive to compare the above theorem with the realization of $\U^+$ via a quantum Grothendieck ring $\tR^+$ by Hernandez-Leclerc \cite{HL15} and the realization of $\tU$ via a quantum Grothendieck ring $\tR$ by Qin \cite{Qin}; see Theorem~\ref{thm:iso HL-Q}. Theorem~\ref{thm:B} allows us to define a filtered algebra structure on $\tRi$ in Proposition~\ref{prop:filtered}, whose associated graded $\tR^{\imath,gr}$ naturally contains $\tR^+$ as a subalgebra; see Proposition~\ref{prop:R+Rigr}.

To show that $\widetilde{\kappa}$ defined in Theorem~\ref{thm:C} is a homomorphism, we need to verify the $\imath$Serre relations for the generators $B_i, \tk_i$ of $\tUi$ hold for their images in $\tRi$ under $\widetilde{\kappa}$. This requires several local rank 1 and 2 computations, which in turn are built on the determination of various $l$-dominant pairs $(\bv, \bw)$ which provide explicit information on the stratification of $\cm_0(\bw, \mcr^\imath)$. These computations are more involved than \cite{HL15, Qin} as we have more cases to handle, and will be carried out in Sections~\ref{sec:pairs}--\ref{sec:computation}. We then show in Section~\ref{sec:main} the elements in $\tRi$ as specified in Theorem~\ref{thm:C} indeed generate the algebra $\tRi$. Finally a filtered algebra argument reduces the injectivity of the homomorphism $\widetilde{\kappa}$ to the theorem of Hernandez-Leclerc which realizes geometrically half a quantum group.

%
%\subsection{}

The following positivity property follows from constructions, see \eqref{eq:positive}.

\begin{customprop}{{\bf D}}  [Positivity]
\label{prop:D}
The structure constants in the $\Z[\tt^{1/2}, \tt^{-1/2}]$-algebra $\tRiZ$ with respect to the basis $\{L(\bv,\bw)\}$ are in $\N[\tt^{1/2}, \tt^{-1/2}]$.
\end{customprop}
This positive basis for $\tRi$ is transferred to a basis for $\tUi$ via the isomorphism in Theorem~\ref{thm:C}, which
should be viewed as a {\em dual $\imath$canonical basis} for $\tUi$. For quasi-split $\tUi$ of type AIII, this dual $\imath$canonical basis in $\tUi$ contains as subsets of (rescaled) dual canonical basis of half a quantum group of type A \cite{Lus90, Lus93}.

\subsection{Perspectives}

As the Drinfeld doubles $\tU$ %(where $K_i, K_i'$ are not required to be invertible)
 arise naturally in the categorical and geometric constructions, we view $\tU$ as a basic concept in its own, not merely as a stepping stone to the construction of quantum groups $\U$.  In the same vein, the universal $\imath$quantum group $\tUi$ can be regarded as a fundamental object on its own; in addition, it helps to demystify the appearance of parameters $\bvs$ in Letzter's $\imath$quantum group $\Ui =\Ui_{\bvs}$.

It will be beneficial to relate the geometric realization of $\tUi$ in this paper and the categorical realization of $\tUi$ in \cite{LW19}. The construction of \cite{LW19} provides an intrinsic basis for $\tUi$, which we call a {\em dual PBW basis}. In a sequel to this paper, %\cite{LW19e},
 we shall show the integral lattice for $\tUi$ spanned by the dual PBW basis coincides with the integral lattice spanned by the dual $\imath$canonical basis constructed in this paper; moreover, we shall show that the dual PBW basis has a {\em positive integral} expansion in terms of the dual $\imath$canonical basis; similar statements hold for $\tU$. In addition, we shall develop an algebraic approach of another ``dual canonical basis" for $\tU$ and $\tUi$, respectively. It is a challenging and important open question to find an algebraic characterization and algorithm for the positive dual canonical bases for $\tU$ and $\tUi$.

There has been a completely different geometric construction of the modified quantum groups $\dot\U$ (of type A) and their canonical bases with positivity \cite{BLM90, Lus93}, which is compatible with the Khovanov-Lauda-Rouquier (KLR)  categorification \cite{KL10, R08}. There has also been a geometric construction of the modified $\imath$quantum groups $\dot\U^\imath$ (of type AIII) and their $\imath$canonical bases with positivity \cite{BKLW, LiW18, BW18b}, which is again compatible with a KLR type categorification \cite{BSWW}. It is natural to ask if there is any connection between these constructions for modified quantum groups and canonical bases (respectively, for $\imath$quantum groups and $\imath$canonical bases) and the geometric constructions of the Drinfeld doubles and dual canonical bases in \cite{Qin, SS16} (respectively, in this paper).

A monoidal categorification of half a quantum group $\U^+$ of type ADE using finite-dimensional module categories of quantum loop algebras was given in \cite{HL15}, which gave rise to the geometric realization of $\U^+$ via quantum Grothendick rings of graded Nakajima quiver varieties. Further connections between the HL categorification and KLR categorication have been firmly established in \cite{KKK15, Fu17}. This provides ``half an answer" to the above question on the whole quantum groups.

It will be important to extend the constructions of this paper and \cite{LW19} to realize $\imath$quantum groups associated to Satake diagrams with black nodes (cf. \cite{Let99, BW18}). In another direction, this paper hopefully brings quantum symmetric pairs a step closer to quantum cluster algebras, as there have been various connections to cluster algebras in the works \cite{HL15, KS16, Qin} and references therein.

%\red{Save for Paper 5?}
%Let us clarify a subtle point. The construction of Qin leads to a version of Drinfeld double $\tU$ where $K_i, K_i'$ are not required to be invertible, while the construction of Bridgeland requires $K_i, K_i'$ to be invertible. The current construction of modified Ringel-Hall algebra leads to $\tU$ where $K_i, K_i'$ are invertible; cf. \cite[\S8]{LW19}, but a simple variation  of the construction can provide a realization of $\tU$ where $K_i, K_i'$ are not invertible.

%
%
\subsection{The organization}
The paper is organized as follows.
In Section~\ref{sec:NKS}, the basic constructions of NKS categories/varieties are reviewed.
We then introduce and study in depth in Section~\ref{sec:iNKS} the $\imath$NKS categories/varieties associated to Dynkin $\imath$quiver varieties. %, and review their quantum Grothendieck rings. 
Connections between $\imath$NKS varieties and the NKS varieties used to realize quantum groups are formulated. We establish Theorem~\ref{thm:A}, which relates the singular $\imath$NKS category to the Dynkin $\imath$quiver algebra $\iLa$. The quantum Grothendieck ring $\tRi$ is then formulated.

In Section~\ref{sec:pairs}, we determine explicitly the $l$-dominant pairs which are used  in the subsequent Section~\ref{sec:computation} for the rank 2 computations of $\imath$Serre relations in the algebra $\tRi$.

In Section~\ref{sec:main}, we construct a new basis (see Theorem~\ref{thm:B}) and an algebra filtration for the quantum Grothendieck ring $\tRi$, and finally establish the algebra isomorphism in Theorem~\ref{thm:C} which provides a geometric realization of the $\imath$quantum groups.

\subsection{Notations}
We list the notations which are often used throughout the paper.
\smallskip

$\triangleright$
$\N,\Z,\Q$, $\C$ -- sets of nonnegative integers, integers, rational  and complex numbers.

$\triangleright$
$k$ -- an algebraically closed field of characteristic zero,

$\triangleright$
$\mod (k)$ -- category of finite-dimensional $k$-linear spaces,

$\triangleright$
$D$ -- the standard duality functor $\Hom_k(-,k)$.

\vspace{2mm}

Let $\cc$ be a $k$-linear category. We denote

$\triangleright$ $\cc_0$-- the set of objects of $\cc$,

$\triangleright$
$\mod(\cc)$ -- category of pointwise finite-dimensional right $\cc$-modules, i.e., $k$-linear functors $\cc^{op}\rightarrow \mod (k)$,

$\triangleright$
$\rep(\bw,\cc)$ -- the variety of $\cc$-modules $M$ such that $M(u)=k^{{\bw}(u)}$ for each $u\in \cc_0$,

$\triangleright$
$\cc(u,v)$ -- Hom-space $\Hom_\cc(u,v)$ for any $u,v\in\cc_0$,

$\triangleright$
$\proj (\cc)$ --  full subcategory of projective modules in $\mod(\cc)$.

\vspace{2mm}

Let $Q =(Q_0, Q_1)$ be a quiver, and $i\in Q_0$. Denote

$\triangleright$
$\rep(Q)$ -- category of finite-dimensional representation of $Q$ over field $k$,

$\triangleright$
$\mod (kQ)$ -- category of finite-dimensional right $kQ$-modules; we identify $\mod(kQ)\simeq \rep(Q^{op})$,

$\triangleright$
$\proj(kQ)$ -- subcategory of projective $kQ$-modules,

$\triangleright$
$\inj(kQ)$ -- subcategory of  injective $kQ$-modules,

$\triangleright$
$\ts_i$, $P_i$, $I_i$ -- the simple, projective, injective $kQ$-modules, respectively,

$\triangleright$
$\cd_Q$ -- bounded derived category of finite-dimensional $kQ$-modules.

\vspace{2mm}
For a triangulated category $\ct$, denote by

$\triangleright$
$\Sigma$ -- the shift functor,

$\triangleright$
$\tau$ -- the Auslander-Reiten (AR) translation (if it exists).

\vspace{2mm}
$\triangleright$
$\mcr^\imath, \cs^\imath, \cp^\imath$ -- the singular/regular/preprojective $\imath$NKS categories,

$\triangleright$
$\cm (\bv, \bw, \mcr^\imath)$ -- $\imath$NKS quiver varieties.

\vspace{2mm}
$\triangleright$
$\tU$ -- Drinfeld double of a quantum group $\U$,

$\triangleright$
$\tUi$ -- universal $\imath$quantum group, whose central reductions produce $\imath$quantum groups $\Ui_\bvs$.

\vspace{2mm}
$\triangleright$
$\bv^i, \bw^i$ -- dimension vectors in  \eqref{eq:vfi},

$\triangleright$
$\bv^{ij}, \bw^{ij}$, $\bw^{ijk}$ -- dimension vectors in  \eqref{eq:vij2}  and \eqref{eq:wijk}.

\subsection*{Acknowledgments}
We thank Mark de Cataldo, Hiraku Nakajima, and Fan Qin for helpful discussions on quiver varieties. 
This project was initiated during a visit of M.L. to University of Virginia. We thank University of Virginia, Sichuan University, and East China Normal University for hospitality and support which has greatly facilitated this work. W.W. is partially supported by NSF grant DMS-1702254 and DMS-2001351. We thank two anonymous referees for very helpful suggestions and comments.

%%%%%%%
%%%%%%%
\section{NKS categories and NKS quiver varieties}
   \label{sec:NKS}

In this section, we review the general constructions of NKS categories and NKS quiver varieties, following \cite{KS16, Sch18}. %We also recall the quantum Grothendieck rings for quiver varieties, following \cite{Na04, VV} (also cf. \cite{Qin, SS16}). 
\subsection{Graded NKS categories}
Let $Q=(Q_0=\I,Q_1)$ be a Dynkin quiver. Define the \emph{repetition quiver} $\Z Q$ of $Q$ as follows:

$\triangleright$ the set of vertices is $\{(i,p)\in Q_0\times \Z\}$;

$\triangleright$ an arrow $(\alpha,p):(i,p)\rightarrow (j,p)$ and an arrow $(\bar{\alpha},p):(j,p-1)\rightarrow (i,p)$ are given, for any arrow $\alpha:i\rightarrow j$ in $Q$ and any vertex $(i,p)$.

Define the automorphism $\tau$ of $\Z Q$ to be the shift by one unit to the left, i.e., $\tau(i,p)=(i,p-1)$ for all $(i,p)\in Q_0\times \Z$.

%Equivalently, $\Z \Qf$ is obtained from $\Z Q$ by adding a new vertex $\sigma(x)$, and two new arrows $\tau(x)\rightarrow \sigma(x)$, $\sigma(x)\rightarrow x$  for each vertex $x$ of $\Z Q$.

%The \emph{framed quiver} $Q^{\mathfrak{f}}$ is obtained from $Q$ such that
%\begin{itemize}
%\item $Q^{\mathfrak{f}}_0=Q_0\cup\{i'\mid i\in Q_0\}$;
%\item $Q^{\mathfrak{f}}_1=Q_1\cup\{i\rightarrow i'\mid i\in Q_0\}$.
%\end{itemize}
%Let $\Z \Qf$ be the \emph{repetition quiver }of $\Qf$. Note that $\Z Q$ is a subquiver of $\Z \Qf$. We refer to the vertices $(i',p)\in Q_0\times\Z$ as the \emph{frozen vertices} of $\Z \Qf$.
%Define $\sigma$ to be a bijection of $Q^{\mathfrak{f}}_0$ by $\sigma(i,p)=(i,p-1)$ and $\sigma(i',p)=(i,p)$ for all $i\in Q_0$, $p\in\Z$.

By a slight abuse of notation, associated to $\beta:y\rightarrow x$ in $\Z Q$, we denote by $\bar{\beta}$ the arrow that runs from $\tau x\rightarrow y$. Let $k(\Z Q)$ be the \emph{mesh category} of $\Z Q$, that is, the objects are given by the vertices of $\Z Q$ and the morphisms are $k$-linear combinations of paths modulo the ideal spanned by the mesh relations
$R_x:=\sum_{\alpha:y\rightarrow x}\alpha\bar{\alpha},$
where the sum runs through all arrows of $\Z Q$ ending at $x$.

By a theorem of Happel \cite{Ha2}, there is an equivalence
\begin{align}
	\label{eq:Happelfunctor}
H: k(\Z Q)\stackrel{\simeq}{\longrightarrow} \Ind \cd_Q,
\end{align}
where $\Ind \cd_Q$ denotes the category of indecomposable objects in $\cd_Q$.
Using this equivalence, we label once and for all the vertices of $\Z Q$ by the isoclasses of indecomposable objects of $\cd_Q$. %The functor $H$ is an equivalence if and only if $Q$ is of Dynkin type.
Note that the action of $\tau$ on $\Z Q$ corresponds to the action of the AR-translation on $\cd_Q$, and this explains the notation $\tau$.

Let $F:\cd_Q\rightarrow \cd_Q$ be a triangulated isomorphism induced by an isomorphism $F$ of $\Z Q$; \eqref{eq:Happelfunctor}. Then $F$ induces a $k$-linear isomorphism of $k(\Z Q)$, which is also denoted by $F$. Let $C$ be a subset of vertices of $\Z Q$, which satisfies the following basic assumptions.
\begin{assumption}
\begin{itemize}
\item[(C1)] $C$ is an $F$-invariant subset, and $F^n$ is not the identity on objects for all $n\in\Z$;
\item[(C2)] The orbit category $\cd_Q/F$ is triangulated \`a la Keller \cite{Ke05} and the canonical projection
$\cd_Q\rightarrow \cd_Q/F$ is a triangulated functor;
\item[(C3)] For each vertex $x\in\Z Q$ there is a vertex $c\in C$ such that $k(\Z Q)(x,c)\neq0$.
\end{itemize}
\end{assumption}
 In this case, $C$ is called an \emph{(admissible) configuration}, and $(F,C)$ an \emph{admissible pair}. For example, the set of all vertices of $\Z Q$ is admissible. Such $(F,C)$ is also an admissible pair in the sense of \cite{Sch18}, that is, it satisfies \cite[Assumption 2.5]{Sch18} thanks to \cite[\S3.3]{KS16}.

%Let $Q$ be a Dynkin quiver with an involution $\btau:Q\rightarrow Q$, and $\hat{\btau}$ be the involution of $\cd_Q$ induced by $\btau$; see \S\ref{subsec:iquiver}. Let $\hat{\btau}$ be the involution of $\cd_Q$ induced by $\btau$. 
%Throughout the paper, we always assume that 
%\begin{align}
%C=\{\Sigma^j S_i\mid i\in Q_0,j\in\Z\},\qquad F=\Sigma^2\text{ or }\Sigma\circ\hat{\btau}.
%\end{align} 

Let $\Z Q_C$ be the quiver constructed from $\Z Q$ by adding to every vertex $c \in C$ a new object denoted by $\sigma c$ together with arrows $\tau c \rightarrow \sigma c$ and $\sigma c \rightarrow c$; we refer to $\sigma c$, for $c\in C$, as {\em frozen vertices}.

The \emph{graded NKS category} $\mcr^{\gr}_C$ is defined to be the $k$-linear category with
\begin{itemize}
\item[$\triangleright$] objects: the vertices of $\Z Q_C$;
\item[$\triangleright$] morphisms: $k$-linear combinations of paths in $\Z Q_C$ modulo the ideal spanned by the mesh relations
$\sum_{\alpha:y\rightarrow x}\alpha\bar{\alpha},$
where the sum runs through all arrows of $\Z Q_C$ ending at $x\in\Z Q$ (including the new arrow $\sigma x \rightarrow x$ if $x\in C$).
\end{itemize}
These categories were formulated in \cite{KS16, Sch18}; here and below NKS stands for Nakajima-Keller-Scherotzke. The work \cite{KS16} was in turn motivated by \cite{Na01,Na04,HL15,LeP13}; also cf. \cite{Qin}.

\subsection{NKS categories}
\label{subsec:NKS cat}
The isomorphism $F$ of $\Z Q$ can be uniquely lifted to an isomorphism of $\Z Q_C$  by setting $F(\sigma c)=\sigma(Fc)$ for any $c\in C$, and then the functor $F$ of the mesh category $k(\Z Q)$ can be uniquely lifted to $\mcr^{\gr}_C$, which is also denoted by $F$.
Let
\[
\mcr=\mcr_{C,F}:=\mcr^{\gr}_C/F,
\]
and let $\cs=\cs_{C,F}$ be the full subcategory of $\mcr$ formed by all $\sigma c$ ($c\in C$), following \cite{Sch18}. Then $\mcr$ and $\cs$ are called the \emph{regular NKS category} and the \emph{singular NKS category} of the admissible configuration $(F,C)$ in this paper. %If $C$ is the set of all vertices of $\Z Q$, we denote $\mcr:=\mcr_{C,F}$ by $\mcr^{\ff}:=\mcr^\ff_F$, and denote $\cs=\cs_{C,F}$ by $\cs^\ff:=\cs^\ff_F$.
The categories studied in \cite{Na01,Na04,LeP13,HL15,Qin} are examples of (graded) NKS categories when $C$ is the set of all vertices of $\Z Q$.

Define $\Z Q_C/F$ to be the quiver with vertices the $F$-orbits on the set of vertices of $\Z Q_C$ and the number of arrows $x\rightarrow y$ between two fixed representatives $x$ and $y$ of $F$-orbits is given by the total number of arrows $x\rightarrow F^i y$ for all $i\in \Z$. The quotient category 
$$\cp:=\mcr/( \cs ),$$ which is equivalent to $k(\Z Q/F)$ is called the \emph{preprojective NKS category}.  By our assumption, $\cd_Q/F$ is a triangulated category and $\Ind \cd_Q/F\simeq \cp$. 

\begin{example}[Type $A_1$]
   \label{example sl2}
Let $Q$ be the quiver of type $A_1$. Let $C=\{\Sigma^j S\mid j\in\Z\}$ and $F=\Sigma^2$. Then $\Z Q_C/F$ is given by an oriented cycle in
the following:
\[
\xymatrix{S\ar[d]^{\alpha_1} & \sigma(S) \ar[l]_{\beta_1} \\
\sigma(\Sigma S) \ar[r]^{\beta_2}& \Sigma S\ar[u]^{\alpha_2} }
\]
%$$\sigma S\xrightarrow{\beta_1} S \xrightarrow{\alpha_1} \sigma(\Sigma S) \xrightarrow{\beta_2} \Sigma S\xrightarrow{\alpha_2} \sigma S.$$
The regular NKS category $\mcr$ is isomorphic to $k(\Z Q_C/F)/ ( \beta_2\alpha_1,\beta_1\alpha_2)$.
The singular NKS category $\cs$ is isomorphic to
$k(\xymatrix{\sigma (\ts)\ar@<0.5ex>[r]^{{\varepsilon}\,\,\,} &\sigma(\Sigma \ts)\ar@<0.5ex>[l]^{{\varepsilon'}\,\,\,}})\big/ ( \varepsilon'\varepsilon,\varepsilon\varepsilon' )$.
%\blue{I do not know how to describe the corresponding paths of $\cs$ in $\mcr$. }
\end{example}

\begin{example}[Type $A_2$]
  \label{example sl3}
Let $Q$ be the quiver $(\xymatrix{1\ar[r]&2})$. Let $C=\{\Sigma^j \ts_i\mid i=1,2,j\in\Z\}$ and $F=\Sigma^2$. Then $\Z Q_C/F$ is given by \eqref{figure 1}.

\begin{center}
\begin{equation}
\label{figure 1}
\begin{tikzpicture}[scale=1.2]
%\fill[opacity=0.5,fill=purple!60] 
(1.9,0) --(5,3.1) -- (5.1,3) -- (2,-0.1)--(1.9,0);
%\fill[opacity=0.5,fill=purple!20] (3.4,-0.2) --(7.2,3.6) -- (7.6,3.2) -- (3.8,-0.6)--(3.4,-0.2);
%\fill[opacity=0.5,fill=purple!60] 
(3.9,0) --(7,3.1) -- (7.1,3) -- (4,-0.1)--(3.9,0);

%\fill[opacity=0.5,fill=black!60]  
(1.9,0) --(3,1.1)--(4.1,0)--(4,-0.1)--(3,0.9)--(2,-0.1)--(1.9,0);
%\fill[opacity=0.5,fill=black!60] 
 (7,3.1) --(7.1,3)--(6,1.9)--(4.9,3)--(5,3.1)--(6,2.1)--(7,3.1);

%\fill[opacity=0.5,fill=purple!60] 
(4.9,3) --(5,3.1) -- (8.1,0) -- (8,-0.1)--(4.9,3);
%\fill[opacity=0.5,fill=purple!40] (6.4,3.2) --(6.8,3.6) -- (10.6,-0.2) -- (10.2,-0.6)--(6.4,3.2);
%\fill[opacity=0.5,fill=purple!60] 
(6.9,3) --(7,3.1) -- (10.1,0) -- (10,-0.1)--(6.9,3);

\node at (2,0) {\tiny$\sigma(\ts_1)$};
\node at (4,0) {\tiny$\sigma(\ts_2)$};
\node at (3,1) {\tiny$\ts_1$};
\node at (4,2) {\tiny$P_2$};
\node at (5,3) {\tiny$\sigma(\Sigma\ts_1)$};
\node at (5,1) {\tiny$\ts_2$};
\node at (6,2) {\tiny$\Sigma\ts_1$};
\node at (7,3) {\tiny$\sigma(\Sigma\ts_2)$};
\node at (7,1) {\tiny$\Sigma P_2$};
\node at (8,0) {\tiny$\sigma(\ts_1)$};
\node at (10,0) {\tiny$\sigma(\ts_2)$};
\node at (9,1) {\tiny$\ts_1$};
\node at (8,2) {\tiny$\Sigma\ts_2$};

\draw[->] (2.1,0.1) -- (2.9,0.9);
\draw[->] (3.1,0.9) -- (3.9,0.1);
\draw[->] (3.1,1.1) -- (3.9,1.9);
\draw[->] (4.1,1.9) -- (4.9,1.1);
\draw[->] (4.1,2.1) -- (4.9,2.9);
\draw[->] (5.1,2.9) -- (5.9,2.1);

\draw[->] (6.1,2.1) -- (6.9,2.9);
\draw[->] (5.1,1.1) -- (5.9,1.9);
\draw[->] (4.1,0.1) -- (4.9,0.9);

\draw[->] (8.1,0.1) -- (8.9,0.9);
\draw[->] (7.1,1.1) -- (7.9,1.9);
\draw[->] (8.1,1.9) -- (8.9,1.1);
\draw[->] (9.1,0.9) -- (9.9,0.1);
\draw[->] (7.1,2.9) -- (7.9,2.1);
\draw[->] (6.1,1.9) -- (6.9,1.1);
\draw[->] (7.1,0.9) -- (7.9,0.1);

%\draw[style=dashed] (3.3,1) -- (4.8,1);
%\draw[->] (3.3,1)--(3.2,1);
%\draw[style=dashed] (4.3,2) -- (5.8,2);
%\draw[->] (4.3,2)--(4.2,2);

%\draw[style=dashed] (5.3,1) -- (6,1);
%\draw[->] (5.3,1)--(5.2,1);
%\draw[style=dashed] (3.2,2) -- (3.9,2);
%\draw[->] (3.2,2)--(3.1,2);
\end{tikzpicture}
\end{equation}
\end{center}

The singular NKS category $\cs$ admits a description in terms of the bound quiver below:
\begin{center}\setlength{\unitlength}{0.7mm}
\vspace{-2cm}
\begin{equation*}
\begin{picture}(100,40)(0,20)
\put(40,8){\tiny $\sigma(\Sigma \ts_1)$}
\put(42,31){\tiny  $\sigma(\ts_1)$}
\put(72,8){\tiny  $\sigma(\Sigma \ts_2)$}
\put(72,31){\tiny  $\sigma(\ts_2)$}

\put(53,10){\vector(1,0){18.5}}
\put(53,32.5){\vector(1,0){18.5}}

\put(60,12.5){$_{\alpha'}$}
\put(60,35){$_\alpha$}
%
%\color{purple}
\put(50,13){\vector(0,1){17}}
\put(52,29.5){\vector(0,-1){17}}
\put(72,13){\vector(0,1){17}}
\put(74,29.5){\vector(0,-1){17}}

\put(45,20){\small $\varepsilon_1'$}
\put(53,20){\small $\varepsilon_1$}
\put(67,20){\small $\varepsilon_2'$}
\put(75,20){\small $\varepsilon_2$}
\end{picture}
\end{equation*}
\vspace{0.2cm}
\end{center}
%
%and $I^{\sharp}$ is generated by all possible quadratic relations
\begin{eqnarray*}
&&\varepsilon_1\varepsilon_1', \,\,\,\,\varepsilon_1'\varepsilon_1, \,\,\,\,\varepsilon_2'\varepsilon_2, \,\,\,\,\varepsilon_2\varepsilon_2',\,\,\,\, \alpha' \varepsilon_1 -\varepsilon_2 \alpha,\,\,\,\, \alpha \varepsilon_1'- \varepsilon_2' \alpha'.
\end{eqnarray*}
%\red{It is instrumental to observe how the black/purple paths in \eqref{figure 1} give rise to this quiver. }
\end{example}

\subsection{NKS quiver varieties}   \label{subsection:generalised quiver var}

In this subsection, we review the generalized quiver varieties introduced by Keller-Scherotzke \cite{KS16} and Scherotzke \cite{Sch18}; we shall call them Nakajima-Keller-Scherozke quiver varieties, or NKS (quiver) varieties for short, in this paper.

Let $\cs$ be a singular NKS category, and $\mcr$ its corresponding regular NKS category. An $\mcr$-module $M$ is \emph{stable} (resp. \emph{costable}) if the support of $\soc(M)$ (respectively,  $\Top(M)$) is contained in $\cs_0$. A module is \emph{bistable} if it is both stable and costable. Equivalently, $M$ does not contain any nonzero submodule supported only on non-frozen vertices (respectively, $M$ does not admit any nonzero quotient supported only on non-frozen vertices).

Let $\bv \in \N^{\mcr_0-\cs_0}$ and $\bw \in \N^{\cs_0}$ be dimension vectors (with finite supports). Denote by $\rep(\bv,\bw,\mcr)$ the variety of $\mcr$-modules of dimension vector $(\bv,\bw)$. Let $\e_x$ denote the characteristic function of $x\in\mcr_0$, which is also viewed as the unit vector supported at $x$. Let $G_\bv:=\prod_{x\in \mcr_0-\cs_0}GL({\bv(x)}, k)$.

\begin{definition}
 \label{def:NKS}
The {\em NKS quiver variety}, or simply {\em NKS variety}, $\cm(\bv,\bw)$ is the quotient $\cs t(\bv,\bw)/G_\bv$,  where $\cs t(\bv,\bw)$ is the subset of $\rep(\bv,\bw,\mcr)$ consisting of all stable $\mcr$-modules of dimension vector $(\bv,\bw)$. Define the affine variety
\begin{equation}
  \label{eq:M0}
\cm_0(\bv,\bw)=\cm_0(\bv,\bw, \mcr) :=\rep(\bv,\bw,\mcr)//G_\bv
\end{equation}
to be the categorical quotient, whose coordinate algebra is $k[\rep(\bv,\bw,\mcr)]^{G_\bv}$.
\end{definition}
According to \cite[Theorem 3.2]{Sch18} and its proof, $\cm(\bv,\bw)$ is a smooth quasi-projective variety, which is pure dimensional. However, we do not know if $\cm(\bv,\bw)$ is connected or not.

We define a partial order $\le$ on the set $\N^{\mcr_0}$ as follows:
 \begin{equation}
   \label{eq:leq}
 \text{  $\bv' \leq \bv
 \Leftrightarrow \bv'(x) \le \bv(x), \forall x \in \mcr_0$. Moreover, $\bv' < \bv \Leftrightarrow \bv'\leq \bv$ and $\bv' \neq \bv$.}
 \end{equation}
Regarding a dimension vector on $\mcr_0-\cs_0$ as a dimension vector on $\mcr_0$ (by extension of zero),
 we obtain by restriction a partial order $\le$ on the set $\N^{\mcr_0-\cs_0}$.

 For $\bv', \bv \in \N^{\mcr_0-\cs_0}$ with $\bv'\leq \bv$ and $\bw \in \N^{\cs_0}$,  there is an inclusion
\[
\rep(\bv',\bw,\mcr)\longrightarrow \rep(\bv,\bw,\mcr)
\]
by taking a direct sum with the semisimple module of dimension vector $\bv-\bv'$. This yields an inclusion
\[
\rep(\bv',\bw,\mcr)//G_{\bv'}\longrightarrow \rep(\bv,\bw,\mcr)//G_\bv.
\]
Define the affine variety
\[
\cm_0(\bw) =\cm_0(\bw, \mcr) :=\colim\limits_{\bv} \cm_0(\bv,\bw)
\]
to be the colimit of $\cm_0(\bv,\bw)$ along the inclusions. 
Then the projection map
\begin{equation}   \label{eq:pi}
\pi:\cm(\bv,\bw)\longrightarrow \cm_0(\bv,\bw),
\end{equation}
which sends the $G_\bv$-orbit of a stable $\mcr$-module $M$ to the unique closed $G_\bv$-orbit in the closure of $G_\bv M$, is proper; see \cite[Theorem 3.5]{Sch18}. %We also denote the fibre $\pi^{-1}(0)$ of the origin by $\mathscr{L}(\bv,\bw)$.

Denote by $\cm^{\text{reg}}(\bv,\bw)\subset \cm(\bv,\bw)$ the open subset consisting of the union of closed $G_\bv$-orbits of stable modules, and then
\[
\cm_0^{\text{reg}}(\bv,\bw):=\pi(\cm^{\text{reg}}(\bv,\bw))
\]
is an open subset of $\cm_0(\bv,\bw)$. \cite[Lemma 3.4]{Sch18} shows that $\cm(\bv,\bw)$ vanishes on all but finitely many dimension vectors $\bv$. Then $\pi$ induces a stratification
\[
\cm_0(\bw)=\bigsqcup_\bv\cm_0^{\text{reg}}(\bv,\bw)
\]
into finitely many smooth locally closed strata $\cm_0^{\text{reg}}(\bv,\bw)$; see \cite[Theorem 3.5]{Sch18} and its proof.

Let $\res:\mod(\mcr)\rightarrow\mod(\cs)$ be the restriction functor. Then $\res$ induces a morphism of varieties
\[
\res:\cm_0(\bw)\longrightarrow \rep(\bw,\cs),\qquad
\res\circ\pi:\cm(\bv,\bw)\longrightarrow \rep(\bw,\cs).
\]
%Furthermore, $\res:\cm_0(\bw)\rightarrow\rep(\bw,\cs)$ induces a bijection on the closed points; see  \cite[Theorem 3.9]{Sch18}.

%\begin{lemma} \cite[Theorem 3.11]{Sch18}
%  \label{lem:isomorhism for M0 and repS}
%Let $Q$ be a Dynkin quiver. Assume that $F$ is an automorphism such that $\Hom_{\cd_Q}(P_i,F^n P_i)=0$ for all indecomposable projective $kQ$-module $P_i$ and all $n\in\Z_{\ge 1}$, then the natural map
%\[
%%\cm_0(\bw)\longrightarrow \rep(\bw,\cs)
%\]
%induces an equivalence of schemes.
%\end{lemma}

%\begin{remark}
%Let $\bv:\mcr_0-\cs_0\rightarrow \N$ and $\bw:\cs_0\rightarrow\N$ be two dimensional vectors. Then $\bw$ can be viewed as a dimension vector $\bw:\cs^f\rightarrow \N$. Then we can consider the %generalised quiver variety $\cm(\bv,\bw)$ for $\mcr^f$. It is isomorphic to $\cm(\bv,\bw)$ for $\mcr$.
%\end{remark}

%\begin{lemma}[\cite{Sch18}]\label{lem:bistable and stable}
%(i) The closed $G_\bv$-orbits in $\rep(\bv,\bw,\mcr)$ are represented by $L\oplus N\in\rep(\bv,\bw,\mcr)$ where $L$ is a bistable module and $N$ is a semi-simple module such that $\res N$ vanishes.

%(ii) A stable $\mcr$-module $M$ belongs to $\cm^{\text{reg}}(\bv,\bw)$ if and only if it is bistable.
%\end{lemma}

%
%
\subsection{The stratification functor}
\label{subsec: strat}

As the restriction functor $\res:\mod(\mcr) \rightarrow \mod(\cs)$ is a localization functor, it admits a left and a right adjoint functor, which are denoted by $K_L$ and $K_R$ respectively. The intermediate extension $K_{LR}:\mod(\cs) \rightarrow \mod(\mcr)$ is the image of the canonical map $K_L\rightarrow K_R$. Note that in this case $K_L$ is right exact and $K_R$ is left exact, $K_{LR}$ preserves monomorphisms and epimorphisms (however, it may not be exact in general).

Let $e$ be the idempotent in $\mcr$ associated to the objects of $\cs$. Then $e\mcr e=\cs$ and $K_L=\mcr e\otimes_\cs-$ and $K_R=\Hom_\cs(e\mcr,-)$.

For any $\cs$-module $M$, the modules $CK(M)$ and $KK(M)$ are defined to be
\begin{itemize}
\item $KK(M)=\ker (K_L(M)\rightarrow K_{LR}(M))$;
\item $CK(M)=\Coker(K_{LR}(M)\rightarrow K_{R}(M))$.
\end{itemize}
Both $CK$ and $KK$ are supported only in $\mcr_0-\cs_0$. As $ \cp=\mcr/(\cs )$, we can view $CK(M)$ and $KK(M)$ as $\cp$-modules. %$CK$ is called the {\em stratification functor} thanks to its  property in the following lemma.

\begin{lemma} 
[\text{\cite[Lemma 2.12, Lemma 3.8]{Sch18}}] 
\label{lem:bistable}
An $\mcr$-module $M$ is bistable if and only if $K_{LR}(\res M)\cong M$. Furthermore, we have $N\in\cm_0^{\mathrm{reg}}(\bv,\bw)$ if and only if $\underline{\dim}K_{LR} (N)=(\bv,\bw)$.
\end{lemma}

%\begin{lemma}[\cite{Sch18}]
%Every point $L\in\cm_0(\bv,\bw)$ corresponds uniquely to a pair $(L_1,L_2)$ where $L_1=\res L\in \rep(\bw,\cs)$ and $L_2$ a representative of the isomorphism class of a semi-simple $\cp$-module. With this identification the map $\pi:\cm(\bv,\bw)\rightarrow \cm_0(\bw)$ is given by $G_\bvN\mapsto(\res N,N')$ where $N'$ is the semi-simple module with the same composition series than
%$\Coker(K_{LR}(\res N)\mapsto N)$.
%\end{lemma}

\begin{example}[Type $A_1$]
Keep the notation as in Example \ref{example sl2}. An $\mcr$-module is of the form
\[
\xymatrix{V(S)\ar[r]^{\beta_1} & W(\sigma(S)) \ar[d]_{\alpha_2} \\
W(\sigma(\Sigma S)) \ar[u]^{\alpha_1}& V(\Sigma S)\ar[l]^{\beta_2} }
\]
%
%$$W(\sigma S)\xleftarrow{\beta_1} V(S) \xleftarrow{\alpha_1} W(\sigma(\Sigma S)) \xleftarrow{\beta_2} V(\Sigma S)\xleftarrow{\alpha_2} W(\sigma S).$$
\end{example}

\begin{example}[Type $A_2$]
Keep the notation as in Example \ref{example sl3}. An $\mcr$-module is of the form
\begin{center}\setlength{\unitlength}{0.6mm}
\begin{equation}
\label{figure 2}
%\begin{figure}
 \begin{picture}(120,65)(0,0)
 \put(-2,20){\tiny$V(\ts_1)$}
 \put(18,40){\tiny $V(P_2)$}
 \put(38,20){\tiny $V(\ts_2)$}
  \put(36,60){\tiny $W(\sigma(\Sigma \ts_1))$}

 \put(16,0){\tiny $W(\sigma(\ts_2))$}
  \put(58,40){\tiny $V(\Sigma \ts_1)$}
 \put(76,60){\tiny $W(\sigma(\Sigma \ts_2))$}
  \put(78,20){\tiny $V(\Sigma P_2)$}
 \put(96,0){\tiny $W(\sigma(\ts_1))$}
 \put(98,40){\tiny $V(\Sigma \ts_2)$}
 \put(118,20){\tiny $V(\ts_1)$}

\put(21,39){\vector(-1,-1){15}}
\put(21,4){\vector(-1,1){15}}

 \put(61,39){\vector(-1,-1){15}}
\put(101,4){\vector(-1,1){15}}
  \put(101,39){\vector(-1,-1){15}}
  \put(81,59){\vector(-1,-1){15}}
  \put(41,59){\vector(-1,-1){15}}
  \put(121,19){\vector(-1,-1){15}}
  \put(41,19){\vector(-1,-1){15}}

 \put(61,44){\vector(-1,1){15}}
 \put(41,24){\vector(-1,1){15}}
 \put(81,24){\vector(-1,1){15}}
 \put(121,24){\vector(-1,1){15}}
 \put(101,44){\vector(-1,1){15}}

\put(7,17){\begin{picture}(0,0)
\dashline{3}(6,4)(31,4)
\put(6,4){\vector(-1,0){1}}
\end{picture}
}

\put(47,17){\begin{picture}(0,0)
\dashline{3}(6,4)(31,4)
\put(6,4){\vector(-1,0){1}}
\end{picture}
}
\put(87,17){\begin{picture}(0,0)
\dashline{3}(10,4)(31,4)
\put(10,4){\vector(-1,0){1}}
\end{picture}
}

\put(67,37){\begin{picture}(0,0)
\dashline{3}(9,4)(31,4)
\put(9,4){\vector(-1,0){1}}
\end{picture}
}

\put(27,37){\begin{picture}(0,0)
\dashline{3}(6,4)(31,4)
\put(6,4){\vector(-1,0){1}}
\end{picture}
}
\end{picture}
%\caption{Representations of the regular NKS category $\mcr$ of type $A_2$}\label{figure 2}
%\end{figure}
\end{equation}
\vspace{3mm}
\end{center}
%{\small\[\xymatrix{  & V(P_2)\ar[dl]& W(\sigma(\Sigma \ts_1)) \ar[l]& V(\Sigma \ts_1)\ar[l]\ar[dl]& W(\sigma(\Sigma \ts_2)) \ar[l]&  V(\Sigma \ts_2)\ar[l]\ar[dl]& \\
%V(\ts_1)  & W(\sigma(\ts_2)) \ar[l] & V(\ts_2)\ar[ul]\ar[l]  &  & V(P_2) \ar[ul] & W(\sigma(\ts_1)) \ar[l]& V(\ts_1) \ar[ul]\ar[l]}\]
%}
\end{example}

\begin{example}
	\label{ex:cyclic}
Let $Q$ be a Dynkin quiver. In Nakajima's original construction \cite{Na01,Na04}, the graded/cyclic quiver varieties are defined as fixed point sets of the (ungraded) quiver varieties with respect to a $\C^*$-action. These quiver varieties are special cases of NKS quiver varieties as explained in \cite[Remark 3.1]{Sch18}.
 Let $C$ be the set of all vertices of $\Z Q$. 

(a) For $F=\tau$, Nakajima's original quiver varieties are obtained as the NKS quiver varieties $\cm_0(\bw),\cm(\bv,\bw)$.

(b) For $F=\tau^n$, the ($n$-)cyclic quiver varieties are obtained as  the NKS quiver varieties $\cm_0(\bw),\cm(\bv,\bw)$.
\end{example}

Given $\bv \in \N^{\mcr_0-\cs_0}$,  we define (cf. \cite{Sch18})
%\red{[this explicit definition appeared right before \cite[Proposition 4.6]{Sch18}] }
\begin{align}
\label{def:Cq}
\begin{split}
{\mathcal C}_q \bv:  \mcr_0-\cs_0 & \longrightarrow\Z,
\\
({\mathcal C}_q\bv)(x)& =\bv(x)+\bv(\tau x) -\sum_{y\rightarrow x}\bv(y),
\quad \text{ for }x \in\mcr_0-\cs_0,
\end{split}
\end{align}
where the sum runs over all arrows $y\rightarrow x$ of $\mcr$ with $y\in\mcr_0-\cs_0$.
Given $\bw \in \N^{\cs_0}$, define a dimension vector
\[
\sigma^*\bw:\mcr_0-\cs_0\longrightarrow\N,
\qquad
x \mapsto
\begin{cases}
\bw(\sigma x), & \text{ if } x\in C,
\\
0, & \text{otherwise}.
\end{cases}
\]
Given $\bv \in \N^{\mcr_0-\cs_0}$, define the dimension vector
\[
\tau^*\bv: \mcr_0-\cs_0 \longrightarrow \N,
\qquad
x \mapsto \bv(\tau x).
\]

By \cite[Proposition 4.6]{Sch18}, if the fibre of $\pi:\cm(\bv,\bw)\rightarrow \cm_0(\bw)$ is non-empty, then $\sigma^*\bw-{\mathcal C}_q\bv \ge \bf 0$. So we obtain the following more precise form of a stratification of $\cm_0(\bw)$
\begin{equation}   \label{eqn:stratification}
\cm_0(\bw)=\bigsqcup_{\bv:\sigma^*\bw-{\mathcal C}_q\bv\geq0} \cm_0^{\text{reg}}(\bv,\bw).
\end{equation}

\begin{definition}   \label{def:pair}
A pair $(\bv,\bw)$ of dimension vectors $\bv \in \N^{\mcr_0-\cs_0}$ and $\bw \in \N^{\cs_0}$ is called $l$\emph{-dominant} if $\sigma^*\bw-{\mathcal C}_q\bv\geq0$, and $(\bv, \bw)$ is called {\em strongly $l$\emph-dominant} if it is $l$\emph-dominant and  $\cm_0^{\mathrm{reg}}(\bv,\bw) \neq \emptyset.$
\end{definition}

Recall $\sigma:C\rightarrow \cs_0$ is a bijection. Denote by $\sigma^{-1}:\cs_0\rightarrow C$ its inverse. Then $(\bv,\bw)$ is an $l$-dominant pair if and only if $\bw-\mathcal{C}_q\bv\sigma^{-1}\geq0$.
 
%
%

%%%%%%%%%
%%%%%%%%%
\section{NKS categories for $\imath$quivers}
   \label{sec:iNKS}

In this section, we formulate and study the $\imath$NKS categories and varieties associated to Dynkin $\imath$quivers and relate them to the $\imath$quiver algebras introduced in \cite{LW19}. We also recall the quantum Grothendieck rings for $\imath$NKS quiver varieties, following \cite{Na04, VV} (also cf. \cite{Qin, SS16}).

\subsection{The $\imath$quivers and doubles}
\label{subsec:iquiver}

We briefly review the notions of $\imath$quivers and $\imath$quiver algebras introduced by the authors in \cite{LW19}. (The notation $\uptau$ for quiver involution is changed to $\btau$ here, as $\uptau$ could be confused with the notation $\tau$ for the AR translation which is used often in this paper.)

%Let $\K$ be a field.
Let $Q=(Q_0=\I,Q_1)$ be an acyclic quiver (i.e., a quiver without oriented cycle). An {\em involution} of $Q$ is defined to be an automorphism $\btau$ of the quiver $Q$ such that $\btau^2=\Id$. In particular, we allow the {\em trivial} involution $\Id:Q\rightarrow Q$. An involution $\btau$ of $Q$ induces an involution of the path algebra $\K Q$, again denoted by $\btau$.
A quiver together with a specified involution $\btau$, $(Q, \btau)$, will be called an {\em $\imath$quiver}.

Let $R_1$ denote the truncated polynomial algebra $\K[\varepsilon]/(\varepsilon^2)$.
Let $R_2$ denote the radical square zero of the path algebra of $\xymatrix{1 \ar[r]<0.5ex>^{{\varepsilon}} & 1' \ar[l]<0.5ex>^{{\varepsilon'}}}$, i.e., $\varepsilon' \varepsilon =0 =\varepsilon\varepsilon '$. Define a $\K$-algebra
\begin{equation}
  \label{eq:La}
\Lambda=\K Q\otimes_\K R_2.
\end{equation}

Associated to the quiver $Q$, the {\em double framed quiver} $Q^\sharp$
is the quiver such that
\begin{itemize}
\item the vertex set of $Q^{\sharp}$ consists of 2 copies of the vertex set $Q_0$, $\{i,i'\mid i\in Q_0\}$;
\item the arrow set of $Q^{\sharp}$ is
\[
\{\alpha: i\rightarrow j,\alpha': i'\rightarrow j'\mid(\alpha:i\rightarrow j)\in Q_1\}
\cup
\{ {{\varepsilon_i: i\rightarrow i' ,\varepsilon'_i: i'\rightarrow i} } \mid i\in Q_0 \}.
\]
\end{itemize}
Note $Q^\sharp$ contains 2 isomorphic copies of subquiver $Q\cong Q'$, and it admits a natural involution called $\swa$, which switches $Q$ and $Q'$.
The involution $\btau$ of a quiver $Q$ induces an involution ${\btau}^{\sharp}$ of $Q^{\sharp}$ which is basically the composition of $\swa$ and $\btau$ (on the two copies of subquivers $Q$ and $Q'$ of $Q^\sharp$).
%defined by
%\begin{itemize}
%\item ${\btau}^{\sharp}(i)=(\btau i)'$, ${\btau}^{\sharp}(i') =\btau i$ for any $i\in Q_0$;
%\item ${\btau}^{\sharp}(\varepsilon_i)= \varepsilon_{\btau i}'$, ${\btau}^{\sharp}(\varepsilon_i')= \varepsilon_{\btau i}$ for any $i\in Q_0$;
%\item ${\btau}^{\sharp}(\alpha)= (\btau\alpha)'$, ${\btau}^{\sharp}(\alpha')=\btau\alpha$ for any $\alpha\in Q_1$.
%\end{itemize}
The algebra $\La$ can be realized in terms of the quiver $Q^{\sharp}$ and a certain admissible ideal $I^{\sharp}$
so that $\Lambda\cong \K Q^{\sharp} \big/ I^{\sharp}$; see \cite[\S2.2]{LW19}.

%The algebra $\Lambda$ can be described in terms of a quiver with relations. Let $I^{\sharp}$ be the admissible ideal of $\K Q^{\sharp}$ generated by
%\begin{itemize}
%\item
%(Nilpotent relations) $\varepsilon_i \varepsilon_i'$, $\varepsilon_i'\varepsilon_i$ for any $i\in Q_0$;
%\item
%(Commutative relations) $\varepsilon_j' \alpha' -\alpha\varepsilon_i'$, $\varepsilon_j \alpha -\alpha'\varepsilon_i$ for any $(\alpha:i\rightarrow j)\in Q_1$.
%\end{itemize}
%
%Let $Q$ (respectively, $Q'$) be the full subquiver of $Q^{\sharp}$ formed by all vertices $i$ (respectively, $i'$) for $i\in Q_0$. Then $Q\sqcup Q'$ is a subquiver of $Q^{\sharp}$.

By \cite[Lemma~2.4]{LW19}, ${\btau}^{\sharp}$ on $Q^\sharp$ preserves $I^\sharp$ and hence induces an involution ${\btau}^{\sharp}$ on the algebra $\Lambda$. By \cite[Definition 2.5]{LW19}, the {\rm $\imath$quiver algebra} of $(Q, \btau)$ is the fixed point subalgebra of $\Lambda$ under ${\btau}^{\sharp}$,
\begin{equation}
   \label{eq:iLa}
\iLa
= \{x\in \Lambda\mid {\btau}^{\sharp}(x) =x\}.
\end{equation}
By \cite[Proposition 2.6]{LW19}, the algebra $\iLa$ can be described in terms of a quiver $\ov Q$ and its admissible ideal $\ov{I}$ as
\begin{align}
\label{eqn:i-quiver}
\iLa \cong \K \ov{Q} \big/ \ov{I},
\end{align}
where

$\triangleright$ $\ov{Q}$ is constructed from $Q$ by adding a loop $\varepsilon_i$ at the vertex $i\in Q_0$ if $\btau i=i$, and adding an arrow $\varepsilon_i: i\rightarrow \btau i$ for each $i\in Q_0$ if $\btau i\neq i$;

$\triangleright$ $\ov{I}$ is generated by
\begin{itemize}
\item[(1)] (Nilpotent relations) $\varepsilon_{i}\varepsilon_{\btau i}$ for any $i\in\I$;
\item[(2)] (Commutation relations) $\varepsilon_i\alpha-\btau(\alpha)\varepsilon_j$ for any arrow $\alpha:j\rightarrow i$ in $Q_1$.
\end{itemize}
 Moreover, it follows by \cite[Proposition 3.5]{LW19} that $\Lambda^{\imath}$ is a $1$-Gorenstein algebra.

By \cite[Corollary 2.12]{LW19}, $\K Q$ is naturally a subalgebra and also a quotient algebra of $\Lambda^\imath$.
Viewing $\K Q$ as a subalgebra of $\Lambda^{\imath}$, we have a restriction functor
\[
\res: \mod (\Lambda^{\imath})\longrightarrow \mod (\K Q).
\]
Viewing $\K Q$ as a quotient algebra of $\Lambda^{\imath}$, we obtain a pullback functor
\begin{equation}\label{eqn:rigt adjoint}
\iota:\mod(\K Q)\longrightarrow\mod(\Lambda^{\imath}).
\end{equation}
Hence a simple module $\ts_i$ of $k Q$ associated to $i \in Q_0$ is naturally a simple $\iLa$-module.

For each $i\in Q_0$, define a $k$-algebra (which can be viewed as a subalgebra of $\iLa$)
\begin{align}\label{dfn:Hi}
\BH _i:=\left\{ \begin{array}{ll}  \K[\varepsilon_i]/(\varepsilon_i^2)
& \text{ if }\btau i=i, \; \begin{picture}(50,13)(0,0)
\put(0,-2){\tiny $i$}
%
%\color{purple}
\put(0,9){\small $\varepsilon_i$}
\qbezier(-1,1)(-3,3)(-2,5.5)
\qbezier(-2,5.5)(1,9)(4,5.5)
\qbezier(4,5.5)(5,3)(3,1)
\put(3.1,1.4){\vector(-1,-2){0.3}}
\end{picture}
 \\
\K(\xymatrix{
i\ar@<0.5ex>[r]^{{\varepsilon_i}} & \btau i\ar@<0.5ex>[l]^{{\varepsilon_{\btau i}}}
})/( \varepsilon_i\varepsilon_{\btau i},\varepsilon_{\btau i}\varepsilon_i)
&\text{ if } \btau i \neq i .\end{array}\right.
\end{align}
Note that $\BH _i=\BH _{\btau i}$ for any $i\in Q_0$. %Recall $\ci$ from \eqref{eq:ci}.
Choose one representative for each $\btau$-orbit on $\I$, and let
\begin{align}   \label{eq:ci}
\ci = \{ \text{the chosen representatives of $\btau$-orbits in $\I$} \}.
\end{align}

Define the following subalgebra of $\Lambda^{\imath}$:
\begin{equation}  \label{eq:H}
\BH =\bigoplus_{i\in \ci }\BH _i.
\end{equation}
Note that $\BH $ is a radical square zero selfinjective algebra. Denote by
\begin{align}
\res_\BH :\mod(\iLa)\longrightarrow \mod(\BH )
\end{align}
the natural restriction functor.
On the other hand, as $\BH $ is a quotient algebra of $\iLa$ (cf. \cite[proof of Proposition~ 2.15]{LW19}), every $\BH $-module can be viewed as a $\iLa$-module.

Recall the algebra $\BH _i$ for $i \in \ci$ from \eqref{dfn:Hi}. For $i\in Q_0 =\I$, define the indecomposable (right) module over $\BH _i$ (if $i\in \ci$) or over $\BH_{\btau i}$ (if $i\not \in \ci$)
\begin{align}
  \label{eq:E}
\E_i =\begin{cases}
k[{\varepsilon_i}]/(\varepsilon_i^2), & \text{ if }\btau i=i;
\\
\xymatrix{
\K\ar@<-0.5ex>[r]_0 & \K\ar@<-0.5ex>[l]_1
}
\text{ on the quiver }
\xymatrix{
i\ar@<0.5ex>[r]^{{\varepsilon_i}} & \btau i\ar@<0.5ex>[l]^{{\varepsilon_{\btau i}}}
}, & \text{ if } \btau i\neq i.
\end{cases}
\end{align}
Then $\E_i$, for $i\in Q_0$, can be viewed as a $\iLa$-module and is called a {\em generalized simple} $\iLa$-module.

\subsection{NKS categories for quantum groups}
 \label{subsec:NKS}

First we review briefly and set up notations for Gorenstein algebras and Gorenstein projective modules. Let $A$ be a finite-dimensional $\K$-algebra.
A complex
\[
P^\bullet:\cdots\longrightarrow P^{-1}\longrightarrow P^0\stackrel{d^0}{\longrightarrow} P^1\longrightarrow \cdots
\]
of finitely generated projective $A$-modules is said to be \emph{totally acyclic} provided it is acyclic and the Hom complex $\Hom_A(P^\bullet,A)$ is also acyclic. % \cite{AM}.
An $A$-module $M$ is said to be (finitely generated) \emph{Gorenstein projective} provided that there is a totally acyclic complex $P^\bullet$ of projective $A$-modules such that $M\cong \Ker d^0$ \cite{EJ}. We denote by $\Gproj(A)$ the full subcategory of $\mod(A)$ consisting of Gorenstein projective modules.

%A $k$-algebra $A$ is called a {\em Gorenstein algebra} \cite{EJ, Ha3}  if $\ind\,_A A<\infty$ and $\ind A_A <\infty$. It is known that a $k$-algebra $A$ is Gorenstein if and only if $\ind {}_AA<\infty$ and $\pd D(A_A)<\infty$. For a Gorenstein algebra $A$, by Zaks' lemma we have $\ind {}_AA=\ind A_A$, and the common value is denoted by $\Gd A$. If $\Gd A\leq d$, we say that $A$ is a \emph{$d$-Gorenstein} algebra.

Let $A$ be a finite-dimensional algebra. The {\em singularity category} of $A$ is defined (cf. \cite{Ha3}) to be the Verdier localization $D_{sg}(A):=D^b(A)/K^b(\proj (A)).$ Buchweitz-Happel's Theorem states that $\Gproj(A)$ is a Frobenius category with projective modules as projective-injective objects. Moreover, $D_{sg}(A) \simeq \underline{\Gproj}(A)$ if $A$ is a Gorenstein algebra.

\begin{definition} [NKS regular/singular categories] 
 \label{def:RS for QG}
Let $Q$ be a Dynkin quiver. Denote by $\mcr$ and $\cs$ the regular and singular NKS categories associated to the admissible pair $(\Sigma^2, C)$, where
\[
C=\{\text{the vertices labelled by }\Sigma^j \ts_i, \text{ for all } i\in Q_0 \text{ and } j\in\Z \}.
\] 
\end{definition}
{\bf In the remainder of this paper, we shall reserve the notations $\mcr$ and $\cs$ for the regular and singular NKS categories in Definition~\ref{def:RS for QG}.} 

\begin{proposition}[cf. \cite{SS16}]
\label{prop:2-complexes}
Let $Q$ be a Dynkin quiver. Let $\mcr$ and $\cs$ be the regular and singular NKS categories as in Definition~\ref{def:RS for QG}. Then $\Gproj(\Lambda)$ is equivalent to $\proj(\mcr)$, $\proj(\Lambda)$ is Morita equivalent to $\cs$, and we have $\Gproj(\Lambda)\simeq \Gproj(\cs)\simeq \proj(\mcr)$
as Frobenius categories.
\end{proposition}

\begin{proof}
Recall $\cp=\mcr/ ( \cs )$. Then $\proj(\cp)\simeq \cd_Q/\Sigma^2$ by definition.
From \cite[Theorems 3.3, 3.5]{Sch17}, $\proj(\mcr)$ is a Frobenius model for $\proj(\cp)$, i.e., $\proj(\mcr)$ is a Frobenius category with the objects in $\proj(\cs)$ (viewed as a subcategory of $\proj(\mcr)$) as its projective-injective objects, and $\underline{\proj}(\mcr)\simeq \proj(\cp)$.

Let $\cc_{\Z/2}(\proj(kQ))$ be the category of $\Z/2$-graded complexes over $\proj(kQ)$, see \cite[Section 3.1]{Br13}. Then $\Gproj(\Lambda)\simeq \cc_{\Z/2}(\proj(kQ))$ (see, e.g.,  \cite[(3.2)]{LW19}). From \cite[Theorem 3.1]{SS16}, we have $\cc_{\Z/2}(\proj(kQ))\simeq \proj(\mcr)$, and hence $\Gproj(\Lambda)\simeq \proj(\mcr)$.

By comparing the projective-injective objects in $\Gproj(\Lambda)\simeq \proj(\mcr)$, we have $\proj(\Lambda)\simeq \proj(\cs)$. The last assertion follows.
\end{proof}

In fact, by viewing $\cs$ as an algebra, we can identify $\cs$ with the algebra $\Lambda$ defined in \eqref{eq:La}, by noting $\Ind\proj(\Lambda)\simeq \cs$.
Also let $\cp=\mcr/( \cs )$ be the corresponding preprojective NKS category.

\begin{remark}
A geometric realization of the Drinfeld double quantum group $\tU$ via a variant of the NKS categories  in Definition~\ref{def:RS for QG} was given in \cite{Qin} (see Theorem~\ref{thm:iso HL-Q} below); also see \cite{SS16}. The original regular NKS category used in \cite{Qin} differs from the one in Definition~\ref{def:RS for QG} in that the admissible configuration $(F,C)$ consists of $F=\Sigma^2$ and $C =\{\text{all the vertices of } \Ind \cd_Q\}$. 
\end{remark}

\subsection{$\imath$NKS categories}

For an $\imath$quiver $(Q,\btau)$, let $\widehat{\btau}$ be the triangulated equivalence of $\cd_Q$ induced by the automorphism $\btau$ of $Q$. Recall $\iLa$ from \eqref{eq:iLa}. Note that $\cd_Q/\Sigma\widehat{\btau}$ and $\cd_Q/\Sigma^2$ are triangulated orbit category \`a la Keller \cite{Ke05}; see e.g., \cite[Lemma 3.7]{LW19}.

\begin{lemma}  [\text{\cite[Theorem 3.18]{LW19}}]
We have $\underline{\Gproj}(\Lambda^\imath)\simeq D_{sg}(\Lambda^\imath)\simeq \cd_Q/\Sigma\widehat{\btau}$.
\end{lemma}

\begin{definition} [$\imath$NKS regular/singular categories] 
 \label{def:RS for iQG}
Let $(Q,\btau)$ be a Dynkin $\imath$quiver. Denote by $\mcr^\imath$ and $\cs^\imath$ the regular and singular NKS categories associated to the admissible pair $(F^\imath, C)$, where
\[
F^\imath=\Sigma \widehat{\btau}, 
\quad \text{and }
C=\{\text{the vertices labelled by }\Sigma^j \ts_i, \text{ for all } i\in Q_0 \text{ and } j\in\Z \}.
\] 
\end{definition}
{\bf From now on, we reserve the notations $\mcr^\imath$ and $\cs^\imath$ for the regular and singular NKS categories in Definition~\ref{def:RS for iQG}.}

\begin{theorem}
\label{thm:iNKS}
Let $(Q,\btau)$ be a Dynkin $\imath$quiver. Let $\mcr^\imath$ and $\cs^\imath$ be the regular and singular NKS categories as in Definition~\ref{def:RS for iQG}. Then $\Gproj(\Lambda^\imath)$ is equivalent to $\proj(\mcr^\imath)$, $\proj(\Lambda^\imath)$ is Morita equivalent to $\cs^\imath$, and there are equivalences of Frobenius categories
$
\Gproj(\Lambda^\imath)\simeq \Gproj(\cs^\imath)\simeq \proj(\mcr^\imath).
$
\end{theorem}

\begin{proof}
As $\Lambda^\imath$ is the $\btau^\sharp$-invariant subalgebra of $\Lambda$, denote by $\cv_*:\mod(\Lambda)\rightarrow \mod (\Lambda^\imath)$ the pushdown functor, see \cite[Remark 2.11]{LW19}. Note that $\Lambda$ and $\Lambda^{\imath}$ are $1$-Gorenstein algebras by \cite[Proposition 3.5]{LW19}. Then $\cv_*$ induces a Galois covering $\cv_* :\Gproj(\Lambda)\longrightarrow \Gproj (\Lambda^\imath)$, and it preserves projective modules and almost split sequences.

%Denote by $\underline{\cv_*}:\underline{\Gproj}(\Lambda)\longrightarrow \underline{\Gproj} (\Lambda^\imath)$ the triangulated functor induced by $\cv_*$. Then $\underline{\cv_*}$ induces a triangulated functor $D_{sg}(\Lambda)\rightarrow D_{sg}(\Lambda^{\imath})$. On the other hand, $D_{sg}(\Lambda^{\imath}) \simeq \cd_Q/\Sigma \Psi_\vartheta$ and $D_{sg}(\Lambda) \simeq D_{sg}(\mod^{\Z/2Z}\Lambda^{\imath})\simeq \cd_Q/\Sigma^2$, we obtain the following commutative diagram
%\[\xymatrix{ \underline{\Gproj}\Lambda \ar[r]^{\sim} \ar[d]^{\underline{\cv_*}} & D_{sg}(\Lambda) \ar[r]^{\sim} \ar[d]^{\underline{\cv_*}} & \cd_Q/\Sigma^2\ar[d]^{\cv} \\
%\underline{\Gproj}\Lambda^{\imath} \ar[r]^{\sim}  & D_{sg}(\Lambda^{\imath}) \ar[r]^{\sim} & \cd_Q/\Sigma \Psi_\vartheta,}\]
%where $\cv$ is the canonical projection. Then $\underline{\cv_*}$ is dense, so is $\cv_*$.

Note that $\proj(\mcr)\simeq\Gproj(\Lambda)\simeq\cc_{\Z/2}(\proj(kQ))$ and $\underline{\Gproj}(\Lambda)\simeq \cd_Q/\Sigma^2$. It follows from \cite[Theorem 3.1]{SS16} and its proof that $\proj(\mcr)\simeq\Gproj(\Lambda)$ is a standard Frobenius model for $\cd_Q/\Sigma^2$.

Since $\cv_*$ preserves almost split sequences and projective modules, similar to the proof of \cite[Theorem 3.1]{SS16}, one can prove that $\Gproj(\Lambda^{\imath})$ is a standard Frobenius model for $\cd_Q/\Sigma \widehat{\btau}$. It follows from \cite[Theorem 3.7]{Sch17} that there exists an admissible configuration ($C',\Sigma\widehat{\btau})$ such that its corresponding regular NKS category $\mcr'$ satisfies $\proj(\mcr')\simeq \Gproj(\Lambda^\imath)$. So there exists a Galois covering functor $\cv'_*$ such that
the following diagram commutes:
\[
\xymatrix{\Gproj(\Lambda) \ar[rr]^{\sim} \ar[d]^{\cv_*} && \proj(\mcr)\ar[d]^{\cv'_*} \\
\Gproj(\Lambda^\imath) \ar[rr]^{\sim}&& \proj(\mcr') }
\]
The categories in the above commutative diagram are standard Frobenius models, and every projective-injective indecomposable object appears in exactly
one mesh of their AR-quivers. Then $C'=C$. So $\Gproj(\Lambda^\imath)\simeq \mcr^\imath$.

The proof of the last assertion is the same as for Proposition \ref{prop:2-complexes}, and hence omitted here.
\end{proof}

It follows from $\Gproj(\Lambda^\imath)\simeq \Gproj(\cs^\imath)$  (see Theorem \ref{thm:iNKS}) that $\proj(\Lambda^\imath)\simeq \proj(\cs^\imath)$. So we can identify $\cs^\imath$ with the algebra $\Lambda^\imath$ defined in \eqref{eq:iLa} by viewing $\cs^\imath$ as an algebra. Also let $\cp^\imath=\mcr^\imath/( \cs^\imath )$ be the preprojective NKS category.  Often we shall refer to these as $\imath$NKS categories (to distinguish from the NKS categories $\mcr, \cs, \cp$ given in \S\ref{subsec:NKS}).

\begin{remark}
\begin{itemize}
\item[(i)]
The equivalence of $\proj(\Lambda)$ and $\cs$ sends $e_i\Lambda$ to $\sigma\ts_i$ for each $i\in Q_0$. Unless otherwise specified, we do not distinguish $\mod(\Lambda)$ from $\mod(\cs)$ below.

\item[(ii)]
Similar remarks apply to the equivalence between $\proj(\Lambda^\imath)$ and $\cs^\imath$, and we shall identify $\mod(\Lambda^\imath)$ and $\mod(\cs^\imath)$.
\end{itemize}
\end{remark}

%From the above, the restriction functor $\Phi$ induces a functor $\Phi:\mod(\cs^\imath)\rightarrow \mod (kQ)$.
%A category $\cc$ is called $d$-Gorenstein if $\mod(\cc)$ is a $d$-Gorenstein category.
As a corollary, we have the following.

\begin{corollary}
The algebras $\cs$ and $\cs^\imath$ are $1$-Gorenstein.
\end{corollary}

\begin{proof}
Both $\Lambda$ and $\Lambda^\imath$ are $1$-Gorenstein algebras by \cite[Proposition 3.5]{LW19}. Then the desired result follows from Theorem \ref{thm:iNKS}.
\end{proof}

\begin{example}[Split type $A_1$]
  \label{example split sl2}
Let $Q$ be the quiver with a single vertex and no arrow. Let $\ts$ denote the simple module of $Q$, $C=\{\Sigma^j \ts \mid j\in\Z\}$ and $F^\imath=\Sigma$. Then $\Z Q_C/F^\imath$ is given by $\xymatrix{\ts \ar@<0.5ex>[r]^{\alpha} & \sigma(\ts) \ar@<0.5ex>[l]^{\beta}}$ with $\beta\alpha=0$.
% \\
%\sigma(\Sigma S) \ar[r]^{\beta_2}& \Sigma S\ar[u]^{\alpha_2}. }\]
%
%$$\sigma S\xrightarrow{\beta_1} S \xrightarrow{\alpha_1} \sigma(\Sigma S) \xrightarrow{\beta_2} \Sigma S\xrightarrow{\alpha_2} \sigma S.$$
The regular NKS category $\mcr$ is isomorphic to $k(\Z Q_C/F^\imath)/(\beta\alpha)$.
The singular NKS category $\cs$ is isomorphic to $k[X]/( X^2)$.
\end{example}

\begin{example} [Split type $A_2$]
  \label{example quasi-split sl3}
Let $Q$ be the quiver $(1 \longrightarrow 2)$. Let $C=\{\Sigma^j \ts_i\mid i=1,2,j\in\Z\}$ and $F^\imath=\Sigma$. Then $\mcr^\imath$ is given by
\eqref{figure 3}:
%\begin{center}\setlength{\unitlength}{0.6mm}
%\begin{figure}
%\begin{equation}
%\label{figure 3}
% \begin{picture}(80,35)(0,30)
% \put(0,20){\tiny$\ts_1$}
% \put(20,40){\tiny $P_2$}
% \put(40,20){\tiny $\ts_2$}
%  \put(38,60){\tiny $\sigma( \ts_1)$}
% \put(18,0){\tiny $\sigma(\ts_2)$}
%  \put(60,40){\tiny $\ts_1$}
%\put(6,24){\vector(1,1){15}}
%\put(6,19){\vector(1,-1){15}}
% \put(46,24){\vector(1,1){15}}
%  \put(26,44){\vector(1,1){15}}
%  \put(26,4){\vector(1,1){15}}
% \put(46,59){\vector(1,-1){15}}
% \put(26,39){\vector(1,-1){15}}
%\put(7,17){\begin{picture}(0,0)
%\dashline{3}(0,4)(31,4)
%\put(0,4){\vector(-1,0){1}}
%\end{picture}
%}

%\put(27,37){\begin{picture}(0,0)
%\dashline{3}(1,4)(31,4)
%\put(1,4){\vector(-1,0){1}}
%\end{picture}
%}

%\put(2,37){\begin{picture}(0,0)
%\dashline{3}(1,4)(16,4)
%\put(1,4){\vector(-1,0){1}}
%\end{picture}
%}

%\put(47,17){\begin{picture}(0,0)
%\dashline{3}(1,4)(14,4)
%\put(1,4){\vector(-1,0){1}}
%\end{picture}
%}
%\end{picture}
%\caption{The regular NKS category $\mcr^\imath$ of split type $\mathfrak{sl}_3$}
%\end{equation}
%\vspace{1.9cm}
%\end{center}

\begin{center}
\begin{equation}
\label{figure 3}
\begin{tikzpicture}[scale=1.2]
%\fill[opacity=0.5,fill=purple!60] 
(1.9,0) --(5,3.1) -- (5.1,3) -- (2,-0.1)--(1.9,0);
%\fill[opacity=0.5,fill=purple!60] 
(3.9,0) --(7,3.1) -- (7.1,3) -- (4,-0.1)--(3.9,0);

%\fill[opacity=0.5,fill=black!60]  
(1.9,0) --(3,1.1)--(4.1,0)--(4,-0.1)--(3,0.9)--(2,-0.1)--(1.9,0);
%\fill[opacity=0.5,fill=black!60] 
 (7,3.1) --(7.1,3)--(6,1.9)--(4.9,3)--(5,3.1)--(6,2.1)--(7,3.1);

%\fill[opacity=0.5,fill=purple!40] (4.9,3) --(5,3.1) -- (8.1,0) -- (8,-0.1)--(4.9,3);
%\fill[opacity=0.5,fill=purple!40] (6.9,3) --(7,3.1) -- (10.1,0) -- (10,-0.1)--(6.9,3);

\node at (2,0) {\tiny$\sigma(\ts_1)$};
\node at (4,0) {\tiny$\sigma(\ts_2)$};
\node at (3,1) {\tiny$\ts_1$};
\node at (4,2) {\tiny$P_2$};
\node at (5,3) {\tiny$\sigma(\ts_1)$};
\node at (5,1) {\tiny$\ts_2$};
\node at (6,2) {\tiny$\ts_1$};
\node at (7,3) {\tiny$\sigma(\ts_2)$};
\draw[->] (2.1,0.1) -- (2.9,0.9);
\draw[->] (3.1,0.9) -- (3.9,0.1);
\draw[->] (3.1,1.1) -- (3.9,1.9);
\draw[->] (4.1,1.9) -- (4.9,1.1);
\draw[->] (4.1,2.1) -- (4.9,2.9);
\draw[->] (5.1,2.9) -- (5.9,2.1);

\draw[->] (6.1,2.1) -- (6.9,2.9);
\draw[->] (5.1,1.1) -- (5.9,1.9);
\draw[->] (4.1,0.1) -- (4.9,0.9);

%\draw[style=dashed] (3.3,1) -- (4.8,1);
%\draw[->] (3.3,1)--(3.2,1);
%\draw[style=dashed] (4.3,2) -- (5.8,2);
%\draw[->] (4.3,2)--(4.2,2);

%\draw[style=dashed] (5.3,1) -- (6,1);
%\draw[->] (5.3,1)--(5.2,1);
%\draw[style=dashed] (3.2,2) -- (3.9,2);
%\draw[->] (3.2,2)--(3.1,2);
\end{tikzpicture}
\end{equation}
\end{center}

Moreover, $\cs^\imath$ is described in terms of the bound quiver \eqref{figure 5}:
\begin{center}\setlength{\unitlength}{0.7mm}
 \begin{equation}
 \label{figure 5}
 \begin{picture}(50,13)(0,0)
\put(-7,-3){\tiny$\sigma(\ts_1)$}
\put(4,0){\vector(1,0){14}}
\put(10,0){$^{\alpha}$}
\put(19,-3){\tiny$\sigma(\ts_2)$}
%
%\color{purple}
\put(0,9){\small $\varepsilon_1$}
\put(20,9){\small $\varepsilon_2$}
\qbezier(-1,1)(-3,3)(-2,5.5)
\qbezier(-2,5.5)(1,9)(4,5.5)
\qbezier(4,5.5)(5,3)(3,1)
\put(3.1,1.4){\vector(-1,-1){0.3}}
\qbezier(19,1)(17,3)(18,5.5)
\qbezier(18,5.5)(21,9)(24,5.5)
\qbezier(24,5.5)(25,3)(23,1)
\put(23.1,1.4){\vector(-1,-1){0.3}}
\end{picture}
\end{equation}
\vspace{0.2cm}
\end{center}
\[
\varepsilon_1^2=0=\varepsilon_2^2, \quad \varepsilon_2 \alpha=\alpha\varepsilon_1.
\]
%\red{The black/purple paths in \eqref{figure 3} give rise to the above bound quiver.}
\end{example}

\begin{example}[Quasi-split type $A_3$]
Let $Q$ be the quiver $(\xymatrix{1\ar[r]^{h_2} & 2& 3\ar[l]_{h_1}})$ with the involution $\btau$ such that $\btau(1)=3$, and $\btau(2)=2$.
Then $\mcr^\imath$ is given by \eqref{figure 4}:

\begin{center}
\begin{equation}
\label{figure 4}
\begin{tikzpicture}[scale=1.2]
%\fill[opacity=0.5,fill=purple!60]
 (1.9,0) --(6,4.1) -- (6.1,4) -- (2,-0.1)--(1.9,0);

%\fill[opacity=0.5,fill=purple!60] 
(1.9,4) --(6,-0.1) -- (6.1,0) -- (2,4.1)--(1.9,4);
%\fill[opacity=0.5,fill=purple!60] 
(0.9,1.9) --(0.9,2.1) -- (5.2,2.1) -- (5.2,1.9)--(0.9,1.9);

%\fill[opacity=0.5,fill=black!60] 
(1.9,0) --(4,2.1) -- (5.2,2.1) -- (5.2,1.9)--(4,1.9)--(2,-0.1)--(1.9,0);

%\fill[opacity=0.5,fill=black!60] 
(1.9,4) --(2,4.1) -- (4,2.1)--(5.2,2.1) -- (5.2,1.9)--(4,1.9)--(1.9,4);

\node at (2,0) {\tiny$\sigma(\ts_1)$};
\node at (3,1) {\tiny$\ts_1$};
\node at (4,2) {\tiny$P_2$};
\node at (5,3) {\tiny$I_1$};
\node at (6,4) {\tiny$\sigma(\ts_3)$};
\node at (2,2) {\tiny$\ts_2$};
\node at (2,4) {\tiny $\sigma(\ts_3)$};
\node at (3,3) {\tiny $\ts_3$};
\node at (5,1) {\tiny$I_3$};
\node at (6,0) {\tiny$\sigma(\ts_1)$};
\node at (1,3) {\tiny$I_1$};
\node at (1,1) {\tiny$I_2$};
\node at (1,2) {\tiny$\sigma(\ts_2)$};
\node at (5,2) {\tiny$\sigma(\ts_2)$};

\node at (6,2) {\tiny$\ts_2$};
\draw[->] (2.1,0.1) -- (2.9,0.9);
\draw[->] (5.1,0.9) -- (5.9,0.1);
\draw[->] (3.1,1.1) -- (3.9,1.9);
\draw[->] (4.1,1.9) -- (4.9,1.1);
\draw[->] (4.1,2.1) -- (4.9,2.9);
\draw[->] (5.1,2.9) -- (5.9,2.1);

\draw[->] (5.1,3.1) -- (5.9,3.9);
\draw[->] (5.1,1.1) -- (5.9,1.9);

\draw[->] (1.1,3.1) -- (1.9,3.9);
\draw[->] (1.1,1.1) -- (1.9,1.9);
\draw[->] (2.1,2.1) -- (2.9,2.9);
\draw[->] (2.1,1.9) -- (2.9,1.1);
\draw[->] (2.1,3.9) -- (2.9,3.1);
\draw[->] (3.1,2.9) -- (3.9,2.1);
\draw[->] (1.1,2.9) -- (1.9,2.1);
\draw[->] (1.1,0.9) -- (1.9,0.1);

\draw[->] (1.3,2) -- (1.85,2);
\draw[->] (5.3,2) -- (5.85,2);
\draw[->] (4.15,2) -- (4.7,2);

%\draw[->] (4.1,0.1) -- (4.9,0.9);

\end{tikzpicture}
\end{equation}
\end{center}

Moreover, $\cs^\imath$ is described in terms of the bound quiver \eqref{figure 7}:
\begin{center}\setlength{\unitlength}{0.7mm}
\begin{equation}
\label{figure 7}
 \begin{picture}(50,20)(0,-10)
\put(0,-2){$1$}
\put(20,-2){$3$}
\put(2,-11){$_{\alpha}$}
\put(17,-11){$_{\beta}$}
\put(2,-2){\vector(1,-2){8}}
\put(20,-2){\vector(-1,-2){8}}
\put(10,-22){$2$}
%\color{purple}
\put(3,1){\vector(1,0){16}}
\put(19,-1){\vector(-1,0){16}}
\put(10,1){$^{\varepsilon_1}$}
\put(10,-4){$_{\varepsilon_3}$}
\put(10,-28){$_{\varepsilon_2}$}
\begin{picture}(50,23)(-10,19)
%\color{purple}
\qbezier(-1,-1)(-3,-3)(-2,-5.5)
\qbezier(-2,-5.5)(1,-9)(4,-5.5)
\qbezier(4,-5.5)(5,-3)(3,-1)
\put(3.1,-1.4){\vector(-1,1){0.3}}
\end{picture}
\end{picture}
\end{equation}
\vspace{1cm}
\end{center}
\[
\varepsilon_1\varepsilon_3=0=\varepsilon_3\varepsilon_1,
\quad
 \varepsilon_2^2=0,
 \quad
 \varepsilon_2 \beta=\alpha\varepsilon_3,
 \quad
 \varepsilon_2 \alpha=\beta\varepsilon_1.
\]

\end{example}
\subsection{Coverings of $\imath$NKS categories}
\label{subsec: covering}

From the proof of Theorem \ref{thm:iNKS}, there exists a Galois covering
$\cv: \mcr\rightarrow \mcr^\imath$, which restricts to the Galois covering $\cv: \cs\rightarrow\cs^\imath$; this is a reformulation of the Galois covering $\cv:\Lambda\rightarrow \Lambda^\imath$ in \cite[Remark 2.11]{LW19}. Hence we have the following commutative diagram
\[
\xymatrix{
\cs \ar@{^{(}->}[r]\ar[d]^\cv & \mcr \ar[d]^\cv \\
\cs^\imath  \ar@{^{(}->}[r] & \mcr^\imath .}
\]

By the covering theory \cite{Ga}, we have a pullback functor $\cv^*:\mod(\mcr^\imath)\rightarrow \mod(\mcr)$ (respectively, $\cv^*: \mod(\cs^\imath)\rightarrow \mod(\cs)$) given by
$\cv^*M:=M\circ\cv$ for any $M\in\mod(\mcr^\imath)$; and a pushdown functor $\cv_*:\mod(\mcr)\rightarrow \mod(\mcr^\imath)$ (respectively, $\cv_*: \mod(\cs)\rightarrow \mod(\cs^\imath)$).
In fact, $\cv_*,\cv^*$ are exact, and $(\cv_*,\cv^*)$ and $(\cv^*,\cv_*)$ are adjoint pairs, see, e.g., \cite[Theorem~ 2.5.1]{NV04}.
%As $\cv$ is dense, $\cv^*$ is full faithful, see e.g. \cite[Page 47]{CM06}.

Let $\texttt{G} =\langle g\mid g^2=1\rangle$ be the cyclic group of order $2$. We define a $\texttt{G}$-action on $\mod(\mcr)$ as follows: for any $M\in \mod(\mcr)$, regarded as a $k$-linear functor $\mcr^{op}\rightarrow \mod(k)$, set $g\cdot M(x):=M(F^\imath x)$ for any $x\in\mcr_0$. By restriction we obtain a $\texttt{G}$-action on $\mod(\cs)$.

Denote by $\mod^\texttt{G}(\cs)$ the subcategory of $\mod(\cs)$
formed by the $\texttt{G}$-invariant modules, see \cite[Page 94]{Ga} or \cite[Definition 6.1]{As}. Then
the pullback functor $\cv^*: \mod(\cs^\imath) \rightarrow \mod(\cs)$ induces an equivalence $\mod(\cs^\imath) \simeq \mod^\texttt{G}(\cs)$, see \cite[Theorem 6.2]{As} or \cite[Theorem 4.3]{CM06}.
Similarly, we can define $\mod^\texttt{G}(\mcr)$ and then the pullback functor $\cv^*: \mod(\mcr^\imath) \rightarrow \mod(\mcr)$ induces an equivalence $\mod(\mcr^\imath) \simeq \mod^\texttt{G}(\mcr)$.
In particular, for any $M\in\mod(\cs)$ or $M\in\mod(\mcr)$, we have $\cv^*\cv_*(M)=M\oplus g\cdot M$, see, e.g., \cite[Page 122]{As}.

The isomorphism $F^\imath: \cs\rightarrow \cs$ corresponds to the involution $\btau^\sharp$ on $\Lambda$. Then the $\texttt{G}$-action of $\mod(\cs)$ corresponds to the following $\texttt{G}$-action on $\mod(\Lambda)$. More explicitly, for any $\Lambda$-module $M=(M_{i},M_{i'}; M(\alpha),M(\alpha'),M(\varepsilon_{i}),M(\varepsilon_{i'}))_{(i,\alpha)\in Q_0\times Q_1}$, let
$$
g\cdot M=(N_{i},N_{i'}; N(\alpha),N(\alpha'),N(\varepsilon_{i}),N(\varepsilon_{i'}))_{(i,\alpha)\in Q_0\times Q_1},
$$
where
$N_{i}=M_{\btau i'}$, $N_{i'}=M_{\btau i}$, $N(\alpha)=M(\btau\alpha')$, $N(\alpha')=M(\btau\alpha)$, $N(\varepsilon_{i})=M(\varepsilon_{\btau i'})$ and
$N(\varepsilon_{i'})=M(\varepsilon_{\btau i})$ for any $(i,\alpha)\in Q_0\times Q_1$.

For any dimension vector ${\bf u}\in \N^{\mcr^\imath_0}$, denote by $\cv^*({\bf u})\in \N^{\mcr_0}$ the dimension vector such that
$$
\cv^*({\bf u})(x)={\bf u}(\cv(x)),\text{ for all }x\in \mcr_0.
$$
We have $\dimv (\cv^*M)=\cv^*(\dimv M)$.
Similarly, we can define $\cv^*(\bw)\in \N^{\cs_0}$ for any $\bw\in \N^{\cs^\imath_0}$,  $F^{\imath*}(\alpha)\in \N^{\mcr_0}$ for any $\alpha\in \N^{\mcr_0}$, and  $F^{\imath*}(\beta)\in \N^{\cs_0}$ for any $\beta\in \N^{\cs_0}$.

For any $x\in\mcr_0$ (respectively, $\cs_0$, $\mcr^\imath_0$ and $\cs^\imath_0$), we denote by $\ts_x$ its corresponding simple module. For any $M\in\mod(\mcr)$, then $M$ is stable (respectively, costable, bistable) if and only if so is $g\cdot M$.

%\begin{proof}
%For any $x\in\mcr_0-\cs_0$, we have $\ts_{x}$ is a submodule of $M$ (respectively, a quotient of $M$) if and only $\ts_{F^\imath (x)}$ is a submodule of $g\cdot M$ (respectively, a quotient of $g\cdot M$). Then the desired result follows by definition.
%\end{proof}

\begin{lemma}\label{lem:rho preserves stable}
Let $M\in \mod(\mcr^\imath)$, $N\in \mod(\mcr)$. Then,
\begin{itemize}
\item[(i)] $M$ is stable (respectively, costable, bistable) if and only if so is $\cv^*M$;
\item[(ii)] $N$ is stable (respectively, costable, bistable) if and only if so is $\cv_*N$.
 \end{itemize}
\end{lemma}

\begin{proof}
(i) We have $\Hom_{\mcr}(\ts_x,\cv^*M)\cong \Hom_{\mcr^\imath}(\cv_* \ts_x,M)=\Hom_{\mcr^\imath}(\ts_{\cv x},M)$ for any $x\in\mcr_0-\cs_0$. Then $M$ is stable if and only if $\cv^*M$ is stable.
Dually, one can prove that $M$ is costable if and only if $\cv^*M$ is costable by using the adjoint pair $(\cv^*,\cv_*)$.

(ii) The argument is similar by noting that $\cv^* \ts_z= \ts_x\oplus \ts_{F^\imath(x)}$ for any $z\in\mcr^\imath_0-\cs^\imath_0$ and $x\in\mcr_0-\cs_0$ such that $\cv(x)=z$.
\end{proof}

%\begin{lemma}
%(i) For any non-projective indecomposable $M\in\Gproj(\Lambda)$, we have $\End(M)\cong k$.

%(ii)
%For any non-projective indecomposable $M\in\Gproj(\Lambda^\imath)$, we have $\End(M)\cong k$.
%\end{lemma}

\begin{lemma}
\label{lem:embedding of varieties}
Let $\bv \in \N^{\mcr^\imath_0-\cs^\imath_0}$ and $\bw \in \N^{\cs^\imath_0}$. Then
the pullback functor $\cv^*$ induces the following embeddings of varieties:
\begin{itemize}
\item[(i)] $\cv^*:\cm(\bv,\bw,\mcr^\imath)\hookrightarrow \cm(\cv^*\bv,\cv^*\bw,\mcr)$,
\item[(ii)] $\cv^*: \cm_0(\bv,\bw,\mcr^\imath)\hookrightarrow \cm_0(\cv^*\bv,\cv^*\bw,\mcr)$,
\item[(iii)] $\cv^*:\cm_0(\bw,\mcr^\imath)\hookrightarrow \cm_0(\cv^*\bw,\mcr)$.
\end{itemize}
\end{lemma}

\begin{proof}
By Lemma \ref{lem:rho preserves stable} we have an embedding $\tilde{\cv}^*: \cs t(\bv,\bw,\mcr^\imath)\hookrightarrow \cs t(\cv^*\bv,\cv^*\bw,\mcr)$ induced by the pullback $\cv^*$. For any $z\in\mcr^\imath_0-\cs^\imath_0$, choose one $x_z\in\mcr_0-\cs_0$ such that $\cv(x_z)=z$.
Because $(\cv^*\bv)(x_z)=(\cv^*\bv)(F^\imath(x_z))=\bv(z)$, we have
two canonical isomorphisms $\iota_1:G_\bv\rightarrow \prod_{z\in \mcr^\imath_0-\cs^\imath_0} (GL(k^{\cv^*\bv(x_z)}))$ and $\iota_2:G_\bv\rightarrow \prod_{z\in \mcr^\imath_0-\cs^\imath_0} (GL(k^{\cv^*\bv(F^\imath (x_z))}))$ by identifying $GL(k^{\cv^*\bv(x_z)})$ and  $GL(k^{\cv^*\bv(F^\imath(x_z))})$ with $GL(k^{\bv(z)})$. In this way, we have identified $G_{\cv^* \bv}\equiv G_\bv\times G_\bv.$ By the diagonal embedding, we can view $G_\bv$ as a subgroup of $G_{\cv^*\bv}$. Therefore, $\tilde{\cv}^*:\cs t(\bv,\bw,\mcr^\imath)\hookrightarrow \cs t(\cv^*\bv,\cv^*\bw,\mcr)$ induces a morphism of varieties
$$\cv^*:\cm(\bv,\bw,\mcr^\imath)=\cs t(\bv,\bw,\mcr^\imath)/G_\bv\rightarrow  \cm(\cv^*\bv,\cv^*\bw,\mcr)=\cs t(\cv^*\bv,\cv^*\bw,\mcr)/G_{\cv^*\bv}.$$

(i)
For any $M,N\in \cs t(\bv,\bw,\mcr^\imath)$ such that $\cv^*M$ and $\cv^*N$ coincide in $\cm(\cv^*\bv,\cv^*\bw,\mcr)$, by definition, there exists $(g_x)_{x\in\mcr_0-\cs_0}\in G_{\cv^*\bv}$ such that
the following diagram commutes:
\[\xymatrix{(\cv^*M)_x= k^{\bv(x)} \ar[r]^{M(\alpha)} \ar[d]^{g_x}  & (\cv^*M)_y= k^{\bv(y)}\ar[d]^{g_y}\\
(\cv^*N)_x= k^{\bv(x)} \ar[r]^{N(\alpha)}   & (\cv^*N)_y= k^{\bv(y)} }\]
for any arrow $\alpha:y\rightarrow x$ in $\mcr$; here we regard $g_z=\Id$ for any frozen vertex $z \in \cs_0$.
In particular, $((g_x)_{x\in\mcr_0-\cs_0}, (g_y=\Id)_{y\in\cs_0})$ induces an isomorphism $\cv^*M\stackrel{\cong}{\longrightarrow} \cv^*N$.
Let $\bar{g}=(\bar{g}_z)_{z\in\mcr^\imath_0-\cs^\imath_0}$, where $\bar{g}_z=\frac{1}{2}( g_{x_z}+g_{F^\imath(x_z)})$ for any $z\in\mcr^\imath_0-\cs^\imath_0$.

By definition, for any arrow $\beta$ in $\mcr^\imath$, we have $(\cv^*M)(\alpha)=M(\beta)=(\cv^*M)(F^\imath(\alpha))$ for any arrow $\alpha$ in $\mcr$ such that $\cv(\alpha)=\beta$.
Therefore,  $\phi:=((\bar{g}_z)_{z\in\mcr^\imath_0-\cs^\imath_0},(g_{\sigma(c)}=\Id)_{c\in C} )$ is an $\mcr^\imath$-module morphism from $M$ to $N$.
As $M,N$ are stable, their socles are supported on frozen vertices. Since the restriction of $\phi$ to any frozen vertices $\sigma(c)$ is the identity, $\phi|_{\soc(M)}:\soc(M)\rightarrow\soc(N)$ is injective. If follows %from (the dual) of \cite[Chapter I.3, Corollary 3.9]{ASS}
that $\phi$ is injective, and then $\phi$ is an isomorphism since $\dimv M=\dimv N$.
Moreover, $(\bar{g}_z)_{z\in\mcr^\imath_0-\cs^\imath_0}\in G_\bv$, and hence $M$ and $N$ coincide in $\cm(\bv,\bw,\mcr^\imath)$. Therefore, $\cv^*:\cm(\bv,\bw,\mcr^\imath)\rightarrow \cm(\cv^*\bv,\cv^*\bw,\mcr)$ is injective.

%(ii) Obviously, $\cv^*:\rep(\bv,\bw,\mcr^\imath)\hookrightarrow \rep(\cv^*\bv,\cv^*\bw,\mcr)$ induces a map of varieties:
%$$\cv^*: \cm_0(\bv,\bw,\mcr^\imath)\rightarrow \cm_0(\cv^*\bv,\cv^*\bw,\mcr).$$
%It is well known that there is a bijection $$\cm_0(\bv,\bw,\mcr^\imath)\simeq \{\text{closed }G_\bv\text{-orbits in }\rep(\bv,\bw,\mcr^\imath)\}$$
%and a bijection $$\cm_0(\cv^*\bv,\cv^*\bw,\mcr)\simeq \{\text{closed }G_{\cv^*\bv}\text{-orbits in }\rep(\cv^*\bv,\cv^*\bw,\mcr)\}.$$
%It follows from Lemma \ref{lem:bistable and stable} that the closed $G_\bv$-orbits in $\rep(\bv,\bw,\mcr^\imath)$ are represented by $L\oplus N\in\rep(\bv,\bw,\mcr^\imath)$ where $L$ is a bistable module and $N$ is a semi-simple module such that $\res N$ vanishes.
%Let $L_1\oplus N_1$ and $L_2\oplus N_2$ be such that $L_1,L_2$ are bistable and $N_1,N_2$ are semi-simple module such that $\res N_1,\res N_2$ vanish. If $\cv^*L_1\oplus \cv^*N_1$ and $\cv^*L_2\oplus \cv^*N_2$ coincide in $\cm_0(\cv^*\bv,\cv^*\bw,\mcr)$, then there is an isomorphism induced by the group action $G_{\cv^*\bv}$, which is denote by $\phi$. Then $\phi=\diag\{\phi_1,\phi_2\}$, where $\phi_1: \cv^*L_1\rightarrow \cv^*L_2$ and $\phi_2:\cv^*N_1\rightarrow \cv^* N_2$ by noting that $\cv^*L_1,\cv^* L_2$ are bistable and $\cv^*N_1,\cv^*N_2$ are semisimple such that $\res N_1,\res N_2$ vanish.
%similar to the proof of (i), we have $L_1\oplus N_1$ and $L_2\oplus N_2$ coincide in $\cm_0(\bv,\bw,\mcr^\imath)$. So $\cv^*  \cm_0(\bv,\bw,\mcr^\imath)\hookrightarrow \cm_0(\cv^*\bv,\cv^*\bw,\mcr)$ is injective.

(ii)
According to \cite{LBP, Lus90}, the ring of invariants (either $\C[\rep(\cv^*\bv,\cv^*\bw,\mcr)]^{G_\bv\times G_\bv}$ or $\C[\rep(\bv,\bw,\mcr^\imath)]^{G_\bv}$) is generated by coordinate maps to paths from a frozen vertex to a frozen vertex and trace maps to cyclic paths in non-frozen vertices.

{\bf Claim.} Any cyclic path $\ell$ in $\mcr$ or $\mcr^\imath$ in non-frozen vertices is nilpotent, i.e., there is some $n\geq1$ such that $\ell^n=0$.

We prove the Claim only for $\mcr^\imath$. In fact, by noting that $\cp^\imath=\mcr^\imath/( \cs^\imath )\simeq \cd_Q/F^\imath$, the proof of \cite[Lemma 9.1]{LW19} shows that
$\End_{\cd_Q/F^\imath}(M)\cong k$ for any indecomposable $kQ$-module $M$. So $\ell$ is zero in $\cp^\imath$, which means that $\ell$ factors through some frozen vertex $\sigma(c)$.
On the other hand, $\Ind\proj(\Lambda^\imath)\cong\cs^\imath$ and $\Lambda^\imath$ is finite-dimensional, so any cyclic path passing through $\sigma(c)$ is nilpotent. The Claim is proved.

By the Claim, the trace maps to cyclic paths in non-frozen vertices are zero. Then both $\C[\rep(\cv^*\bv,\cv^*\bw,\mcr)]^{G_\bv\times G_\bv}$ and $\C[\rep(\bv,\bw,\mcr^\imath)]^{G_\bv}$
are generated by coordinate maps to paths from a frozen vertex to a frozen vertex.
Hence $\cv:\mcr\rightarrow \mcr^\imath$ induces an epimorphism $\cv: \C[\rep(\cv^*\bv,\cv^*\bw,\mcr)]^{G_\bv\times G_\bv}\twoheadrightarrow \C[\rep(\bv,\bw,\mcr^\imath)]^{G_\bv}$, whence the embedding $\cm_0(\bv,\bw,\mcr^\imath)\hookrightarrow \cm_0(\cv^*\bv,\cv^*\bw,\mcr)$.

(iii) Follows from (ii) and the definition.
\end{proof}

The second part of the following lemma was already known in \cite{Sch18}.

\begin{lemma}
  \label{lem:M=rep}
We have $\cm_0(\bw,\mcr^\imath)\cong \rep(\bw,\cs^\imath)$ and $\cm_0(\cv^*\bw,\mcr)\cong \rep(\cv^*\bw,\cs)$.
\end{lemma}

\begin{proof}
%From Lemma \ref{lem:isomorhism for M0 and repS},
By \cite[Theorem 3.12]{Sch18}, the lemma follows once we check the following assumption holds for all indecomposable projective $kQ$-module $P_i$ and all $n\in\Z-\{0\}$:
\begin{align}
	\label{eq:isomvarieties}
\Hom_{\cd_Q}(P_i,F^nP_i)=0,
\qquad
\Hom_{\cd_Q}(P_i, (F^{\imath})^{n}P_i)=0
\end{align}
This indeed follows from \cite[Lemma 9.1]{LW19}.
\end{proof}

The map $\cv^*:\cm_0(\bw,\mcr^\imath)\hookrightarrow \cm_0(\cv^*\bw,\mcr)$ in Lemma \ref{lem:embedding of varieties}(iii) corresponds to the pullback $\cv^*: \rep(\bw,\cs^\imath)\rightarrow \rep(\cv^*\bw,\cs)$. Let $\rep^\texttt{G}(\cv^*\bv,\cv^*\bw,\mcr)$ be the subvariety of $\rep(\cv^*\bv,\cv^*\bw,\mcr)$ consisting of $\texttt{G}$-invariant modules, and $\rep^\texttt{G}(\cv^*\bw,\cs)$ is defined similarly. Then we have $\rep^\texttt{G}(\cv^*\bv,\cv^*\bw,\mcr) \cong \rep(\bv,\bw,\mcr^\imath)$ and $\rep^\texttt{G}(\cv^*\bw,\cs) \cong \rep(\bw,\cs^\imath)$.

Using the diagonal embedding $G_\bv\rightarrow G_\bv\times G_\bv\cong G_{\cv^*\bv}$, we view $G_\bv$ as a subgroup of $G_{\cv^*\bv}$ in the following.
Let $\cs t(\cv^*\bv,\cv^*\bw,\mcr)$ be the subset of $\rep(\cv^*\bv,\cv^*\bw,\mcr)$ consisting of all stable $\mcr$-modules with dimension vector $(\cv^*\bv,\cv^*\bw)$. Let $\cs {\it t}^\texttt{G}(\cv^*\bv,\cv^*\bw,\mcr)$ be the subvariety of $\cs t(\cv^*\bv,\cv^*\bw,\mcr)$ formed by the $\texttt{G}$-invariant $\mcr$-modules.

Let $\cm_0^\texttt{G}(\cv^*\bv,\cv^*\bw,\mcr) :=\rep^\texttt{G}(\cv^*\bv,\cv^*\bw,\mcr)//G_\bv$ be the categorical quotient.
For all $\bv'\leq \bv$ (cf. \eqref{eq:leq}), there is an inclusion
\[
\rep^\texttt{G}(\cv^*\bv',\cv^*\bw,\mcr)\longrightarrow \rep^\texttt{G}(\cv^*\bv,\cv^*\bw,\mcr)
\]
by taking a direct sum with the semisimple module of dimension vector $\cv^*\bv-\cv^*\bv'$. This yields an inclusion
\[
\rep^\texttt{G}(\cv^*\bv',\cv^*\bw,\mcr)//G_{\bv'} \longrightarrow \rep^\texttt{G}(\cv^*\bv,\cv^*\bw,\mcr)//G_\bv.
\]
The affine quiver variety $\cm_0^\texttt{G}(\cv^*\bw,\mcr)$ is defined to be the colimit of $\cm_0^\texttt{G}(\cv^*\bv,\cv^*\bw,\mcr)$ along the inclusions.

\begin{proposition}
  \label{prop:MMi}
For any dimension vectors $\bv \in \N^{\mcr^\imath_0-\cs^\imath_0}$ and $\bw \in \N^{\cs^\imath_0}$, we have the following isomorphisms:
\begin{itemize}
\item[(i)] $\cm(\bv,\bw,\mcr^\imath)\cong\cm^\texttt{G}(\cv^*\bv,\cv^*\bw,\mcr)$;
\item[(ii)] $\cm_0(\bv,\bw,\mcr^\imath)\cong\cm_0^\texttt{G}(\cv^*\bv,\cv^*\bw,\mcr)$;
\item[(iii)] $\cm_0(\bw,\mcr^\imath) \cong \cm_0^\texttt{G}(\cv^*\bw,\mcr)
            \cong  \rep^\texttt{G}(\cv^*\bw,\cs) \cong\rep(\bw,\cs^\imath)$.
\end{itemize}
\end{proposition}

\begin{proof}
(i) follows from Lemma \ref{lem:embedding of varieties}(i) by noting that the pullback functor $\cv^*: \mod(\mcr^\imath)\rightarrow \mod(\mcr)$ induces an equivalence $\mod(\mcr^\imath)\simeq \mod^\texttt{G}(\mcr)$.

(ii) We have $\C[\rep(\bv,\bw,\mcr^\imath)]^{G_\bv}\cong \C[\rep^\texttt{G}(\cv^*\bv,\cv^*\bw,\mcr)]^{{G_\bv}}$ from the equivalence $\mod(\mcr^\imath) \simeq \mod^\texttt{G}(\mcr)$. % as a subring of $\C[\rep(\cv^*\bv,\cv^*\bw,\mcr)]^{G_\bv\times G_\bv}$.
Note that $\C[\rep^\texttt{G}(\cv^*\bv,\cv^*\bw,\mcr)]^{{G_\bv}\times G_\bv}\cong \C[\rep^\texttt{G}(\cv^*\bv,\cv^*\bw,\mcr)]^{{G_\bv}}$.
By restricting the epimorphism $\cv:\C[\rep(\cv^*\bv,\cv^*\bw,\mcr)]^{G_\bv\times G_\bv}\rightarrow \C[\rep(\bv,\bw,\mcr^\imath)]^{G_\bv}$ in the proof of Lemma \ref{lem:embedding of varieties}(ii) to $\C[\rep^\texttt{G}(\cv^*\bv,\cv^*\bw,\mcr)]^{G_\bv\times G_\bv}$, we obtain an isomorphism $$ \C[\rep^\texttt{G}(\cv^*\bv,\cv^*\bw,\mcr)]^{G_\bv} \stackrel{\cong}{\longrightarrow} \C[\rep^\texttt{G}(\cv^*\bv,\cv^*\bw,\mcr)]^{G_\bv\times G_\bv}  \stackrel{\cong}{\longrightarrow} \C[\rep(\bv,\bw,\mcr^\imath)]^{G_\bv}.$$ So $\cm_0(\bv,\bw,\mcr^\imath)\cong\cm_0^\texttt{G}(\cv^*\bv,\cv^*\bw,\mcr)$.

(iii) The first isomorphism follows from (ii), and the rest follows by Lemma~\ref{lem:M=rep}.
\end{proof}

\subsection{Transversal slices}
As explained in Example \ref{ex:cyclic},   
the NKS quiver varieties with $F=\Sigma^2$  %red{what is $n$? or $-n$?} \red{$F=\Sigma^2$? } 
%Before going on, we recall the transversal slice theorem for cyclic quiver varieties obtained in \cite[Theorem 3.3.2]{Na01}. 
%\begin{theorem}[Transversal slice theorem, \text{\cite[Theorem 3.3.2]{Na01}}]
%	\label{lem:Translice1}
%	Let $\cs,\mcr$ be NKS categories constructed from an admissible pair $(F,C)$ such that $F=\tau^n$. Then there exist neighborhoods $\cu$, $\cu_T$, $\cu^\bot$ of $x$ in $\cm_0(\bv,\bw,\mcr)$ and the origins in $T$, $\cm_0(\bv^\bot,\bw^\bot,\mcr)$, respectively, and biholomorphic maps $\cu\rightarrow \cu_T\times \cu^\bot$, $\pi^{-1}(\cu)\rightarrow \cu_T\times \pi^{-1}(\cu^\bot)$, such that the following diagram commutes:
%	\[\xymatrix{ \cm(\bv,\bw,\mcr) \supset \pi^{-1}(\cu) \ar[rr]^{\cong\qquad\quad} \ar[d]_\pi & &\cu_T\times \pi^{-1}(\cu^\bot) \subset T\times \cm(\bv^\bot,\bw^\bot,\mcr) \ar[d]^{1\times\pi} \\
%		\cm_0(\bv,\bw,\mcr)\supset \cu \ar[rr]^{\cong\qquad\quad} && \cu_T\times \cu^\bot \subset T\times \cm_0(\bv^\bot,\bw^\bot,\mcr)   }\]
%	In particular, $\pi^{-1}(x)$ is biholomorphic to the fibre $\pi^{-1}(0)$ over the origin of $\pi:\cm(\bv^\bot,\bw^\bot,\mcr)\rightarrow \cm_0(\bv^\bot,\bw^\bot,\mcr)$. 
%\end{theorem}
associated to NKS categories $\mcr$ and $\cs$ in Definition~\ref{def:RS for QG} and Definition \ref{def:RS for iQG} are Nakajima's cyclic quiver varieties \cite{Na01, Na04} by noting that $\Sigma^2=\tau^{-h}$, where $h$ is the Coxeter number of the Dynkin type of $Q$; see, e.g., \cite[Example 8.3~(2)]{Ke05} and also \cite{Qin}. In this case, there is a fundamental result on the existence of transversal slices to strata; see \cite[Theorem 3.3.2]{Na01}.  
 
 We shall construct the transversal slices to strata of the $\imath$NKS varieties associated to the $\imath$NKS categories $\mcr^\imath$ and $\cs^\imath$ in Definition~\ref{def:RS for iQG} and Theorem \ref{thm:iNKS}; also cf. \cite{Qin14}.  

Let  $\bv \in \N^{\mcr^\imath_0-\cs^\imath_0}$ and $\bw \in \N^{\cs^\imath_0}$ be two dimension vectors.  Let $(\bv^0,\bw)$ be a strongly $l$-dominant pair. Note that $\cm_0^{\text{reg}}(\bv^0,\bw,\mcr^\imath)$ is nonempty. Assume that $x\in \cm_0^{\text{reg}}(\bv^0,\bw,\mcr^\imath)\subset \cm_0(\bv,\bw,\mcr^\imath)$. Let $T$ be the tangent space of $\cm_0^{\text{reg}}(\bv,\bw,\mcr^\imath)$ at $x$. %\red{of the stratum containing $x$?}. 
Define $\bw^\bot$ and $\bv^\bot$ to be 
\[
\bw^\bot :=\bw-{\mathcal{C}}_q\bv^0\sigma^{-1},
\qquad
\bv^\bot :=\bv-\bv^0, 
\]
respectively.

\begin{proposition}[Transversal slices for $\imath$NKS quiver varieties]
	\label{prop:Translice1}
Retain notations $\bv, \bv^0,\bw$, $x, \bv^\bot, \bw^\bot$ above.
There exist neighborhoods $\cu$, $\cu_T$, $\cu^\bot$ of $x$ in $\cm_0(\bv,\bw,\mcr^\imath)$ and the origins in $T$, $\cm_0(\bv^\bot,\bw^\bot,\mcr^\imath)$, respectively, and biholomorphic maps $\cu\rightarrow \cu_T\times \cu^\bot$, $\pi^{-1}(\cu)\rightarrow \cu_T\times \pi^{-1}(\cu^\bot)$, such that the following diagram commutes:
\begin{align}	
	\label{diag:Translice}
	\xymatrix{ \cm(\bv,\bw,\mcr^\imath) \supset \pi^{-1}(\cu) \ar[rr]^{\cong\qquad\quad} \ar@<8ex>[d]_\pi & &\cu_T\times \pi^{-1}(\cu^\bot) \subset T\times \cm(\bv^\bot,\bw^\bot,\mcr^\imath) \ar@<-14ex>[d]^{1\times\pi} \\
		\cm_0(\bv,\bw,\mcr^\imath)\supset \quad\cu\quad \ar[rr]^{\cong\qquad\quad} && \;\;\cu_T\times \cu^\bot\quad\;\; \subset T\times \cm_0(\bv^\bot,\bw^\bot,\mcr^\imath)   }
\end{align}
	In particular, the fibre $\pi^{-1}(x)$ is biholomorphic to the fibre $\pi^{-1}(0)$ over the origin of $\pi:\cm(\bv^\bot,\bw^\bot,\mcr^\imath) \rightarrow \cm_0(\bv^\bot,\bw^\bot,\mcr^\imath)$. 
\end{proposition}

\begin{proof}
We identify the NKS quiver varieties associated to the NKS categories $\mcr$ and $\cs$ in Definition~\ref{def:RS for QG} with Nakajima's cyclic quiver varieties. Fix $x\in \cm_0^{\text{reg}}(\bar{\bv}^0,\bar{\bw},\mcr)\subset \cm_0(\bar{\bv},\bar{\bw},\mcr)$. Let $\ov{T}$ be the tangent space of $\cm_0^{\text{reg}}(\bar{\bv},\bar{\bw},\mcr)$ at $x$. 
%\red{In \cite[\S3.3]{Na01}, $T$ denotes the tangent space $T_x(\mathfrak M_0)_{(G)}$ of the stratum containing $x$? }
The dimension vectors $\bar{\bw}^\bot$, $\bar{\bv}^\bot$ are defined similarly. 
According to  \cite[Theorem~ 3.3.2]{Na01}, there exist neighborhoods $\ov{\cu}$, $\ov{\cu}_{\ov{T}}$, $\ov{\cu}^\bot$ of $x$ in $\cm_0(\bar{\bv},\bar{\bw},\mcr)$ and the origins in $\ov{T}$, $\cm_0(\bar{\bv}^\bot,\bar{\bw}^\bot,\mcr)$, respectively, and biholomorphic maps $\ov{\cu}\rightarrow \ov{\cu}_{\ov{T}}\times \ov{\cu}^\bot$, $\pi^{-1}(\ov{\cu})\rightarrow \ov{\cu}_{\ov{T}}\times \pi^{-1}(\ov{\cu}^\bot)$, such that the following diagram commutes:
\begin{align}	
	\label{diag:TransliceN}
	\xymatrix{ \cm(\bar{\bv},\bar{\bw},\mcr) \supset \pi^{-1}(\ov{\cu}) \ar[rr]^{\cong\qquad\quad} \ar@<8ex>[d]_\pi & &\ov{\cu}_{\ov{T}}\times \pi^{-1}(\ov{\cu}^\bot) \subset \ov{T}\times \cm(\bar{\bv}^\bot,\bar{\bw}^\bot,\mcr) \ar@<-14ex>[d]^{1\times\pi} \\
		\cm_0(\bar{\bv},\bar{\bw},\mcr)\supset \quad\ov{\cu}\quad \ar[rr]^{\cong\qquad\quad} && \;\;\ov{\cu}_{\ov{T}}\times \ov{\cu}^\bot\quad\;\; \subset \ov{T}\times \cm_0(\bar{\bv}^\bot,\bar{\bw}^\bot,\mcr)   }
\end{align}
In particular, the fibre $\pi^{-1}(x)$ is biholomorphic to the fibre $\pi^{-1}(0)$ over the origin for $\pi:\cm(\bar{\bv}^\bot,\bar{\bw}^\bot,\mcr) \rightarrow \cm_0(\bar{\bv}^\bot,\bar{\bw}^\bot,\mcr)$. %\red{shall we use different notations $T, \cu$ ... ?}

%We first prove for $F=\Sigma^2$. Let $h$ be the Coxeter number of the Dynkin type of the quiver $Q$. It is well known that $\Sigma^{2}=\tau^{-h}$; see \cite{Ga79}. So the NKS quiver varieties $\cs,\mcr$ are cyclic quiver varieties, and then the result follows from  \cite[Theorem 3.3.2]{Na01}.  

 Now we consider $\mcr^\imath$ and $\cs^\imath$. 
For any dimension vectors $\bv \in \N^{\mcr^\imath_0-\cs^\imath_0}$ and $\bw \in \N^{\cs^\imath_0}$, it follows from Proposition \ref{prop:MMi} that 
$\cm(\bv,\bw,\mcr^\imath)$, $\cm_0(\bv,\bw,\mcr^\imath)$ and $\cm(\bw,\mcr^\imath)$ are $\texttt{G}$-fixed point subvarieties of $\cm(\bv,\bw,\mcr)$, $\cm_0(\bv,\bw,\mcr)$ and $\cm(\bw,\mcr)$, respectively. The maps in \eqref{diag:TransliceN} %for the case $F=\Sigma^2$ proved above 
commute with the $\texttt{G}$-action. 

For any $x\in\cm_0^{\text{reg}}(\bv^0,\bw,\mcr^\imath)\subset \cm_0(\bv,\bw)$, let $T$ be the tangent space of $\cm_0^{\text{reg}}(\bv,\bw)$ at $x$. Recall that $\bw^\bot=\bw-\mathcal{C}_q\bv^0 $ and $\bv^\bot=\bv-\bv^0$.   Obviously, $\bar{{\mathcal C}}_q\cv^*=\cv^*\mathcal{C}_q$, and then $\cv^*(\bv^\bot)=(\cv^*\bv)^\bot$ and $\cv^*(\bw^\bot)=(\cv^*\bw)^\bot$; here the map defined in \eqref{def:Cq} for $\mcr$ has been denoted by $\bar{{\mathcal C}}_q$.  
Restricting the maps in the diagram \eqref{diag:TransliceN} to the subvarieties $\cm(\bv,\bw,\mcr^\imath)$ and $\cm(\bv^\bot,\bw^\bot,\mcr^\imath)$ of $\cm(\cv^*\bv,\cv^*\bw,\mcr)$ and $\cm(\cv^*(\bv^\bot),\cv^*(\bw^\bot),\mcr)$, respectively, we obtain the desired transversal slice result. 
\end{proof}

\begin{remark}
	\label{rem:Translicegeneral}
For any admissible pair $(F,C)$ such that \eqref{eq:isomvarieties} holds, there exist nonzero integers $m, n$ such that $F^m=\tau^n$. By considering the $n$-cyclic quiver varieties, one can identify the NKS quiver varieties associated to  $(F,C)$ as fixed point subvarieties of Nakajima's cyclic quiver varieties. Therefore the analogous transversal slice result in this generality follows by the same type of arguments as for Proposition \ref{prop:Translice1}.
\end{remark}

\subsection{Quantum Grothendieck rings for $\imath$quivers}
\label{subsec: graded Groth ring}

In this subsection, we always assume that $\cs,\mcr$ are the NKS categories defined in Definition \ref{def:RS for QG} and Definition \ref{def:RS for iQG}. By Lemma \ref{lem:M=rep}, we know that 
$\cm_0(\bw)\cong \rep(\bw,\cs)$ for any dimension vector $\bw$, and identify them below.

Let $X$ be an algebraic variety. We denote by $\cd_c(X):=\cd_c^b(X)$ the bounded derived category of constructible sheaves on $X$. Let $f:X\rightarrow Y$ be a morphism of algebraic varieties. There are induced functors $f^*:\cd_c(Y)\rightarrow \cd_c(X)$, $f_*:\cd_c(X)\rightarrow \cd_c(Y)$ and $f_!:\cd_c(X)\rightarrow \cd_c(Y)$ (direct image with compact support). 

We review the quantum Grothendieck ring and its convolution product, following \cite{VV}; also see \cite{Na01,Na04, Qin}. 

For any two dimension vectors $\alpha,\beta$ of $\cs$, let $V_{\alpha+\beta}$ be a vector space of graded dimension $\alpha+\beta$. Fix a vector subspace $W_0\subset V_{\alpha+\beta}$ of graded dimension $\alpha$, and let
\[
F_{\alpha,\beta}:=\{y\in \rep(\alpha+\beta,\cs)\mid y(W_0)\subset W_0\}
\]
be the closed subset of $\rep(\alpha+\beta,\cs)$. 
Then $y \in F_{\alpha,\beta}$ induces a natural linear map $y': V/W_0 \rightarrow V/W_0$, i.e., $y' \in \rep(\beta,\cs)$.
Hence we obtain the following convolution diagram
\[
\rep(\alpha,\cs)\times \rep(\beta,\cs) \stackrel{p}{\longleftarrow} F_{\alpha,\beta}\stackrel{q}{\longrightarrow}\rep(\alpha+\beta,\cs),
\]
where $p(y):=(y|_{W_0},y')$ and $q$ is the natural closed embedding.

Let $\cd_c(\rep(\alpha,\cs))$ be the derived category of constructible sheaves on $\rep(\alpha,\cs)$. We have the following restriction functor (called comultiplication),
\begin{align}
	\label{eq:Delta1}
	\widetilde{\Delta}^{\alpha+\beta}_{\alpha,\beta}: \cd_c \big(\rep(\alpha+\beta,\cs) \big) \longrightarrow \cd_c \big(\rep(\alpha,\cs) \big)\times \cd_c \big(\rep(\beta,\cs) \big), \quad F\mapsto p_!q^*(F).
\end{align}

For any $(\bv,\bw)$ such that $\cm(\bv,\bw)\neq \emptyset$, we do not know if  $\cm(\bv,\bw)$ is connected or not, however, it is pure dimensional by \cite[Theorem 3.2]{Sch18}. Choose a set $\{\alpha_\bv\}$ such that it parameterizes the connected components 
of $\cm_0(\bv,\bw)$. For any $l$-dominant pair $(\bv,\bw)$, since the restriction of $\pi$ on the regular stratum $\cm_0^{\text{reg}}(\bv,\bw)$ is a homeomorphism by \cite[Definition 3.6]{Sch18}, the set $\{\alpha_\bv\}$  naturally parameterizes the connected components of this regular stratum:
\begin{align}
	\cm_0^{\text{reg}}(\bv,\bw)=\bigsqcup_{\alpha_\bv} \cm_0^{\text{reg};\alpha_\bv}(\bv,\bw).
\end{align}

Let $\underline{k}_{\cm(\bv,\bw)}$ be the constant sheaf on $\cm(\bv,\bw)$. Denote by $\pi^{\cs}(\bv,\bw)\in \cd_c(\rep(\bw,\cs))$ (or $\pi(\bv,\bw)$ when there is no confusion) the pushforward along $\pi: \cm(\bv,\bw)\rightarrow \cm_0(\bw)\cong \rep(\bw,\cs)$ of $\underline{k}_{\cm(\bv,\bw)}$ with a grading shift:
\begin{align}
  \label{eq:piPS}
	\pi(\bv,\bw):= \pi_!(\underline{k}_{\cm(\bv,\bw)}) [\dim \cm(\bv,\bw)].
\end{align}

For a strongly $l$-dominant pair $(\bv,\bw)$, let $\cl^{\cs}(\bv,\bw)$ (or $\cl(\bv,\bw)$) be the intersection cohomology complex associated to the stratum $\cm_0^{\text{reg}}(\bv,\bw)\subseteq \cm_0(\bw)$ with respect to the trivial local system, that is,
\begin{align}
 \label{eq:decomp}
	\cl(\bv,\bw)=IC(\cm_0^{\text{reg}}(v,w))=\bigoplus_{\alpha_\bv}IC(\cm_0^{\text{reg};\alpha_\bv}(\bv,\bw)).
\end{align}
%Thanks to the transverse slice theorem \cite{Na01, Na04}  (also see \cite{Qin}), 
We expect the $\imath$NKS quiver varieties always satisfy the following property (the main point is that only IC sheaves with trivial local systems can show up in the decomposition). 

\begin{hypothesis}
	\label{hypothesis}
We have a decomposition
\begin{equation}
	\label{eqn:decomposition theorem}
	\pi(\bv,\bw)=\sum_{\bv':\sigma^*\bw-{\mathcal C}_q\bv'\geq0,\bv'\leq \bv} a_{\bv,\bv';\bw}(v)\cl(\bv',\bw),
\end{equation}
where we denote by $\cf^{\oplus m}[d]$ by $m\tt^d\cf$ using an indeterminate $\tt$, for any sheaf $\cf$, $m\in\N$, and $d\in\Z$. 
Moreover, we have $a_{\bv,\bv';\bw}(\tt)\in\N[\tt, \tt^{-1}]$, $a_{\bv,\bv';\bw}(\tt^{-1})=a_{\bv,\bv';\bw}(\tt)$, and $a_{\bv, \bv;\bw}(\tt) =1$. (Any $f(\tt)\in \N[\tt, \tt^{-1}]$ such that $f(\tt^{-1})=f(\tt)$ is called \emph{bar-invariant}.)
\end{hypothesis}

%\begin{proof}
%	Using the transversal slice theorem (see Lemma \ref{lem:Translice1} and the last paragraph of \S\ref{subsec: covering}),
%\red{TBA}	
%\end{proof}

In fact, a class of $\imath$NKS quiver varieties are cyclic quiver varieties defined in \cite[\S4]{Na04} (which are listed in Table \ref{table:values}),
%; see, e.g., \cite[Proposition 3.3.2]{XZ05} for the automorphism groups of the repetition quiver $\Z Q$ of type ADE),
 so the decomposition \eqref{eqn:decomposition theorem} (i.e., Hypothesis~\ref{hypothesis}) already holds in these cases. 
%%%%%%
%	\begin{center}
\begin{table}[h]
	\caption{List of $(Q,\btau)$ whose $\imath$NKS varieties are cyclic  \qquad\qquad\qquad}
	\label{table:values}
	\begin{tabular}{| c | c | c | c | c | c |}
		\hline
		\begin{tikzpicture}[baseline=0]
			\node at (0, 0.2) {Type of $Q$};
		\end{tikzpicture}
		&
		\begin{tikzpicture}[baseline=0]
			\node at (0, 0.2) {$A_n$ ($n$ odd)};
		\end{tikzpicture}
		&
		\begin{tikzpicture}[baseline=0]
			\node at (0, 0.2) {$D_n$};
		\end{tikzpicture}
		&
		\begin{tikzpicture}[baseline=0]
			\node at (0, 0.2) {$E_6$};
		\end{tikzpicture}
		&
		\begin{tikzpicture}[baseline=0]
			\node at (0, 0.2) {$E_7$};
		\end{tikzpicture}
		&
		\begin{tikzpicture}[baseline=0]
			\node at (0, 0.2) {$E_8$};
		\end{tikzpicture}
		\\
		\hline
		%&
		$\btau$
		&
		$\neq\Id$
		&
		$=\Id$
		&
		$\neq\Id$
		&
		$=\Id$
		&
		$=\Id$
		\\
		\hline
			
		%&
		$\Sigma\widehat{\varrho}$
		&
		$\tau^{-\frac{n+1}{2}}$
		&
		$\tau^{-n}$
		&
		$\tau^{-6}$
		&
		$\tau^{-9}$
		&
		$\tau^{-15}$
		\\
		\hline
			\end{tabular}
	\newline
\end{table}

Let us comment on the calculations of $\Sigma \widehat{\varrho}$ in Table \ref{table:values}. If $\varrho=\Id$, the calculation is already given in the proof of \cite[Proposition 3.3.2]{XZ05}. In the remaining cases with types $A_n$ ($n$ odd) and $E_6$, one can calculate $\Sigma \widehat{\varrho}$ for a suitable orientation, and the formulas hold generally since the derived categories are independent of orientations.

The proof in \cite{Na01,Na04} for graded/cyclic quiver varieties uses the transversal slice theorem and the following property of standard modules of quantum affine algebras: a standard module is generated by the $l$-highest weight vector, which is indecomposable. 

Note that the transversal slice result holds for $\imath$NKS quiver varieties (see Proposition~ \ref{prop:Translice1}).  By using the same approach of \cite[\S14]{Na01} or \cite[Theorem 8.6]{Na04},  it remains to verify that  simple perverse sheaves whose shift appears in a direct summand of $\pi(\bv,\bw)$ are the intersection cohomology complexes $IC(\cm_0^{\text{reg};\alpha_\bv}(\bv,\bw))$ associated with the trivial local system.  

%\begin{remark}
%We expect such decomposition \eqref{eqn:decomposition theorem} also holds for general NKS quiver varieties (with  \eqref{eq:isomvarieties} satisfied); cf. Remark \ref{rem:Translicegeneral}.
%\end{remark}

{\bf
Throughout this paper, we shall assume that Hypothesis \ref{hypothesis} holds.
}

For each $\bw\in \N^{\cs_0}$, the Grothendieck group $K_\bw(\mod(\cs))$ is defined as the free abelian group generated by the perverse sheaves $\cl(\bv,\bw)$ appearing in (\ref{eqn:decomposition theorem}), for various $\bv$. It has two distinguished $\Z[\tt, \tt^{-1}]$-bases by \eqref{eqn:decomposition theorem}:
\begin{align}
	\label{eq:bases}
	\begin{split}
		\{\pi(\bv,\bw) \mid \sigma^*\bw-{\mathcal C}_q\bv\geq0, \cm_0^{\text{reg}}(\bv,\bw)\neq\emptyset\};
		\\
		\{\cl(\bv,\bw) \mid \sigma^*\bw-{\mathcal C}_q\bv\geq0, \cm_0^{\text{reg}}(\bv,\bw)\neq\emptyset\}.
	\end{split}
\end{align}
Consider the free $\Z[\tt, \tt^{-1}]$-module
\begin{equation}
	\label{eq:Kgr}
	K^{gr}(\mod(\cs)) := \bigoplus_\bw K_\bw(\mod(\cs)).
\end{equation}
Then $\{\widetilde{\Delta}^{\bw}_{\bw_1,\bw_2} \}$ induces a comultiplication $\widetilde{\Delta}$ on $K^{gr}(\mod(\cs))$.

Introduce a bilinear form $d(\cdot,\cdot)$ on $\N^{\mcr_0-\cs_0}$ by letting
\begin{equation}\label{definition:d}
	d\big((\bv_1,\bw_1),(\bv_2,\bw_2) \big)=(\sigma^*\bw_1-{\mathcal C}_q\bv_1)\cdot \tau^* \bv_2+\bv_1\cdot \sigma^*\bw_2,
\end{equation}
where $\cdot$ denotes the standard inner product, i.e., $\bv' \cdot \bv'' =\sum_{x\in \mcr_0-\cs_0} \bv'(x)\bv''(x)$.

Using the same proof of \cite{VV}, Scherotzke-Sibilla obtained the following result. 
\begin{proposition}[\text{\cite[Proposition 4.8]{SS16}; see also \cite[Lemma 4.1]{VV},\cite[eq. (11)]{Qin} }]
The comultiplication $\widetilde{\Delta}$ is coassociative and given by
	\begin{align}
		\label{eqn:comultiplication}
		\widetilde{\Delta}^{\bw}_{\bw_1,\bw_2} \big(\pi(\bv,\bw) \big)=\bigoplus_{\stackrel{\bv_1+\bv_2=\bv}{\bw_1+\bw_2=\bw}} \tt^{d((\bv_2,\bw_2),(\bv_1,\bw_1))-d((\bv_1,\bw_1),(\bv_2,\bw_2))}\pi(\bv_1,\bw_1)\boxtimes\pi(\bv_2,\bw_2).
	\end{align}
\end{proposition}

\begin{remark}
	The graded/cyclic quiver varieties in \cite{Na01,Na04} are NKS varieties with the configuration $C= \{\text{vertices of }\Z Q\}$; also see \cite{Qin}. The definition of $\mathcal C_q$ in \eqref{def:Cq} (which makes sense for a more flexible configuration $C$) essentially coincides with Qin's up to a shift by $\sigma$, and our bilinear form \eqref{definition:d} coincides with \cite[Eq. (12)]{Qin}. The formula \eqref{definition:d} was used implicitly though not written down explicitly in \cite{SS16}.
\end{remark}

Denote
\begin{align}
	\label{eq:Kgr2}
	\begin{split}
		{\bf R}_\bw(\mod(\cs)) &=\Hom_{\Z[\tt, \tt^{-1}]} (K_\bw(\mod(\cs)),\Z[\tt, \tt^{-1}]),
		\\
		K^{gr*}(\mod(\cs)) &=\bigoplus_{\bw\in \N^{\cs_0}}{\bf R}_\bw(\mod(\cs)).
	\end{split}
\end{align}
Then as the graded dual of the coalgebra $K^{gr}(\mod(\cs))$, $K^{gr*}(\mod(\cs) )$ becomes a $\Z[\tt, \tt^{-1}]$-algebra, whose multiplication is denoted by $\ast$.
Note that $K^{gr*}(\mod(\cs))$ is a $\N^{\cs_0}$-graded algebra (called the {\em quantum Grothendieck ring}).  It has two distinguished bases
\begin{align}
	\label{eq:bases dual}
	\begin{split}
		\{\chi (\bv,\bw) \mid \sigma^*\bw-{\mathcal C}_q \bv\geq0, \cm_0^{\text{reg}}(\bv,\bw)\neq\emptyset\},
		\\
		\{L (\bv,\bw) \mid \sigma^*\bw-{\mathcal C}_q \bv\geq0, \cm_0^{\text{reg}}(\bv,\bw)\neq\emptyset\},
	\end{split}
\end{align}
dual to the 2 bases in \eqref{eq:bases}, respectively. We shall drop the subscript $\cs$ in $\chi_\cs$ and $L_\cs$  in \eqref{eq:bases dual} when there is no confusion.
Note that $\cl(\bv,\bw)$, $L(\bv,\bw)$ make sense only for strongly $l$-dominant pairs $(\bv,\bw)$.

%For any $M\in\mod(\cs)$, let $\cl(M)$ be the Intersection cohomology complex associated to the stratum containing $M$. By definition $\cl(M)\in K^{gr}(\mod(\cs))$. Denote $L(M)\in K^{gr*}(\mod(\cs))$ the dual of $\cl(M)$.

By the transversal slice Theorem \cite{Na01},  the direct summands appearing in $\pi(\bv,\bw)$ are the shifts of sheaves $\cl(\bv',\bw)$ with $\bv'\leq \bv$; cf. \eqref{eq:leq}. So we have
\begin{equation}
	\label{equation multiplication}
	L(\bv_1,\bw_1)\ast  L(\bv_2,\bw_2)=\sum_{\bv \geq \bv_1+\bv_2} c_{\bv_1,\bv_2}^\bv(\tt) L(\bv,\bw_1+\bw_2),
\end{equation}
with a leading term $c_{\bv_1,\bv_2}^{\bv_1+\bv_2}(\tt) L(\bv_1+\bv_2,\bw_1+\bw_2)$; moreover, we have
\begin{align}
	c_{\bv_1,\bv_2}^\bv(\tt) & \in \N[\tt, \tt^{-1}],
	\label{eq:positive}
	\\
	c_{\bv_1,\bv_2}^{\bv_1+\bv_2}(\tt) &= \tt^{d((\bv_2,\bw_2),(\bv_1,\bw_1))-d((\bv_1,\bw_1),(\bv_2,\bw_2))}.
	\label{eqn: leading term}
\end{align}
The positivity \eqref{eq:positive} follows from the fact that the image of a perverse sheaf $\cl(\bv, \bw)$ under the comultiplication $\widetilde{\Delta}$ (see \eqref{eq:Delta1}) is a semisimple perverse sheaf. %\blue{It is not semisimple, semisimple means in abelian category, no shitss}

\subsection{Twisted quantum Grothendieck rings for $\imath$quivers}
  \label{subsec:iGring}

Let $(Q,\btau)$ be a Dynkin $\imath$quiver.
Let $\cs^\imath$, $\mcr^\imath$ and $\cp^\imath$ be the $\imath$NKS categories defined in Definition \ref{def:RS for iQG} and Theorem \ref{thm:iNKS}.

Consider the natural projection $p:\mcr^{\gr}\rightarrow \mcr^\imath$. For any $z\in\mcr^\imath_0-\cs^\imath_0$, there exists a unique vertex, denoted again by $z$ by abuse of notation, in $\mcr^{\gr}_0-\cs^{\gr}_0$ corresponding to a stalk complex concentrated at degree zero (in $\mod(kQ)$) whose image under $p$ is the $z$ which we start with; see \cite[Theorem 3.18, Corollary 3.21]{LW19}. In other words, there exists a bijection between the set of vertices in $\cp^\imath$ and the set of indecomposable $kQ$-modules.

For any $i\in Q_0$, we denote
\begin{align}
\label{eq:vfi}
\begin{split}
\bw^{i}&=\e_{\sigma \ts_i}+\e_{\sigma\ts_{\btau i}},
\\
\bv^{i}&= \sum_{z\in\mcr_0-\cs_0} \dim\cp^\imath(\ts_i ,z)\e_z.
\end{split}
\end{align}
Note that $\bw^{i}=\bw^{{\btau i}}$. Recall $\e_x$ denotes the characteristic function of $x\in\mcr_0^\imath$, which is also viewed as the unit vector supported at $x$. Denote
\begin{align}
 \label{eq:VW+}
W^{+}&=\bigoplus_{x\in\{\ts_i,i\in Q_0\}}\N \e_{\sigma x},&&
V^{+}=\bigoplus_{
x\in\Ind \mod(kQ),\, x\text{ is not injective}} \N \e_x,\\
W^0&=\bigoplus_{i\in Q_0} \N \bw^{i},&&
V^0=\bigoplus_{i\in Q_0} \N \bv^{i}.
\notag
\end{align}

Recall the restriction functor $\res:\mod(\cs^\imath) \rightarrow \mod(kQ)$. It is natural to identify the Grothendieck groups $K_0(\mod(\cs^\imath))$ and $K_0(\mod(kQ))$. Recall the Euler form $\langle\cdot,\cdot\rangle$ of $kQ$. We define the bilinear forms $\langle\cdot,\cdot\rangle_a$ and $(\cdot, \cdot)$ as follows: for any dimension vectors $\bw,\bw' \in \N^{\cs^\imath_0}$, let
\begin{align}
\langle \bw,\bw'\rangle_a&= \langle \bw,\bw'\rangle-\langle \bw',\bw\rangle,\label{eqn: antisymmetric bilinear form}\\
(\bw,\bw')&=\langle \bw,\bw'\rangle+\langle \bw',\bw\rangle.\label{eqn: symmetric bilinear form}
\end{align}
For any pairs $(\bv,\bw)$ and $(\bv',\bw')$ with dimension vectors $\bw, \bw' \in \N^{\cs^\imath_0}$, we define
\begin{align*}
\langle (\bv,\bw),(\bv',\bw')\rangle_a:=\langle \bw,\bw'\rangle_a,\quad ((\bv,\bw),(\bv',\bw')):=( \bw,\bw').
\end{align*}

Let us fix a square root $\tt^{1/2}$ of $\tt$ once for all. As we need a twisting involving $\tt^{1/2}$ shortly, we shall consider the ring $\Z[\tt^{1/2}, \tt^{-1/2}]$ and the field $\Q(\tt^{1/2})$.

The coalgebra $K^{gr}(\mod(\cs^\imath))$ (cf. \eqref{eq:Kgr}) and its graded dual (cf. \eqref{eq:Kgr2}) up to a base change here in the current setting read as follows:
\begin{align}
K^{gr}(\mod(\cs^\imath)) &=\bigoplus_{\bw\in W^{+}} K_\bw(\mod(\cs^\imath)),
  \label{eq:KSi}
\\
\tRiZ=\bigoplus_{\bw\in W^{+}} \tR^{\imath}_{\Z,\bw},
\quad \text{where } \tR^{\imath}_{\Z,\bw}  &=   \Hom_{\Z[\tt, \tt^{-1}]} \big(K_\bw(\mod(\cs^\imath)),\Z[\tt^{1/2}, \tt^{-1/2}]\big).
\label{eq:tRi}
\end{align}

As by definition there is a canonical bijection $W^+\leftrightarrow \N \I$,
we have two $\N \I$-graded algebra structures on the $\Z[\tt^{1/2}, \tt^{-1/2}]$-module $\tRiZ$ as follows. The second one $(\tRiZ, \cdot)$, called the {\em quantum Grothendieck ring} associated to $\imath$quiver $(Q, \btau)$, is what we are really interested in.

$\triangleright$ $(\tRiZ, \ast )$ is the $\Z[\tt^{1/2}, \tt^{-1/2}]$-algebra corresponding to the coalgebra $K^{gr}(\mod(\cs^\imath))$ with the comultiplication $\{\widetilde{\Res}^{\bw}_{\bw^1,\bw^2} \}$, cf. \eqref{eqn:comultiplication}.

$\triangleright$ $(\tRiZ, \cdot)$ is the $\Z[\tt^{1/2}, \tt^{-1/2}]$-algebra corresponding to the coalgebra $K^{gr}(\mod(\cs^\imath))$ with the {\em twisted} comultiplication
\begin{equation}
  \label{eq:tw}
\{\Res^{\bw}_{\bw^1,\bw^2} :=\tt^{- \frac12\langle \bw^1,\bw^2\rangle_{a}} \widetilde{\Res}^{\bw}_{\bw^1,\bw^2}\};
\end{equation}
in practice, the product sign $\cdot$ is often omitted.

We shall need the $\Q(\tt^{1/2})$-algebra obtained by a base change below:
\begin{align}
 \label{eq:tRiQZ}
  \tRi  =\Q(\tt^{1/2}) \otimes_{\Z[\tt^{1/2}, \tt^{-1/2}]} \tRiZ.
\end{align}

%For any $\bw=\bw_1+\bw_2$, $\widetilde{\Delta}^\bw_{\bw_1,\bw_2}$ is a homomorphism from $(K_\bw)_{t^{1/2}}$ to $(K_\bw)_{t^{1/2}}\otimes_{\Z[t^{\pm (1/2)}]} (K_\bw)_{t^{1/2}}$.
%The twisted comultiplication  $\Delta^\bw_{\bw_1,\bw_2}$ of $K_{t^{1/2}}$ is defined to be $\widetilde{\Delta}^\bw_{\bw_1,\bw_2} t^{-(1/2)\langle \bw_1,\bw_2\rangle_a}$. Correspondingly, we have a twisted multiplication on $R_{t^{1/2}}$, which is denoted by $\cdot$.

%%%%%%%
%%%%%%%
\section{The $l$-dominant pairs}
   \label{sec:pairs}

In this section, we shall determine various strongly $l$-dominant pairs for $\mcr^\imath$ (cf. Definition~\ref{def:pair}); see Propositions~\ref{prop:dominant pairs 1}, \ref{prop:dominant pairs for gij}, \ref{prop:witauij} and \ref{prop:dominant for gij}. The results in this section will play a fundamental role in Section~\ref{sec:computation}.

We shall work with the $\imath$NKS categories $\cs^\imath$, $\mcr^\imath$ and $\cp^\imath$ defined in Definition \ref{def:RS for iQG} and Theorem \ref{thm:iNKS}, associated to a Dynkin $\imath$quiver $(Q,\btau)$. Recall that the vertices in $\cp^\imath$ are labeled by the indecomposable $kQ$-modules.
From now on, we shall denote $\cm_0(\bw)=\cm_0(\bw,\mcr^\imath)$, and $\cm_0(\bv,\bw)=\cm_0(\bv,\bw,\mcr^\imath)$.
\subsection{{Strongly $l$-dominant pairs $(\bv,\bw^{{i}})$}}

%For the sake of notational simplicity, the simple $\mcr^\imath$-module $\ts_{\ts_i}$ corresponding to the vertex $\ts_i$ is denoted by $\ts_i$, and the simple $\mcr^\imath$-module corresponding to the vertex $\sigma\ts_i$ is also denoted by $\red{\sigma}\ts_i$. \red{Do we use them somewhere?}

%Since the matrix ${\mathcal C}_q$ defined in \eqref{def:Cq} is not injective in this case, we do not have the bijection between strata and orbits as described in \cite[Theorem 3.14]{LeP13} for graded NKS categories with $C=\{\text{vertices of }\Z Q\}$.

Recall there is a map from  the set of isomorphism classes of representations in $\rep(\bw,\cs^\imath)$ to the set of strongly $l$-dominant pairs $(\bv,\bw)$, which sends $M\in \rep(\bw,\cs^\imath)$ to the dimension vector of $K_{LR}(M)$, which is surjective by Lemma~ \ref{lem:bistable}.

As $\tau$ is the AR functor of the triangulated category $\cd_Q/\Sigma \widehat{\btau}$ and $\cp^\imath\simeq \Ind \cd_Q/\Sigma \widehat{\btau}$, we have the following natural isomorphism for any $X,Y\in\cp^\imath$
\begin{align*}
%\label{eqn: serre duality}
\eta_{X,Y}:\Hom(X,Y )\stackrel{\sim}\longrightarrow D\Hom(Y,\Sigma \tau X)=D\Hom(Y, \tau \widehat{\btau} X).
\end{align*}
%the natural isomorphism

Recall $\bv^i, \bw^i$ from \eqref{eq:vfi} and $\mathcal C_q$ from \eqref{def:Cq}.

\begin{lemma}
    \label{lem:dimCartan}
For any $i\in Q_0$, the following holds:
\begin{itemize}
\item[(i)] $\bv^{i}(\ts_i)=1$, $\bv^{i}(\tau\ts_{\btau i})=1$;
\item[(ii)] $\bv^{i}(\tau \ts_i)=0$ if $\btau i\neq i$;
\item[(iii)]
If there exists $j\in Q_0$ such that $\Ext^1_{kQ}(\ts_j,\ts_i)=k$, then $\bv^{i}(\ts_j)=0$, $\bv^{i}(\tau \ts_j)=1$, $\bv^{j}(\tau \ts_i)=0$, and
\[
\bv^{j}(\ts_i)=\left\{ \begin{array}{ccc} 1& \text{ if }\btau j=j \text{ or }\btau i=i,\\
0&\text{ otherwise.} \end{array} \right.
\]
\item[(iv)]
$\sigma^*\bw^{i}-{\mathcal C}_q \bv^{i}$ vanishes. In particular, $(\bv^{i},\bw^{i})$ is an $l$-dominant pair.
\end{itemize}
\end{lemma}

\begin{proof}
(i) It is clear that $\bv^{i}(\ts_i)=1$. Since $kQ$ is hereditary, we have
\begin{align*}
\bv^{i}(\tau\ts_{\btau i})
 & = \dim \Hom_{\cp^\imath}(\ts_i ,\tau\ts_{\btau i})
 \\
&= \dim \Hom_{\cd_Q} \big( \ts_i, \oplus_{m\in\Z}(\Sigma\widehat{\btau})^m \tau\ts_{\btau i} \big)
\\
&= \dim  \Hom_{\cd_Q}( \ts_i, \tau \ts_{\btau i})+\dim  \Hom_{\cd_Q}(\ts_i,\Sigma\tau \ts_i).
\end{align*}
Note also that $ \Hom_{\cd_Q}( \ts_i, \tau \ts_{\btau i})\cong \Ext^1_{kQ}(\ts_{\btau i},\ts_i)=0$, and $ \Hom_{\cd_Q}(\ts_i,\Sigma\tau \ts_i)=D( \Hom_{\cd_Q}(\ts_i,\ts_i))\cong k$.
Then our desired formula follows.

(ii)
As $\btau i\neq i$, we have
\begin{align*}
\bv^{i}(\tau \ts_i)
& =\dim\Hom_{\cp^\imath}(\ts_i,\tau \ts_i)
\\
& =\dim\Hom_{\cd_Q}(\ts_i,\tau \ts_i)+\dim \Hom_{\cd_Q}(\ts_i,\Sigma\tau \ts_{\btau i})
\\
& = \dim \Hom_{\cd_Q}(\ts_i,\Sigma\tau \ts_{\btau i})=\dim \Hom_{kQ}(\ts_{\btau i},\ts_i)=0.
\end{align*}

(iii)
Note that $\bv^{i}(\ts_j)=\dim\Hom_{\cp^\imath}(\ts_i,\ts_j)= \dim\Hom_{\cd_Q}(\ts_i,\Sigma \ts_{\btau j})= \dim\Ext^1_{kQ}(\ts_i,\ts_{\btau j})$. If $\Ext^1_{kQ}(\ts_i,\ts_{\btau j})\neq0$, then there exists an arrow $\beta:\btau j\rightarrow i$ in $Q$. As $\Ext^1_{kQ}(\ts_j,\ts_i)=k$, there exists an arrow $\alpha:i\rightarrow j$. Thus there exists an oriented cycle
$i\xrightarrow{\alpha} j\xrightarrow{\btau \beta}\btau i\xrightarrow{\btau\alpha}\btau j \xrightarrow{\beta}i$
in $Q$, which is a contradiction. So $\bv^{i}(\ts_j)=0$.

Also, $\bv^{i}(\tau \ts_j)=\dim\Hom_{\cp^\imath}(\ts_i,\tau \ts_j)=\dim \Ext^1_{kQ}(\ts_j,\ts_i)=1$. Next, we have
\[
\bv^{j}(\tau \ts_i)=\dim\Hom_{\cp^\imath}(\ts_j,\tau \ts_i)=\dim \Ext^1_{kQ}(\ts_i,\ts_j)+\dim \Hom_{kQ}(\ts_{\btau i},\ts_j)=0,
\]
thanks to
$\Ext^1_{kQ}(\ts_j,\ts_i)\neq0$ and $\btau i\neq j$.

Finally, we have
\begin{align*}
\bv^{j}(\ts_i)
=\dim\Hom_{\cp^\imath}(\ts_j,\ts_i) &=\dim \Ext^1_{kQ}(\ts_j,\ts_{\btau i})
=\left\{ \begin{array}{ccc} 1& \text{ if }\btau j=j \text{ or }\btau i=i,\\
0&\text{ otherwise.} \end{array} \right.
\end{align*}
Indeed, if $\Ext^1_{kQ}(\ts_j,\ts_{\btau i})\neq 0$ in case $\btau(j)\neq j$ and $\btau i\neq i$, then we would obtain a cycle (not oriented) involving $\btau i, j, i, \btau j$, which is a contradiction.

(iv) The argument here is based on \cite[Lemma 4.1.2]{Qin}.
For any indecomposable $kQ$-module $M$, we have an AR triangle in $\cd_Q$ (and hence in the triangulated orbit category $\cp^\imath\simeq \cd_Q/\Sigma \widehat{\btau}$) of the form
\[
\tau M\longrightarrow E\longrightarrow M\longrightarrow \Sigma \tau M.
\]
Applying the functor $\Hom_{\cp^\imath}(\ts_i,-)$ to this triangle, we obtain a long exact sequence
\begin{align*}
\Hom_{\cp^\imath}(\ts_i,\Sigma^{-1}E)
& \longrightarrow\Hom_{\cp^\imath}(\ts_i,\Sigma^{-1}M) \stackrel{\mu'}{\longrightarrow} \Hom_{\cp^\imath}(\ts_i,\tau M)
\\
 \longrightarrow \Hom_{\cp^\imath}(\ts_i,E)
 & \longrightarrow \Hom_{\cp^\imath}(\ts_i,M) \stackrel{\mu''}{\longrightarrow}\Hom_{\cp^\imath}(\ts_i,\Sigma\tau M).
\end{align*}

If $M\neq \ts_i$ or $\ts_{\btau i}$, we have $\mu'=0=\mu''$ by the universal property of AR triangles, and consequently,
$({\mathcal C}_q \bv^{i})(M)= \dim \Hom_{\cp^\imath}(\ts_i,\tau M)-\dim \Hom_{\cp^\imath}(\ts_i, E)+ \dim\Hom_{\cp^\imath}(\ts_i,M)=0$.

The cases when $M=\ts_i$ or $\ts_{\btau i}$ are similar, and we only give the detail for $M=\ts_i$. We proceed by separating into two subcases (i)--(ii) below.
(i) If $\btau i=i$, we have $\ker \mu'=0=\ker \mu''$, and then
$({\mathcal C}_q \bv^{i})(\ts_i)= 2\dim\Hom_{\cp^\imath}(\ts_i,\ts_i)=2$; also note $\sigma^*\bw^{i} (\ts_i)=2$ in this case.
(ii) If $\btau i\neq i$, we have $\mu'=0$, and $\ker \mu''=0$, and then
$({\mathcal C}_q \bv^{i})(\ts_i)= \dim\Hom_{\cp^\imath}(\ts_i,\ts_i)=1$; also note $\sigma^*\bw^{i} (\ts_i)=1$.
\end{proof}

For any $i\in Q_0$, there exists an arrow $\varepsilon_i:i\rightarrow \btau i$ in the quiver $\ov{Q}$ of $\Lambda^\imath$; cf. \eqref{eqn:i-quiver}. Recall that $\E_i$ denotes the generalized simple $\Lambda^\imath$-module with composition factors $\ts_i,\ts_{\btau i}$. Then $\E_i$ can be viewed as a $\cs^\imath$-module. Recall $K_{LR}(\cdot) , K_L(\cdot)$ and $K_R(\cdot)$ from \S\ref{subsec: strat}.

\begin{proposition}\label{prop:dominant pairs 1}
We have $K_{LR}(\E_i)=K_L(\E_i)=K_R(\E_i)$, and $\dimv K_{LR}(\E_i)=(\bv^{i},\bw^{i})$, for $i\in Q_0$. Moreover, the strongly $l$-dominant pairs $(\bv,\bw^{i})$ are exactly the pairs with $\bv\in \{ {\bf 0}, \bv^{i}, \bv^{{\btau (i)}} \}$.
\end{proposition}

\begin{proof}
Denote by $\bar{P}_i$ the indecomposable $\cs^\imath$-module corresponding to $\sigma \ts_i$. It follows by the proof of Lemma \ref{lem:embedding of varieties} that $\mcr$ and $\mcr^\imath$ viewed as algebras are finite-dimensional. Hence each projective $\mcr^\imath$-module $\Hom_{\mcr^\imath}(-,\sigma c)$ is bistable for any $c\in C$ by \cite[Lemma 4.2(1)]{SS16}; moreover, by the proof {\em loc. cit.} we have $K_{LR}(\bar{P}_i)=K_L(\bar{P}_i)=K_R(\bar{P}_i)=\Hom_{\mcr^\imath}(-,\sigma \ts_i)$ for any $i\in\I$.
Following \cite[Lemma 3.7, (3.4)]{LW19}, we let
\begin{align}
\label{eqn:proj resolution E}
0\longrightarrow \bigoplus_{(\alpha: j\rightarrow i)\in Q_1} \bar{P}_j\longrightarrow \bar{P}_i\longrightarrow \E_i \longrightarrow0
\end{align}
be the minimal projective resolution of $\E_i$.  By applying the functors $K_L$ and $K_R$ to \eqref{eqn:proj resolution E}, we obtain the following exact sequences
\begin{align*}
\bigoplus_{(\alpha: j\rightarrow i)\in Q_1} K_L(\bar{P}_j)\longrightarrow K_L(\bar{P}_i)\longrightarrow K_L(\E_i) \longrightarrow0,\\
0\rightarrow\bigoplus_{(\alpha: j\rightarrow i)\in Q_1} K_R(\bar{P}_j)\longrightarrow K_R(\bar{P}_i)\longrightarrow K_R(\E_i).
\end{align*}
Since $K_{LR}(\bar{P}_i)=K_L(\bar{P}_i)=K_R(\bar{P}_i)$ for any $i\in\I$, using the canonical map $K_L\rightarrow K_R$, these two sequences are short exact and coincide with each other. Therefore, we have established  $K_L(\E_i)=K_R(\E_i)=K_{LR}(\E_i)$.

Recall the following exact sequence:
\[
0\longrightarrow \ts_i\stackrel{f}{\longrightarrow} \E_i \stackrel{g}{\longrightarrow} \ts_{\varrho(i)} \longrightarrow0.
\]
By definition of $\mcr^\imath$, there is only one arrow $\tau \ts_{i}\xrightarrow{\beta} \sigma \ts_{i}$ ending at $\sigma \ts_{i}$; moreover, we have $\beta p\neq0$ for any nonzero path $p$ ending at $\tau\ts_{i}$ in $\mcr^\imath$.
%Denote by $\ts_i^\vee:=D\Hom_{\cp^\imath}(\ts_i,-)$ the indecomposable injective $\cp^\imath$-module corresponding to the vertex $\ts_i$, and
Denote by
\[
(\tau\ts_i)^\wedge:=\Hom_{\cp^\imath}(-,\tau\ts_{i})
\]
the indecomposable projective $\cp^\imath$-module corresponding to $\tau\ts_{i}$.
%In particular, \eqref{eqn: serre duality} shows that $\ts_i^\vee\cong \tau(\ts_{\varrho(i)})^\wedge$.
We can view $(\tau\ts_i)^\wedge$ as $\mcr^\imath$-module by the projection $\mcr^\imath\rightarrow \cp^\imath$. Then $\dimv (\tau\ts_i)^\wedge =(\bv^{\varrho(i)},0)$ by definition since $\Hom_{\cp^\imath}(-,\tau\ts_{i}) \cong D\Hom_{\cp^\imath}(\ts_{\varrho(i)},-)$. The arrow $\beta$ induces
the following exact sequence in $\mod(\mcr^\imath)$:
\begin{align*}
%&0\rightarrow \ts_{\sigma(\ts_i)}\rightarrow K_R(\ts_i)\rightarrow \ts_i^\vee \rightarrow0,\\
0\longrightarrow (\tau\ts_i)^\wedge \stackrel{\beta}{\longrightarrow} K_L(\ts_{i}) \stackrel{\alpha}{\longrightarrow} \ts_{\sigma\ts_{i}} \longrightarrow0.
\end{align*}
Hence, we obtain $\dimv K_L(\ts_{i})= (\bv^{\varrho(i)},\e_{\sigma \ts_{i}})$. Similarly, $\dimv K_L(\ts_{\varrho(i)})=(\bv^i,\e_{\sigma\ts_{\varrho(i)} })$.

%Similarly, $\dimv K_L(\ts_i)=( \bv^{\varrho(i)},\e_{\sigma(\ts_i)} )$ and $\dimv K_R(\ts_{\varrho(i)})=(\bv^{\varrho(i)},\e_{\sigma(\ts_{\varrho(i)})})$.
By applying the functor $K_L$, we have the following exact sequence
\begin{align*}
K_L(\ts_i)\xrightarrow{K_L(f)} K_L(\E_i) \xrightarrow{K_L(g)} K_L(\ts_{\varrho(i)}) \longrightarrow0.
\end{align*}
Since $K_L(\E_i)$ is bistable, we have $K_L(f)\circ \beta=0$ by noting that $(\tau\ts_i)^\wedge$ is a $\cp^\imath$-module. Then $K_L(f)$ factors through $\alpha$, i.e., $K_L(f)=h\circ \alpha$ for some $h:\ts_{\sigma\ts_i} \rightarrow K_L(\E_i)$. Since $\ts_{\sigma\ts_i}$ is simple, $h$ is injective. So we obtain a short exact sequence
\[
0\longrightarrow\ts_{\sigma\ts_i} \stackrel{h}{\longrightarrow} K_L(\E_i)\xrightarrow{K_L(g)} K_L(\ts_{\varrho(i)}) \longrightarrow0.
\]
It follows that $\dimv K_L(\E_i)=\dimv \ts_{\sigma\ts_i}+ \dimv K_L(\ts_{\varrho(i)})=(0,\e_{\sigma \ts_i})+(\bv^i,\e_{\sigma \ts_{\varrho(i)}})=(\bv^i,\bw^i)$.
Then $\dimv K_{LR}(\E_i)=\dimv K_L(\E_i)=(\bv^i,\bw^i)$.

For the last assertion, we shall assume $\btau i\neq i$ (and the other case is simpler). Since
$K_{LR}$ preserves simple modules, the dimension vector of $K_{LR}(\ts_i\oplus \ts_{\btau (i)})$ is $(0,\bw^{i})$.
There are only three indecomposable $\Lambda^\imath$-modules (up to isomorphisms) with composition factors $\ts_i,\ts_{\btau (i)}$: $\ts_i\oplus \ts_{\btau (i)}$, $\E_i$, and $\E_{\btau (i)}$.
It follows from Lemma \ref{lem:dimCartan} that there are only three strongly $l$-dominant pairs $(\bv,\bw^{i})$ with $\bv \in \{ {\bf 0}, \bv^{i}, \bv^{{\btau (i)}} \}$. In particular, $\dimv K_{LR}(\E_{\btau (i)})= (\bv^{{\btau (i)}},\bw^{i})$.
\end{proof}

Let $\cp^{<\infty}(\cs^\imath)$ be the subcategory consisting of $\cs^\imath$-modules with finite projective dimension. Note that $\cp^{<\infty}(\cs^\imath)$ is equivalent to $\cp^{<\infty}(\Lambda^\imath)$, and we shall identify them below. Then the Grothendieck group $K_0(\cp^{<\infty}(\cs^\imath))$ is freely generated by $\widehat{\E}_i$ ($i\in Q_0$), see \cite[Corollary~ 3.9]{LW19}.

\begin{corollary}
  \label{cor:KLR for finite projective dimension}
For any $\cs^\imath$-module $M$ with finite projective dimension, we have $K_{L}(M)=K_{LR}(M)=K_R(M)$. In particular, if the class of $M$ in $K_0(\cp^{<\infty}(\cs^\imath))$ is
$\sum_{i\in \I}a_i \widehat{\E}_i$, then $\dimv K_{LR}(M)=(\sum_{i\in \I} a_i\bv^{i},\sum_{i\in \I}a_i\bw^{i})$.
\end{corollary}

\begin{proof}
The second assertion follows from the first one and Proposition~\ref{prop:dominant pairs 1}.

For any $\cs^\imath$-module $M$ with finite projective dimension, by \cite[Corollary 3.9]{LW19}, it has a filtration $0=M_t\subset M_{t-1}\subset \cdots \subset M_{1}\subseteq M_0=M$ such that $M_j/M_{j+1}\cong \E_{i_j}$ for some $i_j\in Q_0$ with $0\leq j\leq t-1$.
We shall prove the first assertion by induction on the length of the filtration of $M$. The base case when $M=\E_i$ follows from Proposition~ \ref{prop:dominant pairs 1}. For a general $M$, there exists a short exact sequence
\[
0\longrightarrow \E_j\stackrel{f}{\longrightarrow} M\stackrel{g}{\longrightarrow} N\longrightarrow0,
\]
where $N\in \cp^{<\infty}(\cs^\imath)$ has a shorter filtration length. By inductive assumption on $N$, we have the following commutative diagram where the first and third rows are exact:
\[
\xymatrix{ &K_L(\E_j) \ar[r]^{K_L(f)} \ar[d]^{\cong} &K_L(M) \ar[r]^{K_L(g)}\ar@{->>}[d]^{h_1} & K_L(N)\ar[r]\ar[d]^{\cong} &0\\
&K_{LR}(\E_j) \ar[r]^{K_{LR}(f)}\ar[d]^{\cong} & K_{LR}(M) \ar[r]^{K_{LR}(g)} \ar@{>->}[d]^{h_2} &K_{LR}(N) \ar[d]^{\cong}& \\
0\ar[r]& K_R(\E_j) \ar[r]^{K_R(f)}& K_R(M) \ar[r]^{K_R(g)} & K_R(N) & }
\]
Then $K_{LR}(f)$ is injective and $K_{LR}(g)$ is surjective. It follows that $K_L(f)$ is injective and $K_R(g)$ is surjective.
So the sequences in the first and third rows become short exact. Then the Five Lemma shows that $h_2h_1$ is an isomorphism and $K_{L}(M)=K_{LR}(M)=K_R(M)$.
\end{proof}

\subsection{Strongly $l$-dominant pairs $(\bv,\bw^{{ij}})$}

If $\Ext^1_{kQ}(\ts_j,\ts_i)=k$ for $i,j\in Q_0$, i.e., there exists an arrow $\alpha:i\rightarrow j$, we let
%\begin{align}
% \label{Xij}
%X_{ij} =\text{ the indecomposable $kQ$-module with $\ts_i,\ts_j$ as composition factors}.
%\end{align}
%
$X_{ij}$ be the indecomposable $kQ$-module with $\ts_i,\ts_j$ as its composition factors.
Then $X_{ij}$ is a $\Lambda^\imath$-module, which can be viewed as an $\cs^\imath$-module. There exists a non-split short exact sequence in $\mod(kQ)$ (and also in $\mod(\Lambda^\imath)$ or $\rep(\cs^\imath)$)
\begin{equation}\label{eqn:standard sequence}
0\longrightarrow \ts_i\stackrel{\alpha}{\longrightarrow} X_{ij}\stackrel{\beta}{\longrightarrow} \ts_j\longrightarrow0,
\end{equation}
 which yields a triangle
 \begin{equation}   \label{eqn:standard triangle}
\Sigma^{-1}\ts_j\stackrel{\xi}{\longrightarrow} \ts_i\stackrel{\alpha}{\longrightarrow} X_{ij}\stackrel{\beta}{\longrightarrow} \ts_j
\end{equation}
 in $\cd_Q$ (and also in $\cp^\imath$), where $\xi\neq0$.

For any morphism $\gamma:L\rightarrow N$ in $\cp^\imath$, denote by
\begin{align*}
\gamma_M^* &=\Hom_{\cp^\imath}(\gamma,M):\Hom_{\cp^\imath}(N,M) \longrightarrow \Hom_{\cp^\imath}(L,M),
\\
\gamma^M_* &= \Hom_{\cp^\imath}(M,\gamma):\Hom_{\cp^\imath}(M,L) \longrightarrow \Hom_{\cp^\imath}(M,N).
\end{align*}
We denote
\begin{align}
\bw^{{ij}}&=\e_{\sigma \ts_i}+\e_{\sigma\ts_{j}},
\notag \\
\bv^{{ij}}&= \sum_{z\in\mcr^\imath_0-\cs^\imath_0}\dim(\Hom_{\cp^\imath}(\ts_i,z)/\ker\xi^*_z )\e_z
\label{eq:vij2} \\
%\notag&=\sum_{z\in\mcr^\imath_0-\cs^\imath_0}\big(\dim\Hom_{\cp^\imath}(\ts_i,z)-\dim\Im\alpha^*_z\big) \e_z\\
\notag&=\bv^{i}-\sum_{z\in\mcr^\imath_0-\cs^\imath_0}\dim(\Im\alpha^*_z) \e_z.
\end{align}

\begin{lemma}
  \label{lem:v-gij}
If $\Ext^1_{kQ}(\ts_j,\ts_i)=k$ for some $i,j\in Q_0$, then
\begin{itemize}
\item[(i)] $\bv^{{ij}}(\ts_i)=1$, $\bv^{{ij}}(\tau \ts_i)=0$, $\bv^{{ij}}(\tau\ts_{j})=1$, $\bv^{{ij}}(\ts_j)=0$;
\item[(ii)] $\bv^{{ij}}\leq \bv^{i}$\text{ and }$\bv^{{ij}}\leq \bv^{{\btau(j)}}$.
\end{itemize}
\end{lemma}

\begin{proof}
(i) It is clear that $\bv^{{ij}}(\ts_i)=1$. To prove the second identity, we note
\[
\Hom_{\cp^\imath}(\ts_i,\tau \ts_i)=\Hom_{\cd_Q}(\ts_i,\tau \ts_i)\oplus \Hom_{\cd_Q}(\ts_i,\Sigma\tau \ts_{\btau i})= \Hom_{\cd_Q}(\ts_i,\Sigma\tau \ts_{\btau i}).
\]
For any $f\in \Hom_{\cd_Q}(\ts_i,\Sigma\tau \ts_{\btau i})$, we have $f\xi=0$, cf. \eqref{eqn:standard triangle}. So
\[
\bv^{{ij}}(\tau \ts_i)=\dim\Hom_{\cp^\imath}(\ts_i,\tau \ts_i)-\dim\ker\xi^*_{\tau \ts_i}=0.
\]

Let us prove the third identity. Note first by definition that
\[
\bv^{{i}}(\tau\ts_{j})=\dim\Hom_{\cp^\imath}(\ts_i,\tau \ts_j)=\dim\Ext^1_{kQ}(\ts_j,\ts_i)=1.
\]
On the other hand, recalling \eqref{eqn:standard sequence}, we have
\begin{align*}
\Hom_{\cp^\imath}(X_{ij},\tau \ts_j)&=\Hom_{\cd_Q}(X_{ij},\tau \ts_j)\oplus \Hom_{\cd_Q}(X_{ij},\Sigma\tau \ts_{\btau j})\\
&\cong D\Ext^1_{kQ}(\ts_j,X_{ij})\oplus \Hom_{kQ}(\ts_{\btau j},X_{ij}).
\end{align*}
We claim that $\Hom_{kQ}(\ts_{\btau j},X_{ij})=0$. In fact, it is clear if $\btau j\neq j$; if $\btau j=j$, as $kQ$ is representation-directed, it follows by $\Hom_{kQ}(X_{ij},\ts_j)\neq0$ that $\Hom_{kQ}(\ts_j,X_{ij})=0$. Also note that $\Ext^1_{kQ}(\ts_{j},X_{ij})=0$. Therefore, $\Hom_{\cp^\imath}(X_{ij},\tau \ts_j)=0$, which shows that $\Im\alpha_{\tau \ts_j}^*=0$. So $\bv^{{ij}}(\tau\ts_{j})=\bv^{i}(\tau \ts_j)=1$.

Let us prove the fourth identity. Note that
\begin{align*}
&\bv^{{ij}}(\ts_j)\leq\dim\Hom_{\cp^\imath}(\ts_i,\ts_j)\\
&=\dim\Hom_{\cd_Q}(\ts_i,\ts_j)+\dim\Hom_{\cd_Q}(\ts_i,\Sigma \ts_{\btau(j)})=\dim\Ext^1_{kQ}(\ts_i,\ts_{\btau(j)}).
\end{align*}
If $\Ext^1_{kQ}(\ts_i,\ts_{\btau(j)})\neq0$, there exists an arrow $\btau j\rightarrow i$ in $Q$. However, $\Ext^1_{kQ}(\ts_j,\ts_i)=k$, so there exists an arrow $i\rightarrow j$ in $Q$. Then there is an oriented cycle $i\rightarrow j\rightarrow \btau i\rightarrow \btau j\rightarrow i$ in $Q$, which is a contradiction.
So $\Ext^1_{kQ}(\ts_i,\ts_{\btau(j)})=0$, and then $\bv^{{ij}}(\ts_j)=0$.

(ii) The first inequality follows by \eqref{eq:vij2}. Let us prove the second inequality.
For any indecomposable $kQ$-module $M$, applying $\Hom_{\cp^\imath}(-,M)$ to (\ref{eqn:standard triangle})
yields an exact sequence
$$\Hom_{\cp^\imath}(X_{ij},M)\xrightarrow{\alpha_M^*} \Hom_{\cp^\imath}(\ts_i,M)\xrightarrow{\xi_M^*} \Hom_{\cp^\imath}(\Sigma^{-1}\ts_j,M)=  \Hom_{\cp^\imath}(\ts_{\btau(j)},M).$$
Then
$\bv^{{ij}}(M)= \dim\big(\Hom_{\cp^\imath}(\ts_i,M)/\ker\xi^*_M\big)\leq \dim \Hom_{\cp^\imath}(\ts_{\btau(j)},M)=\bv^{{\btau(j)}}(M)$.
\end{proof}

\begin{lemma}\label{lem:v-gij 2}
If $\Ext^1_{kQ}(\ts_j,\ts_i)=k$ for $i,j\in Q_0$, then $(\bv^{{ij}},\bw^{{ij}})$ is an $l$-dominant pair.
\end{lemma}

\begin{proof}
For any indecomposable $kQ$-module $M$, we have an AR triangle in $\cd_Q$
\begin{align}
\tau M\stackrel{\rho_M}{\longrightarrow} N_M\stackrel{\mu_M}{\longrightarrow} M \stackrel{\nu_M}{\longrightarrow} \Sigma \tau M,
\end{align}
for some $kQ$-module $N_M$, which can also be viewed as an AR triangle in the triangulated orbit category $\cp^\imath\simeq \cd_Q/\Sigma\widehat{\btau}$.
Together with (\ref{eqn:standard triangle}), we have the following commutative diagram where every sequence in each row and each column is exact except the third column:
{\tiny\[\xymatrix{\Hom_{\cp^\imath}(\ts_j,\tau M) \ar[r]^{\beta^*_{\tau M}} \ar[dd]^{(\rho_M)_*^{\ts_j}} &\Hom_{\cp^\imath}( X_{ij},\tau M) \ar[rr]^{\alpha^*_{\tau M}} \ar[dd]^{(\rho_M)_*^{X_{ij}}}\ar[dr]& &\Hom_{\cp^\imath}(\ts_i,\tau M)\ar[dd]^{(\rho_M)_*^{\ts_i}} \ar[r]^{\xi^*_{\tau M}} &\Hom_{\cp^\imath}(\Sigma^{-1}\ts_j,\tau M)\ar[dd]^{(\rho_M)_*^{\Sigma ^{-1} \ts_j}} \\
&& \Im\alpha^*_{\tau M} \ar[dd]^>>>>>>{\mu'} \ar[ur] & \\
 \Hom_{\cp^\imath}( \ts_j,N_M) \ar[r]^{\beta^*_{N_M}} \ar[dd]^{(\mu_M)_*^{\ts_j}} &\Hom_{\cp^\imath}( X_{ij},N_M) \ar[rr]^{\alpha^*_{N_M}\quad\quad} \ar[dd]^{(\mu_M)_*^{X_{ij}}} \ar[dr]&&  \Hom_{\cp^\imath}(\ts_i,N_M)\ar[dd]^{(\mu_M)_*^{\ts_i}}\ar[r]^{\xi_{N_M}^*} & \Hom_{\cp^\imath}( \Sigma^{-1}\ts_j,N_M)\ar[dd]^{(\mu_M)_*^{\Sigma^{-1}\ts_j}}\\
 &&\Im\alpha^*_{N_M} \ar[dd]^>>>>>>{\mu''} \ar[ur] &\\
 \Hom_{\cp^\imath}( \ts_j, M) \ar[r]^{\beta^*_M} &\Hom_{\cp^\imath}( X_{ij},M) \ar[rr]^{\alpha^*_M\quad\quad}\ar[dr] & &\Hom_{\cp^\imath}(\ts_i, M)\ar[r]^{\xi^*_{M}} & \Hom_{\cp^\imath}( \Sigma^{-1}\ts_j, M)  \\
 &&\Im\alpha^*_{ M}  \ar[ur] & }\]}

The proof is divided into five cases (a)--(e) below.

(a) 
\underline{$M\notin \{\ts_i,\ts_{\btau i}, \ts_{j},\ts_{\btau(j)},X_{ij}, \widehat{\btau}(X_{ij})\}$}. Similar to the proof of Lemma \ref{lem:dimCartan}(iv), the sequences in the first, second, fourth and fifth columns are short exact, and thus the sequence in the third column is short exact. By \eqref{def:Cq} and Lemma \ref{lem:dimCartan}(iv),  we have
\begin{align*}
(\sigma^*\bw^{{ij}}-{\mathcal C}_q \bv^{{ij}})(M)
&=(\e_{\ts_i}+\e_{\ts_j}-{\mathcal C}_q (\bv^{i})+{\mathcal C}_q\Big (\sum_{z\in\mcr^\imath_0-\cs^\imath_0}\dim(\Im\alpha^*_z) \e_z) \Big) (M)\\
&=(\e_{\ts_i}+\e_{\ts_j}-\sigma^* \bw^{i})(M)=0.
\end{align*}

(b)
\underline{$M=X_{ij}$}. Then the sequences in the fourth and fifth column are short exact, so the Snake Lemma shows that the sequence in the third column is exact with $\mu'$ injective.
It follows by the assumption that $\btau(j)\neq i$ and then $\Hom_{kQ}(\ts_{\btau(j)},X_{ij})=0$. So $\Hom_{\cp^\imath}( \Sigma^{-1}\ts_j, X_{ij})\cong \Ext^1_{kQ}(\ts_j,X_{ij})\oplus \Hom_{kQ}(\ts_{\btau(j)},X_{ij})=0$, which shows that $\xi_M^*=0$ and thus $\alpha_M^*$ is surjective. Together with $\Hom_{\cp^\imath}(\ts_i, X_{ij})=k$, we obtain $\dim \Im\alpha_M^*=1$.
On the other hand, clearly $\Hom_{\cp^\imath}( X_{ij},X_{ij})=k$. By the property of AR triangles, it follows that $(\mu_M)_*^{X_{ij}}=0$. So $\mu''=0$. This implies that $\mu'$ is an isomorphism, and then $\dim \Im\alpha^*_{\tau M} =\dim \Im\alpha^*_{N_M}$.
So $(\sigma^*\bw^{{ij}}-{\mathcal C}_q \bv^{{ij}})(X_{ij})=1$.

(c)
\underline{$M=\widehat{\btau}(X_{ij})$, and $\btau i\neq i$ or $\btau(j)\neq j$}. Then $\widehat{\btau}(X_{ij})\neq X_{ij}$. In this case, the sequence in the third column is exact with $\mu'$ injective. Note that $\Hom_{\cp^\imath}(X_{ij},\widehat{\btau}(X_{ij}))=0$, and so $\Im \alpha_M^*=0$ and $\mu''=0$. This implies that  $\dim \Im\alpha^*_{\tau M} =\dim \Im\alpha^*_{N_M}$.
Therefore, we have $(\sigma^*\bw^{{ij}}-{\mathcal C}_q \bv^{{ij}})(\widehat{\btau}(X_{ij}))=0$.

(d)
\underline{$M=\ts_i$ or $\ts_{\btau i}$}. The sequence in the first and second columns are short exact, the Snake Lemma shows that the sequence in the third column is exact with $\mu''$ epic. We divide the proof into the following two cases.

{\em Subcase (d1)} \underline{$\btau i\neq i$}. We first consider $M=\ts_i$; in this case, we have $\Hom_{\cp^\imath}(X_{ij},\tau \ts_i)\cong D\Ext^1_{kQ}(\ts_i,X_{ij})\oplus D\Hom_{kQ}(\ts_{\btau i},X_{ij})=0$, $\dim \Im \alpha_{\tau M}^*=0$, and $\mu'=0$. Thus $\mu''$ is an isomorphism and $\dim \Im\alpha^*_{N_M}=\dim \Im\alpha^*_{M}$.
Then
\[
(\sigma^*\bw^{{ij}}-{\mathcal C}_q \bv^{{ij}})(\ts_i)= \bigg(\e_{\ts_i}+\e_{\ts_j}-{\mathcal C}_q(\bv^{i})+{\mathcal C}_q\Big(\sum_{z\in\mcr^\imath_0-\cs^\imath_0}\dim(\Im\alpha^*_z) \e_z\Big)\bigg)(\ts_i)=0.
\]

Now consider $M=\ts_{\btau i}$. Note that $\Hom_{\cp^\imath}(\ts_i,\tau \ts_{\btau i})\cong D \Ext^1_{kQ}(\ts_{\btau i},\ts_i) \oplus D \Hom_{(kQ)}(\ts_i, \ts_i)$, which is of dimension $1$, thanks to $\Ext^1_{kQ}(\ts_{\btau i},\ts_i)=0$. Also, note that
$\Hom_{\cp^\imath}(\Sigma^{-1}\ts_j, \tau \ts_{\btau i})\cong D\Hom_{kQ}(\ts_{\btau i},\ts_j)\oplus D\Ext^1_{kQ}(\ts_i,\ts_j) =0$, thanks to $\btau i\neq j$ and $\Ext^1_{kQ}(\ts_j,\ts_i)=k$. Then $\eta_{\tau M}^*=0$, which implies that $\alpha_{\tau M}^*$ is epic, and hence $\dim\Im \alpha_{\tau M}^*=1$.
On the other hand, $\Hom_{\cp^\imath}(\ts_i,\tau \ts_{\btau i})\cong\Hom_{\cd_Q}(\ts_i,\Sigma \tau \ts_i)$. Let us apply $\Hom_{\cd_Q}(\ts_i,-)$ to the AR triangle
\[
\ts_i\xrightarrow{\nu_{\ts_i}} \Sigma \tau \ts_i\xrightarrow{\Sigma(\rho_{\ts_i})} \Sigma N_{\ts_i}\xrightarrow{\Sigma(\mu_{\ts_i})} \Sigma \ts_i.
\]
Clearly, any morphism in $\Hom_{\cd_Q}(\ts_i,\Sigma\tau \ts_{i})$ factors through $\nu_{\ts_i}$, so
$\Hom_{\cd_Q}(\ts_i,\Sigma(\rho_{\ts_i}))=0$. Then $(\rho_M)^{\ts_i}_*=0$, which shows that $\mu'=0$, and it follows that $\dim \Im\alpha^*_{N_{\ts_{\btau i}}} = \dim \Im\alpha^*_{\ts_{\btau i}}$.
Thus we have
\begin{align*}
(\sigma^*\bw^{{ij}}-{\mathcal C}_q \bv^{{ij}})(\ts_{\btau i})
= & \bigg(\e_{\ts_i}+\e_{\ts_j}-{\mathcal C}_q(\bv^{i})+{\mathcal C}_q\Big(\sum_{z\in\mcr^\imath_0-\cs^\imath_0}\dim(\Im\alpha^*_z) \e_z\Big) \bigg) (\ts_{\btau i})
\\
=&-\sigma^*{\bw^{i}}(\ts_{\btau i})+1=0.
\end{align*}

{\em Subcase (d2)} \underline{$\btau i=i$}. Then $\Hom_{\cp^\imath}(X_{ij},\ts_i)= \Hom_{\cd_Q}(X_{ij},\ts_i)\oplus \Hom_{\cd_Q}(X_{ij},\Sigma \ts_i)$. As $kQ$ is representation-directed, we see that $\Hom_{\cp^\imath}(X_{ij},\ts_i)=0$, and thus $\Im\alpha^*_{ \ts_i}=0$.
Note that $\Hom_{\cp^\imath}(X_{ij},\tau \ts_i)\cong D\Ext^1_{kQ}(\ts_i,X_{ij})\oplus D\Hom_{kQ}(\ts_i,X_{ij})$, which is of dimension one.
Similarly, $\Hom_{\cp^\imath}(\ts_j,\tau \ts_i)=0$. So $\beta_{\tau \ts_i}^*=0$, which implies that $\alpha^*_{\tau \ts_i}$ is injective and then $\dim \Im\alpha^*_{\tau \ts_i}=1$.
As $\cd_Q$ is a directed category, we see that each indecomposable summand of $N_{\ts_i}$ is a predecessor of $X_{ij}$, we obtain that
\begin{align*}
\Hom_{\cp^\imath}( X_{ij},N_{\ts_i})
&\cong \Hom_{\cd_Q}(X_{ij}, N_{\ts_i})\oplus \Hom_{\cd_Q}(X_{ij},\Sigma \widehat{\btau} (N_{\ts_i}))\\
&= \Hom_{\cd_Q}(X_{ij},\Sigma \widehat{\btau} (N_{\ts_i})).
\end{align*}
Applying $\Hom_{\cd_Q}(-, \Sigma \widehat{\btau} (N_{\ts_i}))$ to (\ref{eqn:standard triangle}) yields a long exact sequence
$$
\Hom_{\cd_Q}(\ts_j, \Sigma \widehat{\btau} (N_{\ts_i}))\rightarrow \Hom_{\cd_Q}(X_{ij},\Sigma \widehat{\btau} (N_{\ts_i}))\xrightarrow{\Hom_{\cd_Q}(\alpha, \Sigma \widehat{\btau} (N_{\ts_i}))} \Hom_{\cd_Q}(\ts_{i},\Sigma \widehat{\btau} (N_{\ts_i})).
$$
Furthermore, we have $\Hom_{\cd_Q}(\ts_{i},\Sigma \widehat{\btau} (N_{\ts_i}))\cong \Hom_{\cd_Q}(\ts_{i},\Sigma (N_{\ts_i}))\cong D\Hom_{\cd_Q}(N_{\ts_i},\tau \ts_i)$.
The AR triangle $\tau \ts_i\rightarrow N_{\ts_i}\rightarrow \ts_i\rightarrow \Sigma\tau \ts_i$ implies that each indecomposable summand of $N_{\ts_i}$ is a successor of $\tau \ts_i$, so $\Hom_{\cd_Q}(N_{\ts_i},\tau \ts_i)=0$. Therefore, we have
$\Hom_{\cd_Q}(\ts_{i},\Sigma \widehat{\btau} (N_{\ts_i}))=0,$ and then
$\Hom_{\cd_Q}(\alpha, \Sigma \widehat{\btau} (N_{\ts_i}))=0$. Thus we obtain $\alpha^*_{N_{\ts_i}}=0$ and $\dim \Im\alpha^*_{N_{\ts_i}}=0$.
Hence, $\dim \Im\alpha^*_{\tau \ts_i}- \dim \Im\alpha^*_{N_{\ts_i}} +\dim \Im\alpha^*_{\ts_i}=1$.
By Lemma \ref{lem:dimCartan}(ii), we have ${\mathcal C}_q(\bv^{{i}})(\ts_i)=\sigma^*(\bw^{{ij}})(\ts_i)=2$. Then we obtain
\begin{align*}
&(\sigma^*\bw^{{ij}}-{\mathcal C}_q \bv^{{ij}}) (\ts_i)\\
&=1-{\mathcal C}_q(\bv^{{i}})(\ts_i)+ \dim \Im\alpha^*_{\tau \ts_i}- \dim \Im\alpha^*_{E} + \Im\alpha^*_{\ts_i}=1-2+1=0.
\end{align*}

(e)
\underline{$M=\ts_j$}.
We have
\[
\Hom_{\cp^\imath}(\ts_i,\ts_j)=\Hom_{\cd_Q}(\ts_i,\ts_j)\oplus\Hom_{\cd_Q}(\ts_i,\Sigma \ts_{\btau(j)})\cong \Ext^1_{kQ}(\ts_i,\ts_{\btau(j)}).
\]
By a similar argument as for Lemma \ref{lem:v-gij}(i), one proves that $\Ext^1_{kQ}(\ts_i,\ts_{\btau(j)})=0$, and then $\Im \alpha_{\ts_j}^*=0$.
Similarly, $\Hom_{\cp^\imath}(X_{ij},\tau \ts_j)=\Hom_{\cd_Q}(X_{ij},\tau \ts_j)\oplus\Hom_{\cd_Q}(X_{ij},\Sigma \tau \ts_{\btau(j)})\cong D\Ext^1_{kQ}(\ts_j,X_{ij})\oplus D\Hom_{kQ}(\ts_{\btau(j)},X_{ij})=0$. So $\Im \alpha_{\tau \ts_j}^*=0$.

We have
\begin{align*}
\Hom_{\cp^\imath}(\ts_i,\tau \ts_j)&=\Hom_{\cd_Q}(\ts_i,\tau \ts_j)\oplus\Hom_{\cd_Q}(\ts_i,\Sigma \tau \ts_{\btau(j)})\\
&\cong D\Ext^1_{kQ}(\ts_j,\ts_{i})\oplus D\Hom_{kQ}(\ts_{\btau(j)},\ts_i),
\end{align*}
which is of dimension $1$. Together with $\Hom_{\cp^\imath}(\ts_i,\ts_j)=0$, it follows from the exact sequence in the fourth column that
$\dim \Hom_{\cp^\imath}(\ts_i,N_{\ts_j})\leq1$. Then  $\dim \Im\alpha^*_{N_{\ts_j}}\leq1$.
Therefore,
\begin{align*}
& ( \sigma^* \bw^{{ij}} -{\mathcal C}_q\bv^{{ij}}) (\ts_j)\\
&=1-{\mathcal C}_q(\bv^{{i}})(\ts_j)+ \dim \Im\alpha^*_{\tau \ts_j}- \dim \Im\alpha^*_{N_{\ts_j}} + \Im\alpha^*_{\ts_j}=1-\dim \Im\alpha^*_{N_{\ts_j}}\geq0.
\end{align*}

Summarizing the above discussion in  (a)--(e), we have proved that $(\bv^{{ij}},\bw^{{ij}})$ is an $l$-dominant pair.
The lemma is proved.
\end{proof}

 \begin{proposition}
   \label{prop:dominant pairs for gij}
Retain the notation in \eqref{eqn:standard sequence}--\eqref{eq:vij2}.  Then $\dimv K_{LR}(X_{ij})=(\bv^{{ij}},\bw^{{ij}})$, and the strongly $l$-dominant pairs $(\bv,\bw^{{ij}})$ are exactly the pairs with $\bv \in \{{\bf 0}, \bv^{{ij}} \}$. In particular, we have $\sigma^*\bw^{{ij}}-{\mathcal C}_q \bv^{{ij}}=\e_{X_{ij}}$.
\end{proposition}

\begin{proof}
We view $\bw^{{ij}}$ as a dimension vector for $\cs^{\gr}$. From the proof of Lemma \ref{lem:v-gij 2}, we obtain that $(\bv^{{ij}},\bw^{{ij}})$ is also an $l$-dominant pair for $\mcr^{\gr}$. However, by \cite[Theorem 3.14]{LeP13}, there exists a bijection between the set of of isomorphism classes of representations in $\rep(\bw^{{ij}},\cs^{\gr})$ and the set of $l$-dominant pairs $(\bv,\bw^{{ij}})$. There are only two classes of representations in $\rep(\bw^{{ij}},\cs^{\gr})$: $\ts_i\oplus \ts_j$ and $X_{ij}$. So the dimension vector of $K_{LR}(X_{ij})$ is $(\bv^{{ij}},\bw^{{ij}})$ in $\rep(\mcr^{\gr})$. Denote by $\cv^{\gr}:\mcr^{\gr}\rightarrow \mcr^\imath$ the natural projection. Applying the pushforward functor $\cv^{\gr}_*:\rep(\mcr^{\gr})\rightarrow \rep(\mcr^\imath)$ and \cite[Lemma 3.16]{Sch18}, we obtain that $\cv^{\gr}_*(K_{LR}(X_{ij}))=K_{LR}(\cv^{\gr}_* X_{ij})=K_{LR}(X_{ij})$.
So $\dimv K_{LR}(X_{ij})=(\bv^{{ij}},\bw^{{ij}})$.
The remaining proof concerning the strongly $l$-dominant pairs $(\bv,\bw^{{ij}})$ is similar to that of Proposition \ref{prop:dominant pairs 1}, and hence omitted.

For the last assertion, as $(\bv^{{ij}},\bw^{{ij}})$ is also an $l$-dominant pair for $\mcr^{\gr}$, by \cite[Theorem~ 3.14]{LeP13},  we have $\sigma^*\bw^{{ij}}-{\mathcal C}_q \bv^{{ij}}=\e_{X_{ij}}$ as dimension vectors in $\mcr^{\gr}$. One notes that $\big(\sigma^*\bw^{{ij}}-{\mathcal C}_q \bv^{{ij}}\big)(X)$ is the same  when regarded in $\mcr^{\gr}$ or $\mcr^\imath$ for any indecomposable $kQ$-module $X$; also see  the proof of the ``if part'' of Lemma \ref{lem:sameDP}. The desired formula follows.
\end{proof}

\subsection{{Strongly $l$-dominant pairs $(\bv,\bw^{{ijk}})$}}

For $i,j,k\in Q_0$, we denote
\begin{align}
  \label{eq:wijk}
\bw^{{ijk}}:=\e_{\sigma \ts_i}+\e_{\sigma \ts_j}+\e_{\sigma \ts_{k}}.
\end{align}

For any $i,j\in Q_0$ such that $\Ext^1_{kQ}(\ts_j,\ts_i)=k$, there exists an arrow $\alpha:i\rightarrow j$ in $Q$ (and then in $\ov{Q}$, cf. \eqref{eqn:i-quiver}). There are also arrows $\varepsilon_{i}:i\rightarrow \btau i$ and $\varepsilon_{\btau(j)}:\btau(j)\rightarrow j$ in $\ov{Q}$. Set
\begin{align}
  \label{eq:XY}
\begin{cases}
X_{i\btau(i)j} = \text{the indecomposable string $\Lambda^\imath$-module with string
$\btau (i) \xleftarrow{\varepsilon_i} i\xrightarrow{\alpha} j$},
\\
Y_{j\btau(j)i} \,= \text{the indecomposable string $\Lambda^\imath$-module with string
$i\xrightarrow{\alpha} j\xleftarrow{\varepsilon_{\btau(j)}}\btau(j)$}.
\end{cases}
\end{align}
Note that the strings are in $\ov{Q}$, and we consider the right modules (i.e., the representations of the opposite quivers).

\begin{lemma}\label{lemma Xijk}
Retain the notation \eqref{eq:XY}. For $i,j\in Q_0$ such that $\Ext^1_{kQ}(\ts_j,\ts_i)=k$, we have
\begin{align*}
\dimv K_{LR}(X_{i\btau(i)j}) &=(\bv^{i}, \bw^{{i\btau(i)j}}), \\
\dimv K_{LR}(Y_{j\btau(j)i}) &=(\bv^{{\btau(j)}}, \bw^{{j\btau(j)i}}).
\end{align*}
\end{lemma}

\begin{proof}
There is a short exact sequence (of right modules) $0\rightarrow \E_i\xrightarrow{f} X_{i\btau(i)j}\xrightarrow{g} \ts_j\rightarrow0$.
As $K_L(\E_i)=K_{LR}(\E_i)=K_R(\E_i)$, we have the following commutative diagram:
\[
\xymatrix{ & KK(\E_j)=0\ar[r] \ar[d] & KK(X_{i\btau(i)j}) \ar[r]^{KK(g)} \ar[d]^{h_0} & KK(\ts_j) \ar[d]^{u_1} & \\
&K_L(\E_i) \ar[r]^{K_L(f)} \ar[d]^{\cong} & K_L(X_{i\btau(i)j}) \ar[r]^{K_L(g)} \ar[d]^{h_1} & K_L(\ts_j) \ar[r]\ar[d]^{u_2} &0\\
&K_{LR}(\E_i) \ar[r]^{K_{LR}(f)} \ar[d]^{\cong} & K_{LR}(X_{i\btau(i)j}) \ar[r]^{K_{LR}(g)} \ar[d]^{h_2} & K_{LR}(\ts_j)\ar[d] &\\
0\ar[r]&K_R(\E_i) \ar[r]^{K_R(f)} & K_R(X_{i\btau(i)j}) \ar[r]^{K_R(g)} &K_R(\ts_j) &}
\]

We claim that the sequence in the third row is short exact. In fact, $K_{LR}(f)$ is injective, $K_{LR}(g)$ is surjective, and $K_{LR}(g)\circ K_{LR}(f)=0$. Then $K_L(f)$ is injective, and $KK(g)$ is an isomorphism by the Snake Lemma. For any morphism $x:K_{LR}(X_{i\btau(i)j})\rightarrow Z$ such that $x\circ K_{LR}(f)=0$, $x\circ h_1$ factors through $K_L(g)$, i.e., there exists $x': K_L(\ts_j)\rightarrow Z$ such that $x\circ h_1=x'\circ K_L(g)$. Furthermore, $x'\circ u_1\circ KK(g)=x'\circ h_1\circ h_0=0$, so $x'\circ u_1=0$ by $KK(g)$ is isomorphic. Thus there exists $x'': K_{LR}(\ts_j)\rightarrow Z$ such that $x''\circ u_2= v'$, and moreover, $x\circ h_1=x''\circ u_2\circ K_L(g)=x''\circ K_{LR}(g) \circ h_1$. As $h_1$ is surjective, we have $x= K_{LR}(g)\circ h_1$. In this way we have obtained a short exact sequence
\[
0\rightarrow K_{LR}(\E_i) \xrightarrow{K_{LR}(f)} K_{LR}(X_{i\btau(i)j}) \xrightarrow{K_{LR}(g)} K_{LR}(\ts_j)\rightarrow0.
\]
It follows that $\dimv K_{LR}(X_{i\btau(i)j}) =\dimv K_{LR}(\E_i)+\dimv K_{LR}(\ts_j)= (\bv^{i}, \bw^{{i\btau(i)j}})$ since $K_{LR}(\ts_j)=\ts_j$.

The proof for the second equation, which is dual to the first one,  is entirely similar and omitted.
\end{proof}

Recall that the set of dimension vectors on $\mcr^\imath$ admits a partial order $\leq$ as in \eqref{eq:leq}. Note that $\bv^{i}\ngtr \bv^{{\btau i}}$ for any $i\in \I$. % such that $\btau i\neq i$.

\begin{lemma}   \label{lem:dominant pairs hijk}
For any $i,j\in Q_0$ and any $l$-dominant pair $(\bv,\bw^{{i\btau(i)j}})$, we have  $\bv\ngtr \bv^{i}$ and $\bv\ngtr \bv^{{\btau i}}$.
\end{lemma}

\begin{proof}
Let $n_{ij}$ be the number of arrows between $i$ and $j$. The proof is divided into three cases (a)-(c) below.

(a) \underline{$n_{ij}=0=n_{\btau(i) j}$}. Then there are only three $\Lambda^\imath$-modules (also $\cs^\imath$-modules) of dimension vector $\bw^{{i\btau(i)j}}$: $\ts_i\oplus \ts_{\btau i}\oplus \ts_j$, $\ts_j\oplus \E_i$, $\ts_j\oplus \E_{\btau i}$. Hence our desired result holds.

(b) \underline{$n_{ij}\neq 0$}. We shall prove the case when there is an arrow $\alpha:i\rightarrow j$, i.e., $\Ext^1_{kQ}(\ts_j,\ts_i)\neq0$; the other similar case will be skipped. If $\btau(j)\neq j$, then there are only four $\Lambda^\imath$-modules (also $\cs^\imath$-modules) of dimension vector $\bw^{{i\btau(i)j}}$: $\ts_i\oplus \ts_{\btau i}\oplus \ts_j$, $\ts_j\oplus \E_i$, $\ts_j\oplus \E_{\btau i}$ and $X_{i\btau(i)j}$. The desired result follows from Lemma \ref{lemma Xijk} and Lemma~\ref{lem:bistable}(ii).

If $\btau(j)=j$, then there is also an arrow $\btau(\alpha):\btau i\rightarrow j$. Comparing to the above, there is an additional $\Lambda^\imath$-module (also $\cs^\imath$-module) $Z$ of dimension vector $\bw^{{i\btau(i)j}}$
which is also a string module with its string $\btau i\xrightarrow{\btau(\alpha)}j\xleftarrow{\alpha}i$.
Suppose for a contradiction that $\dimv K_{LR}(Z)> \bv^{i}$. Since $\bv^{i}(\tau \ts_{\btau i})=1$, we get that the simple $\mcr^\imath$-module $\ts_{ (\tau\ts_{\btau i})}$ is a composition factor of $K_{LR}(Z)$. As $\Top(K_{LR}(Z))=\ts_{\sigma( \ts_j)}$ and there is only one arrow $\tau \ts_{j}\rightarrow \sigma(\ts_j)$ in $\mcr^\imath$, we have
$\Hom_{\cp^\imath}(\tau \ts_{\btau i},\tau \ts_j)\neq 0$ by noting that we consider the representations of opposite quivers. However, $\Hom_{\cp^\imath}(\tau \ts_{\btau i},\tau \ts_j)\cong \Hom_{kQ}(\ts_{\btau i},\ts_j)\oplus \Ext^1_{kQ}(\ts_i,\ts_j)=0$ since
$\btau i\neq j$ and $\Ext^1_{kQ}(\ts_j,\ts_i)\neq 0$, which is a contradiction.
Similarly we prove that $\dimv K_{LR}(Z)\ngtr \bv^{{\btau i}}$.

(c) \underline{$n_{\btau (i)j}\neq 0$}. It is similar to Case (b) and hence omitted.
\end{proof}

\begin{proposition}
 \label{prop:witauij}
Assume that $i,j\in Q_0$ is not connected by an arrow. Then,
\begin{itemize}
\item[(i)] the strongly $l$-dominant pairs $(\bv,\bw^{{i\btau(i)j}})$ are exactly the pairs with $\bv\in \{{\bf 0}, \bv^{i}, \bv^{{\btau (i)}} \}$;
\item[(ii)] the strongly $l$-dominant pairs $(\bv,\bw^{{ij\btau(j)}})$ are exactly the pairs with $\bv \in \{ {\bf 0}, \bv^{j}, \bv^{{\btau(j)}} \}$.
\end{itemize}
\end{proposition}

\begin{proof}
The two cases are similar, and we shall only prove (i). The only $\Lambda^\imath$-modules with dimension vector $\bw^{{i\btau(i)j}}$ are $\ts_i\oplus \ts_{\btau (i)}\oplus \ts_j$, $\E_i\oplus \ts_j$
and $\E_{\btau (i)}\oplus \ts_j$. The remaining argument is the same as for Proposition~ \ref{prop:dominant pairs 1}, and will be omitted.
\end{proof}

\begin{proposition}  \label{prop:dominant for gij}
Assume that $\Ext^1_{kQ}(\ts_j,\ts_i)\neq0$, for $i,j\in Q_0$.
\begin{itemize}
\item[(i)] If $\btau(j)\neq j$, then the strongly $l$-dominant pairs $(\bv,\bw^{{ijj}})$ are exactly the pairs with $\bv \in \{ {\bf 0}, \bv^{{ij}} \}$;
\item[(ii)] If $\btau(j)=j$, then the strongly $l$-dominant pairs $(\bv,\bw^{{ijj}})$ are exactly the pairs with $\bv \in \{ {\bf 0}, \bv^{{ij}}, \bv^{j} \}$.
\end{itemize}
\end{proposition}

\begin{proof}
We will only give the detailed proof for Part (ii) as the proofs for (i) is similar and simpler.
%
%Proposition~ \ref{prop:dominant pairs for gij} shows that $(0,\bw^{{ijj}})$ and $(\bv^{{ij}}, \bw^{{ijj}})$ are strongly $l$-dominant pairs. On the other hand, since $\btau(j)\neq j$, the $\Lambda^\imath$-modules (also $\cs^\imath$-modules) with dimension vector $\bw^{{ijj}}$ are $\ts_i\oplus \ts_j\oplus \ts_j$ and $\ts_j\oplus X_{ij}$. So there are at most two strongly $l$-dominant pairs $(\bv,\bw^{{ijj}})$. The desired result follows.
%
It follows by Proposition~ \ref{prop:dominant pairs for gij} and Lemma~ \ref{lemma Xijk} that $(\bv,\bw^{{ijj}})$ with $\bv \in \{ {\bf 0}, \bv^{{ij}}, \bv^{j} \}$ are strongly $l$-dominant pairs. On the other hand, since $\btau(j)=j$, the $\Lambda^\imath$-modules (also $\cs^\imath$-modules) with dimension vector $\bw^{{ijj}}$ are $\ts_i\oplus \ts_j\oplus \ts_j$, $\ts_j\oplus X_{ij}$, $\E_j\oplus \ts_i$ and the string module $M$ with $i\rightarrow j\xleftarrow{\varepsilon_j} j$ as its string in $Q$. Note that $\dimv K_{LR}(\E_j\oplus \ts_i)=(\bv^j,\bw^{ijj})=\dimv K_{LR}(M)$. So there are at most three strongly $l$-dominant pairs $(\bv,\bw^{{ijj}})$. This proves (ii).
\end{proof}

%%%%%%%%
%%%%%%%%
\section{Computation in quantum Grothendieck rings for $\imath$quivers}
   \label{sec:computation}

In this section, we  study in depth the quantum Grothendieck rings for Dynkin $\imath$quivers equipped with a twisted multiplication, and determine their relations in cases of $\imath$quivers of rank $2$. They will be used in Section~\ref{sec:main} to verify the $\imath$Serre relations in the algebra $(\tRiZ, \cdot)$ defined in \eqref{eq:tRi}.

\subsection{A bilinear form}

Let $(Q,\btau)$ be a Dynkin $\imath$quiver. We calculate the bilinear pairings $d(\cdot, \cdot)$ defined in \eqref{definition:d} for the cases which will be needed later. Recall $\bv^{{ij}}, \bw^{{ij}}$ from \eqref{eq:vij2}.
%Recall that
%\begin{equation*}
%d((\bv_1,\bw_1),(\bv_2,\bw_2))=(\sigma^*\bw_1-{\mathcal C}_q\bv_1)\cdot \tau^* \bv_2+\bv_1\cdot \sigma^*\bw_2.
%\end{equation*}

\begin{proposition}  \label{prop:bilinear form d}
If $\Ext^1_{kQ}(\ts_j,\ts_i)=k$, for $i,j\in Q_0$, then
\begin{align}
\label{eqn:d 1}
d\big((\bv^{{ij}},\e_{\sigma \ts_i} ), (0,\e_{\sigma \ts_j} )\big)=0, &&d\big( (0,\e_{\sigma \ts_j} ),(\bv^{{ij}},\e_{\sigma \ts_i} )\big)=1;\\
\label{eqn:d 2}
d\big((\bv^{{ij}},\e_{\sigma \ts_i} ), (0,\bw^{{ij}})\big)=1,&&d\big( (0,\bw^{{ij}}),(\bv^{{ij}},\e_{\sigma \ts_i} )\big)=1;\\
\label{eqn:d 3}
d\big( (0,\e_{\sigma \ts_i} ),(\bv^{{ij}},\bw^{{ij}} )\big)=0,&& d\big( (\bv^{{ij}},\bw^{{ij}} ),(0,\e_{\sigma \ts_i} )\big)=1;
\end{align}
\begin{align}
\label{eqn:d 4}
d\big((0,\e_{\sigma \ts_i}),(\bv^{i},\bw^{{ij}})\big)=\left\{ \begin{array}{cc}1&\text{ if }\btau i=i,\\
0&\text{ otherwise,} \end{array}\right.
\end{align}
\begin{align}
\label{eqn:d 5}
d\big((\bv^{i},\bw^{{ij}}),(0,\e_{\sigma \ts_i})\big) =1,&&
d\big((\bv^{i},\e_{\sigma \ts_i} ),( 0,\bw^{{ij}} ) \big)=1;
\end{align}
\begin{align}
\label{eqn:d 6}
d\big(( 0,\bw^{{ij}} ) ,(\bv^{i},\e_{\sigma \ts_i} )\big)=\left\{ \begin{array}{cc}2&\text{ if }\btau i=i,\\
1&\text{ otherwise;} \end{array}\right.
\end{align}
\begin{align}
\label{eqn:d 7}
d\big((0,\e_{\sigma \ts_j} ),(\bv^{{ij}}, \bw^{{ij}})\big)=1,&& d\big((\bv^{{ij}}, \bw^{{ij}}),(0,\e_{\sigma \ts_j} )\big)=0;
\end{align}
\begin{align}
\label{eqn:d 8}
d\big((\bv^{j},\e_{\sigma \ts_j} ),(0, \bw^{{ij}})\big)=\left\{ \begin{array}{cc}2 & \text{ if }\btau i=i \text{ or }\btau j=j,\\
1&\text{ otherwise;}\end{array}\right.
\end{align}
\begin{align}
\label{eqn:d 9}
d\big((0, \bw^{{ij}}),(\bv^{j},\e_{\sigma \ts_j} )\big)=\left\{ \begin{array}{cc}1 & \text{ if }\btau j=j,\\
0&\text{ otherwise.}\end{array}\right.
\end{align}
If in addition $\btau(j)=j$, then
\begin{align}
\label{eqn:d 10}
d\big((\bv^{{j}}-\bv^{{ij}},\e_{\sigma \ts_{j}} ),(\bv^{{ij}}, \bw^{{ij}})\big)=d\big((\bv^{{ij}}, \bw^{{ij}}),(\bv^{j}-\bv^{{ij}},\e_{\sigma \ts_{j}} )\big)=1.
\end{align}
\end{proposition}

\begin{proof}
The identities \eqref{eqn:d 1}--\eqref{eqn:d 3} follow by a direct computation using Lemma \ref{lem:v-gij}(i):
\begin{align*}
d\big((\bv^{{ij}},\e_{\sigma \ts_i} ), (0,\e_{\sigma \ts_j} )\big)
&=\bv^{{ij}}\cdot \e_{\ts_j}=\bv^{{ij}}(\ts_j)=0,\\
d\big( (0,\e_{\sigma \ts_j} ),(\bv^{{ij}},\e_{\sigma \ts_i} )\big)
&=\e_{\ts_j}\cdot\tau^*(\bv^{{ij}})=\bv^{{ij}}(\tau \ts_j)=1,\\
d\big((\bv^{{ij}},\e_{\sigma \ts_i} ), (0,\bw^{{ij}})\big)
&=\bv^{{ij}}\cdot\sigma^*(\bw^{{ij}})=\bv^{{ij}}(\ts_i)+\bv^{{ij}}(\ts_j)=1,\\
d\big( (0,\bw^{{ij}}),(\bv^{{ij}},\e_{\sigma \ts_i} )\big)
&=\sigma^*(\bw^{{ij}})\cdot \tau^*(\bv^{{ij}})=\bv^{{ij}}(\tau \ts_i)+\bv^{{ij}}(\tau \ts_j)=1,\\
d\big( (0,\e_{\sigma \ts_i} ),(\bv^{{ij}},\bw^{{ij}} )\big)
&=\sigma^*(\e_{\sigma \ts_i})\cdot \tau^*(\bv^{{ij}})=\bv^{{ij}}(\tau \ts_i)=0,\\
d\big( (\bv^{{ij}},\bw^{{ij}} ),(0,\e_{\sigma \ts_i} )\big)
&=\bv^{{ij}}(\ts_i)=1.
\end{align*}

The identities \eqref{eqn:d 4}--\eqref{eqn:d 6} follow by Lemma \ref{lem:dimCartan} and Lemma \ref{lem:v-gij} as follows:
\begin{align*}
d\big((0,\e_{\sigma \ts_i}),(\bv^{i},\bw^{{ij}})\big)&= \sigma^*(\e_{\sigma \ts_i})\cdot \tau^*(\bv^{i})=\bv^{i}(\tau \ts_i)=\left\{ \begin{array}{cc}1&\text{ if }\btau i=i,\\
0&\text{ otherwise,} \end{array}\right. \\
d\big((\bv^{i},\bw^{{ij}}),(0,\e_{\sigma \ts_i})\big) &=\bv^{i}\cdot \sigma^*(\e_{\sigma \ts_i})=\bv^{i}(\ts_i)=1,\\
d\big((\bv^{i},\e_{\sigma \ts_i} ),( 0,\bw^{{ij}} ) \big)&=\bv^{i}\cdot \sigma^*(\bw^{{ij}})=\bv^{i}(\ts_i)+\bv^{i}(\ts_j)=1,\\
d\big(( 0,\bw^{{ij}} ) ,(\bv^{i},\e_{\sigma \ts_i} )\big)&=\sigma^*(\bw^{{ij}}) \cdot \tau^* (\bv^{i}) =\bv^{i}(\tau \ts_i)+\bv^{i}(\tau \ts_j)=\left\{ \begin{array}{cc}2&\text{ if }\btau i=i,\\
1&\text{ otherwise.} \end{array}\right.
\end{align*}

The identities \eqref{eqn:d 7}--\eqref{eqn:d 9} follow by Lemma \ref{lem:dimCartan} and Lemma \ref{lem:v-gij} as follows:
\begin{align*}
d\big((0,\e_{\sigma \ts_j} ),(\bv^{{ij}}, \bw^{{ij}})\big)&=\sigma^*(\e_{\sigma \ts_j})\cdot \tau^* (\bv^{{ij}}) =\bv^{{ij}}(\tau \ts_j)=1,
\\
d\big((\bv^{{ij}}, \bw^{{ij}}),(0,\e_{\sigma \ts_j} )\big)&=\bv^{{ij}}(\ts_j)=0,
\\
%\end{align*}
%For \eqref{eqn:d 8}--\eqref{eqn:d 9}, by Lemma \ref{lem:dimCartan}, we have
%\begin{align*}
d\big((\bv^{j},\e_{\sigma \ts_j} ),(0, \bw^{{ij}})\big)&=\bv^{j}\cdot \sigma^* (\bw^{{ij}}) =\bv^{j}(\ts_i)+\bv^{j}(\ts_j)
=\left\{ \begin{array}{cc}2 & \text{ if }\btau i=i \text{ or }\btau j=j,\\
1&\text{ otherwise;}\end{array}\right.
\\
d\big((0, \bw^{{ij}}),(\bv^{j},\e_{\sigma \ts_j} )\big)&=\sigma^* (\bw^{{ij}}) \cdot \tau^*(\bv^{j})
=\bv^{j}(\tau \ts_i)+\bv^{j}(\tau \ts_j)
= \left\{ \begin{array}{cc}1& \text{ if }\btau j=j,\\
0&\text{ otherwise;}\end{array}\right.
\end{align*}

It remains to prove \eqref{eqn:d 10}. Since $\btau(j)=j$, we obtain $2\e_{ \ts_j}=\sigma^*(\bw^{j})={\mathcal C}_q(\bv^{j})$ by Lemma~ \ref{lem:dimCartan}(iv). Using Lemmas \ref{lem:v-gij}--\ref{lem:v-gij 2}, we have
\begin{align*}
&d\big((\bv^{{j}}-\bv^{{ij}},\e_{\sigma \ts_{j}} ),(\bv^{{ij}}, \bw^{{ij}})\big)\\
&= \big(\e_{\ts_{j}} -{\mathcal C}_q(\bv^{{j}}-\bv^{{ij}}) \big)\cdot \tau^*(\bv^{{ij}})+\sigma^* (\bw^{{ij}}) \cdot (\bv^{j}-\bv^{{ij}})\\
&= \big(\e_{\ts_{j}} -\sigma^*(\bw^{{j}})+{\mathcal C}_q(\bv^{{ij}}) \big)\cdot \tau^*(\bv^{{ij}}) +\bv^{j}(\ts_i)+\bv^{j}(\ts_j)-\bv^{{ij}}(\ts_i)-\bv^{{ij}}(\ts_j)\\
&= - \big(\e_{\ts_{j}}-{\mathcal C}_q(\bv^{{ij}}) \big)\cdot \tau^*(\bv^{{ij}})+1+1-1-0\\
&= - \big(\sigma^*(\bw^{ij})-{\mathcal C}_q(\bv^{{ij}}) \big)\cdot \tau^*(\bv^{{ij}})+ \e_{\ts_i}\cdot\tau^*(\bv^{{ij}})+  1\\
&= - \e_{X_{ij}} \cdot \tau^*(\bv^{{ij}})+ \e_{\ts_{i}} \cdot \tau^*(\bv^{{ij}})+1\\
&= -\bv^{{ij}}(\tau X_{ij})+1=1,
\end{align*}
where the last equality follows from
\[
\Hom_{\cp^\imath}(\ts_i,\tau X_{ij})\cong D\Ext^1_{kQ}(X_{ij},\ts_i)\oplus D \Hom_{kQ}(X_{ij},\ts_{\btau i})=0.
\]

Similarly, by Lemmas \ref{lem:v-gij}--\ref{lem:v-gij 2} we have
\begin{align*}
d\big((\bv^{{ij}}, \bw^{{ij}}),(\bv^{j}-\bv^{{ij}},\e_{\sigma \ts_{j}} )\big)
&= (\sigma^* \bw^{{ij}}-{\mathcal C}_q \bv^{{ij}})\cdot \tau^*(\bv^{j}-\bv^{{ij}})+\bv^{{ij}}\cdot  \e_{\ts_{j}}
\\
&= \e_{X_{ij}} \cdot \tau^*(\bv^{j}-\bv^{{ij}}) +\bv^{{ij}}(\ts_{j}) \\
&= \bv^{j}(\tau X_{ij})- \bv^{{ij}}(\tau X_{ij})\\
&= \bv^{j}(\tau X_{ij})=1,
\end{align*}
where the last equality follows from
\begin{align*}
\Hom_{\cp^\imath}(\ts_j,\tau X_{ij})&= \Hom_{\cd_Q}(\ts_j,\tau X_{ij})\oplus \Hom_{\cd_Q}(\ts_j,\Sigma\tau \widehat{\btau}(X_{ij}))\\
&\cong D\Ext^1_{kQ}(X_{ij},\ts_j)\oplus D\Hom_{kQ}(X_{ij},\ts_j)\\
&\cong D\Hom_{kQ}(X_{ij},\ts_j).
\end{align*}
The proposition is proved.
\end{proof}

\subsection{Computation for rank 2 $\imath$quivers, I}

We remind that only the twisted multiplication in $\tRi$ associated to a Dynkin $\imath$quiver $(Q,\btau)$ is used in the whole Section~\ref{sec:computation}. As we shall see in Section~\ref{sec:main}, $L(\bv^{i},\bw^{i})$, for $i\in \I$, correspond to generators for the Cartan subalgebra of $\tUi$ and $L(0,\e_{\sigma \ts_i})$, for $i\in \I$, correspond to Chevalley generators of $\tUi$.

%\blue{Let $C=(c_{ij})$ be the {\em Cartan matrix} of $Q$. Recall the symmetric Euler form defined in \eqref{eqn: symmetric bilinear form}. Then
%$c_{ij}=(\ts_i,\ts_j)$ for any $i,j\in Q_0$.}

Let $n_{ij}$ be the number of edges connecting vertex $i$ and $j$ of $Q$. Let $C=(c_{ij})_{i,j \in \I}$ be the symmetric Cartan matrix of the underlying graph of $Q$, defined by $c_{ij}=2\delta_{ij}-n_{ij}$. Recall the symmetric Euler form defined in \eqref{eqn: symmetric bilinear form}. Then
$c_{ij}=(\ts_i,\ts_j)$ for any $i,j\in Q_0$.

\begin{lemma}
\label{lem:cartan and positive}
The following identity holds in the ring $\tRi$, for any $i,j\in Q_0$:
 \begin{align*}
L(0,\e_{\sigma \ts_i})  L(\bv^{j},\bw^{j})
&= \tt^{-c_{i,\btau(j)}+c_{ij}} L(\bv^{j},\bw^{j})  L(0,\e_{ \sigma \ts_i})
\\
&= \tt^{\frac12(-c_{i,\btau(j)}+c_{ij})} L(\bv^{j},\e_{\sigma \ts_i} +\bw^{j}).
\end{align*}
\end{lemma}

\begin{proof}
It follows by \eqref{equation multiplication} that the term $L(\bv,\e_{\sigma \ts_i}+\bw^{j})$ has nonzero coefficients in each side of the desired identity in the lemma only if $\bv\geq \bv^{j}$. It follows from Lemma~ \ref{lem:dominant pairs hijk} that the only possibility is  $\bv=\bv^{j}$.
Then it follows from \eqref{eqn: leading term} that
\begin{align*}
& L(0,\e_{\sigma \ts_i})  L(\bv^{j},\bw^{j})\\
&= \tt^{-1/2\langle \ts_i,\ts_j\oplus \ts_{\btau(j)}\rangle+1/2\langle \ts_j\oplus \ts_{\btau(j)},\ts_i\rangle  }
\tt^{d((\bv^{j},\bw^{j}),(0,\e_{\sigma \ts_i}))-d((0,\e_{\sigma \ts_i}),(\bv^{j},\bw^{j}))}
L(\bv^{j},\e_{\sigma \ts_i}+\bw^{j})\\
&= \tt^{-1/2\langle \ts_i,\ts_j\oplus \ts_{\btau(j)}\rangle+1/2\langle \ts_j\oplus \ts_{\btau(j)},\ts_i\rangle  }\tt^{\dim\Hom_{\cp^\imath}(\ts_j,\ts_i)-\dim\Hom_{\cp^\imath}(\ts_j,\tau \ts_i)} L(\bv^{j},\e_{\sigma \ts_i}+\bv^{j})\\
&= \tt^{-1/2\langle \ts_i,\ts_j\oplus \ts_{\btau(j)}\rangle+1/2\langle \ts_j\oplus \ts_{\btau(j)},\ts_i\rangle  } \\
&\quad \cdot \tt^{\dim\Hom_{kQ}(\ts_j,\ts_i)+ \dim\Ext^1_{kQ}(\ts_{\btau(j)},\ts_i) -\dim\Hom_{kQ}(\ts_i,\ts_{\btau(j)})-\dim\Ext^1_{kQ}(\ts_i,\ts_{j})}L(\bv^{j},\e_{\sigma \ts_i}+\bw^{j})\\
&= \tt^{-1/2 (\ts_i,\ts_{\btau(j)})+1/2 (\ts_i,\ts_j)} L(\bv^{j},\e_{\sigma \ts_i}+\bw^{j}).
\end{align*}
Similarly we have
\begin{align*}
 L(\bv^{j},\bw^{j})  L(0,\e_{\sigma \ts_i})
 =\tt^{1/2(\ts_i,\ts_{\btau(j)})-1/2 (\ts_i,\ts_j)} L(\bv^{j},\e_{\sigma \ts_i}+\bw^{j}).
\end{align*}
Here, we use the fact $\Hom_{kQ}(\ts_j,\ts_i)=\Hom_{kQ}(\ts_i,\ts_j)=\delta_{ij}$.
Thus,
\begin{align*}
L(0,\e_{\sigma \ts_i})  L(\bv^{j},\bw^{j})
&= \tt^{-(\ts_i,\ts_{\btau(j)})+(\ts_i,\ts_j)} L(\bv^{j},\bw^{j})  L(0,\e_{\sigma \ts_i})\\
&= \tt^{-c_{i,\btau(j)}+c_{ij}} L(\bv^{j},\bw^{j})  L(0,\e_{\sigma \ts_i})
\end{align*}
and then our desired result follows.
\end{proof}

\begin{lemma}
\label{lem:cartan part}
The following identity holds in $\tRi$, for any $i,j\in Q_0$:
\begin{align*}
L(\bv^{i},\bw^{i})  L(\bv^{j},\bw^{j})=L(\bv^{i}+\bv^{j},\bw^{i}+\bw^{j}).
\end{align*}
\end{lemma}

\begin{proof}
By applying (\ref{equation multiplication}), we see that the term $L(\bv,\bw^{i}+\bw^{j})$ in $L(\bv^{i},\bw^{i})  L(\bv^{j},\bw^{j})$ has nonzero coefficients
only if $\bv\geq \bv^{i}+\bv^{j}$. By Corollary~ \ref{cor:KLR for finite projective dimension}, there exists no strongly $l$-dominant pair $(\bv, \bw^{i}+\bw^{j})$ such that $\bv> \bv^{i}+\bv^{j}$. Then we obtain
\begin{align}
\label{cart 1}
L(\bv^{i},\bw^{i})  L(\bv^{j},\bw^{j})
=&\tt^{-1/2\langle \ts_i\oplus \ts_{\btau i},\ts_j\oplus \ts_{\btau j}\rangle_a }
 \tt^{ \bv^{j}\cdot \sigma^* \bw^{i}-\bv^{i}\cdot \sigma^*\bw^{j}} L(\bv^{i}+\bv^{j},\bw^{i}+\bw^{j})\\\notag
=& \tt^{-1/2\langle \ts_i\oplus \ts_{\btau i}, \ts_j\oplus \ts_{\btau j}\rangle+1/2\langle \ts_j\oplus \ts_{\btau j},\ts_i\oplus \ts_{\btau i}\rangle}
  \tt^{ \bv^{j}\cdot \sigma^* \bw^{i}-\bv^{i}\cdot \sigma^*\bw^{j}}
  L(\bv^{i}+\bv^{j},\bw^{i}+\bw^{j}).
\end{align}
By \eqref{eq:vfi}, we have
\begin{align}
\label{cart 2}
&\bv^{j}\cdot \sigma^* \bw^{i}-\bv^{i}\cdot \sigma^*\bw^{j} \\
=&\bv^{j}(\e_{\ts_i}+\e_{\ts_{\btau i}})-\bv^{i}(\e_{\ts_j}+\e_{\ts_{\btau j}}) \notag\\
\notag
=&\dim \Hom_{\cp^\imath}(\ts_j,\ts_i)+\dim \Hom_{\cp^\imath}(\ts_j,\ts_{\btau i}) -\dim \Hom_{\cp^\imath}(\ts_i,\ts_j)-\dim \Hom_{\cp^\imath}(\ts_i,\ts_{\btau j})\\\notag
=&\dim \Hom_{kQ}(\ts_j,\ts_i)+\dim\Ext^1_{kQ}(\ts_j,\ts_{\btau i}) +\dim\Hom_{kQ}(\ts_j,\ts_{\btau i}) +\dim\Ext^1_{kQ}(\ts_j,\ts_{i})\\\notag
&-\dim \Hom_{kQ}(\ts_i,\ts_j)- \dim\Ext^1_{kQ}(\ts_i,\ts_{\btau j})- \dim\Hom_{kQ}(\ts_i,\ts_{\btau j})-\dim\Ext^1_{kQ}(\ts_i,\ts_{j})\\\notag
=&\langle \ts_i,\ts_j\rangle-\langle \ts_{j},\ts_{ \btau i}\rangle +\langle \ts_{\btau i},\ts_j\rangle -\langle \ts_{j},\ts_i\rangle,
\end{align}
where we used the fact $\Hom_{kQ}(\ts_i,\ts_j)=\Hom_{kQ}(\ts_j,\ts_i)=\delta_{ij}k$ and $\langle \ts_i,\ts_j\rangle=\langle \ts_{\btau i},\ts_{\btau j}\rangle$ for all $i,j\in\I$.

On the other hand,  we have
\begin{align}
\label{cart 3}
-1/2  \langle \ts_i & \oplus \ts_{\btau i}, \ts_j\oplus \ts_{\btau j}\rangle+1/2\langle \ts_j\oplus \ts_{\btau j},\ts_i\oplus \ts_{\btau i}\rangle\\\notag
 & =-\langle \ts_i,\ts_j\rangle -\langle \ts_i,\ts_{\btau j}\rangle+\langle \ts_j,\ts_i\rangle +\langle \ts_{\btau j},\ts_{ i}\rangle.
\end{align}
%by noting that $\langle \ts_i,\ts_j\rangle=\langle \ts_{\btau i},\ts_{\btau j}\rangle$ for any $i,j\in\I$.

Combining \eqref{cart 1}--\eqref{cart 3}, we have established the desired formula.
\end{proof}

In fact, similar to the proof of Lemma \ref{lem:cartan part}, using Corollary~ \ref{cor:KLR for finite projective dimension} one can prove that
\begin{align}
\label{eqn: cartan mult 2}
L(\bv^{{i_1}},\bw^{{i_1}})  \cdots   L(\bv^{{i_t}},\bw^{{i_t}})=L \Big(\sum_{j=1}^t\bv^{{i_j}},\sum_{j=1}^t\bw^{{i_j}} \Big),
\end{align}
for any $i_1,\dots,i_t\in \I$.

\begin{lemma}
\label{eqn:cartan}
The following identity holds in $\tRi$, for any $i\in Q_0$:
\begin{align*}
L(0,\e_{\sigma \ts_{i}})   L(0,\e_{\sigma \ts_{\btau i}}) - L(0,\e_{\sigma \ts_{\btau i}})  L(0,\e_{\sigma \ts_i})
= (\tt-\tt^{-1}) \big(L(\bv^{i},\bw^{i})-L(\bv^{{\btau i}},\bw^{{\btau i}} ) \big).
\end{align*}
\end{lemma}

\begin{proof}
The identity trivially holds if $\btau i=i$.

Assume now $\btau i\neq i$. Recall from \eqref{eq:vfi} that $\bw^{{\btau i}} =\bw^{i} =\e_{\sigma \ts_i}+\e_{\sigma\ts_{\btau i}}$. Then the only $l$-dominant pairs $(\bv,\bw^{i})$ are given by $\bv=0$, $\bv=\bv^{i}$ and $\bv=\bv^{{\btau i}}$ by Proposition~ \ref{prop:dominant pairs 1}. Furthermore, we have
\begin{align*}
\Res_{ \e_{\sigma \ts_i}, \e_{\sigma \ts_{\btau i}} }^{\bw^{i}} \big(\cl(0, \bw^{i}) \big)
&=\Res_{ \e_{\sigma \ts_i},\e_{\sigma \ts_{\btau i}} }^{\bw^{i}} \big(\pi(0, \bw^{i}) \big)\\
&=\pi(0,\e_{\sigma \ts_i} )\boxtimes\pi(0, \e_{\sigma \ts_{\btau i}} )=\cl(0,\e_{\sigma \ts_i} )\boxtimes\cl(0, \e_{\sigma \ts_{\btau i}}).
\end{align*}
Similar to \cite[Example 4.4.3]{Qin}, we have $\cl(\bv^{i}, \bw^{i})=\pi(\bv^{i},\bw^{i})$. Then
\begin{align*}
\Res_{ \e_{\sigma \ts_i}, \e_{\sigma \ts_{\btau i}} }^{\bw^{i}} \big(\cl(\bv^{i}, \bw^{i}) \big )
&=\Res_{ \e_{\sigma \ts_i},\e_{\sigma \ts_{\btau i}} }^{\bw^{i}} \big(\pi(\bv^{i}, \bw^{i}) \big)\\
&= \tt\pi(\bv^{i},\e_{\sigma \ts_i} )\boxtimes\pi(0, \e_{\sigma \ts_{\btau i}} )
=\tt\cl(0,\e_{\sigma \ts_i} )\boxtimes\cl(0, \e_{\sigma \ts_{\btau i}});
\\
\Res_{ \e_{\sigma \ts_i}, \e_{\sigma \ts_{\btau i}} }^{\bw^{i}} \big(\cl(\bv^{{\btau i}}, \bw^{i}) \big)
&=\Res_{ \e_{\sigma \ts_i},\e_{\sigma \ts_{\btau i}} }^{\bw^{i}} \big(\pi(\bv^{{\btau i}}, \bw^{i}) \big)\\
&=\tt^{-1}\pi(0,\e_{\sigma \ts_i} )\boxtimes\pi(\bv^{{\btau i}}, \e_{\sigma \ts_{\btau i}} )
=\tt^{-1}\cl(0,\e_{\sigma \ts_i} )\boxtimes\cl(0, \e_{\sigma \ts_{\btau i}}).
\end{align*}
So we have
\begin{align}
\label{eqn: EF1}
L(0,\e_{\sigma \ts_i})   L(0,\e_{\sigma \ts_{\btau i}})
&= L(0,\bw^{i})+\tt L(\bv^{i},\bw^{i})+\tt^{-1}L(\bv^{{\btau i}},\bw^{i}),\\
L(0,\e_{\sigma \ts_{\btau i}})   L(0,\e_{\sigma \ts_i})
&= L(0,\bw^{i})+\tt^{-1}L(\bv^{i},\bw^{i})+\tt L(\bv^{{\btau i}},\bw^{i}),
\end{align}
and then our desired formula follows.
\end{proof}

\subsection{Computation for rank 2 $\imath$quivers, II}

\begin{lemma}\label{lem:relations of no edges}
Assume $i,j\in Q_0$ are not  connected by an arrow and $i\neq \btau(j)$. Then the following identity holds in $\tRi$:
\[
L(0,\e_{\sigma \ts_i})  L(0,\e_{\sigma \ts_j})=L(0,\e_{\sigma \ts_j})  L(0,\e_{\sigma \ts_i}).
\]
%\blue{This corresponds to the relation $B_iB_{j}-B_jB_i =0, \quad \text{ if }a_{ij} =0 \text{ and }\tau i\neq j$.}

%\red{-- We haven't fixed $\I_\btau$ as in \eqref{eqn:representative} yet, so we may have subdiagram like $i$---$\btau(j)$\quad $\btau(i)$---$j$}

%\blue{Even we have subdiagram $i$--$\btau j$ or $\btau(i)$---$j$, since no arrow between $i$ and $j$, and $\varrho(j)\neq i$, $B_i$ and $B_j$ commutes.}

%\red{-- I have avoided $a_{ij}$, or rather $c_{ij}$, which is only defined in Section~\ref{sec:main} }
\end{lemma}

\begin{proof}
Since the only $\Lambda^\imath$-modules with dimension vector $\e_{\sigma \ts_i}+\e_{\sigma \ts_j}$ is $\ts_i\oplus \ts_j$, the only $l$-dominant pair $(\bv, \e_{\sigma \ts_i}+\e_{\sigma \ts_j})$ is $(0,\e_{\sigma \ts_i}+\e_{\sigma \ts_j})$. The remaining proof is the same as for Lemma \ref{eqn:cartan}, and will be omitted here.
\end{proof}

\begin{proposition}
\label{prop:iserre split}
Let $i,j\in Q_0$ be such that $\btau i=i$ and $\btau(j)=j$. If $\Ext^1_{kQ}(\ts_j,\ts_i)=k$, then the following identities hold in $\tRi$:
\begin{align}
\label{split serre 1}L(0,\e_{\sigma \ts_i})^{ 2}   L(0,\e_{\sigma \ts_j})
& -(\tt+\tt^{-1}) L(0, \e_{\sigma \ts_i})   L(0,\e_{\sigma \ts_j})   L(0,\e_{\sigma \ts_i})
+L(0,\e_{\sigma \ts_j})   L(0,\e_{\sigma \ts_i})^{2} \\
\notag
&=-(\tt-\tt^{-1})^2L(\bv^{i}, \bw^{i})  L(0,\e_{\sigma \ts_j});\\
\label{split serre 2}
L(0,\e_{\sigma \ts_j})^{2}    L(0,\e_{\sigma \ts_i})
& -(\tt+\tt^{-1}) L(0, \e_{\sigma \ts_j})    L(0,\e_{\sigma \ts_i})    L(0,\e_{\sigma \ts_j})
+L(0,\e_{\sigma \ts_i})    L(0,\e_{\sigma \ts_j})^{2} \\
\notag
&=-(\tt-\tt^{-1})^2L(\bv^{j}, \bw^{j})   L(0,\e_{\sigma \ts_i}).
\end{align}
\end{proposition}

\begin{proof}
Recall
$\bw^{{ij}}=\e_{\sigma \ts_i} +\e_{\sigma \ts_j}$ from \eqref{eq:vij2}, and $\bw^{{iij}}=2\e_{\sigma \ts_i}+\e_{\sigma \ts_j}$ from \eqref{eq:wijk}. Note $\tt^{-(1/2) \langle (0,\e_{\sigma \ts_i}), (0,\e_{\sigma \ts_j})\rangle_a}=\tt^{-1/2}.$  In this proof, we denote $\bw=\bw^{{iij}}$, and $\heartsuit=\tt^{-1/2}.$
\vspace{2mm}

(1)
\underline{Proof of \eqref{split serre 1}}.
Note first that $\pi(0,\e_{\sigma \ts_i}) =\cl(0,\e_{\sigma \ts_i})$, $\pi(0,\e_{\sigma \ts_j})=\cl(0,\e_{\sigma \ts_j})$ and $\pi(0,\bw^{{ij}})=\cl(0,\bw^{{ij}})$. By Proposition~ \ref{prop:dominant pairs for gij}, the only $l$-dominant pairs $(\bv,\bw^{{ij}})$ are given by
$\bv=0$ and $\bv=\bv^{{ij}}$. We compute by using \eqref{eqn:comultiplication} that
\begin{align*}
 \Res_{ \e_{\sigma \ts_i}, \e_{\sigma \ts_j} }^{\bw^{{ij}}} \big(\cl(0, \bw^{{ij}})\big)
 &=\Res_{ \e_{\sigma \ts_i}, \e_{\sigma \ts_j} }^{\bw^{{ij}}} \big(\pi(0, \bw^{{ij}})\big)\\
 &= \heartsuit\pi(0,\e_{\sigma \ts_i} )\boxtimes\pi(0,\e_{\sigma \ts_j} )=\heartsuit\cl(0,\e_{\sigma \ts_i} )\boxtimes\cl(0,\e_{\sigma \ts_j} ).
\end{align*}

There exists an arrow from $\sigma\ts_i$ to $\ts_i$ in $\mcr^\imath$, which is the only one arrow starting from $\sigma\ts_i$ by definition. Lemma \ref{lem:v-gij}(i) implies that $\bv^{{ij}}-\e_{\ts_j} \ngeq0$, so $\cm(\bv, \e_{\sigma \ts_j})=\emptyset$ for any $\bv\leq \bv^{{ij}}$. It follows from \eqref{eqn:comultiplication} and (\ref{eqn:d 1}) that
\begin{align*}
 \Res_{ \e_{\sigma \ts_i}, \e_{\sigma \ts_j} }^{\bw^{{ij}}} \big(\pi(\bv^{{ij}}, \bw^{{ij}})\big)
= \heartsuit \tt \pi(\bv^{{ij}},\e_{\sigma \ts_i} )\boxtimes\pi(0,\e_{\sigma \ts_j} ).
\end{align*}
Note that $\cm(\bv^{{ij}},\e_{\sigma \ts_i}) =\cm_0(\e_{\sigma \ts_i})= \cm_0(0, \e_{\sigma \ts_i})$ is a point. Then we have $\pi(\bv^{{ij}},\e_{\sigma \ts_i} )= \pi(0,\e_{\sigma \ts_j} ) =\cl(0,\e_{\sigma \ts_i})$. So $\Res_{ \e_{\sigma \ts_i}, \e_{\sigma \ts_j} }^{\bw^{{ij}}} \big(\pi(\bv^{{ij}}, \bw^{{ij}})\big)= \heartsuit \tt \cl(\bv^{{ij}},\e_{\sigma \ts_i} )\boxtimes\cl(0,\e_{\sigma \ts_j} )$.

On the other hand, by \eqref{eqn:decomposition theorem}, we assume that
$$
\pi(\bv^{{ij}}, \bw^{{ij}})=\cl(\bv^{{ij}}, \bw^{{ij}})+ a_{\bv^{{ij}},0,\bw^{{ij}}}(\tt) \cl(0,\bw^{{ij}}).
$$
So we have
\begin{align*}
 \Res_{ \e_{\sigma \ts_i}, \e_{\sigma \ts_j} }^{\bw^{{ij}}} \big(\pi(\bv^{{ij}}, \bw^{{ij}})\big)
&= \Res_{ \e_{\sigma \ts_i}, \e_{\sigma \ts_j} }^{\bw^{{ij}}} \big( \cl(\bv^{{ij}}, \bw^{{ij}}) \big)+ a_{\bv^{{ij}},0,\bw^{{ij}}}(\tt)  \heartsuit \cl(0,\e_{\sigma \ts_i} )\boxtimes\cl(0,\e_{\sigma \ts_j} )\\
&= \heartsuit \tt \cl(0,\e_{\sigma \ts_i} )\boxtimes\cl(0,\e_{\sigma \ts_j} ).
\end{align*}
We must have $a_{\bv^{{ij}},0,\bw^{{ij}}}(\tt)=0$ by noting that $ a_{\bv^{{ij}},0,\bw^{{ij}}}(\tt)$ is bar-invariant.
%In fact, it can also be verified by studying the fiber of $\cm(\bv^{{ij}},\bw^{{ij}})$ over the origin of $\cm_0(\bw^{{ij}})$ similar to \cite[Example 4.4.3]{Qin}.
Then
$$\Res_{ \e_{\sigma \ts_i}, \e_{\sigma \ts_j} }^{\bw^{{ij}}} \big( \cl(\bv^{{ij}}, \bw^{{ij}}) \big)= \heartsuit \tt \cl(\bv^{{ij}},\e_{\sigma \ts_i} )\boxtimes\cl(0,\e_{\sigma \ts_j} ).$$
Therefore, we obtain
\begin{equation}\label{eqn: sisj multiplication}
L(0, \e_{\sigma \ts_i})    L(0, \e_{\sigma \ts_j})= \heartsuit L(0,\bw^{{ij}})+\heartsuit \tt L(\bv^{{ij}},\bw^{{ij}}).
\end{equation}

A variant of Proposition \ref{prop:dominant for gij}(ii) (with $i,j$ switched) shows that the only strongly $l$-dominant pairs $(\bv,\bw)$ are the pairs with
$\bv \in \{0, \bv^{{ij}}, \bv^{i} \}$. Note $\pi(0,\bw)=\cl(0,\bw)$.
Then
\begin{align*}
 \Res_{ \e_{\sigma \ts_i}, \bw^{{ij}} }^{\bw} \big(\cl(0, \bw)\big)
 &=\Res_{ \e_{\sigma \ts_i},\bw^{{ij}} }^{\bw} \big(\pi(0, \bw)\big)\\
&= \heartsuit\pi(0,\e_{\sigma \ts_i} )\boxtimes\pi(0,\bw^{{ij}} ) =\heartsuit\cl(0,\e_{\sigma \ts_i} )\boxtimes\cl(0,\bw^{{ij}}).
\end{align*}

Similarly,  for any nonzero dimension vector $\bv$, if $\cm(\bv,\e_{\sigma \ts_i} )\neq\emptyset$, then $\bv -\e_{\ts_i}\geq0$; if $\cm(\bv,\bw^{{ij}})\neq \emptyset$, then
$\bv-\e_{\ts_i}\geq0$. However, $\bv^{{ij}}(\ts_i)=1$ by Lemma \ref{lem:v-gij}(i). Using \eqref{eqn:d 2}--\eqref{eqn:d 3}, by \eqref{eqn:comultiplication} we have
\begin{align*}
 \Res_{ \e_{\sigma \ts_i}, \bw^{{ij}}}^{\bw} \big(\pi(\bv^{{ij}}, \bw)\big)
&= \heartsuit  \pi(\bv^{{ij}},\e_{\sigma \ts_i} )\boxtimes\pi(0,\bw^{{ij}} ) + \heartsuit \tt\pi(0,\e_{\sigma \ts_i} )\boxtimes\pi(\bv^{{ij}},\bw^{{ij}} )\\
&= \heartsuit \cl(0,\e_{\sigma \ts_i} )\boxtimes\cl(0,\bw^{{ij}} ) + \heartsuit \tt\cl(0,\e_{\sigma \ts_i} )\boxtimes\cl(\bv^{{ij}},\bw^{{ij}} ).
\end{align*}

As $\pi(\bv^{{ij}},\bw)=\cl(\bv^{{ij}},\bw)+a_{\bv^{{ij}},0,w}(\tt)\cl(0,\bw)$, as before we must have $a_{\bv^{{ij}},0,w}(t)=0$.
Then
\begin{align*}
 \Res_{ \e_{\sigma \ts_i}, \bw^{{ij}}}^{\bw} \big(\cl(\bv^{{ij}}, \bw)\big)
= \heartsuit \cl(0,\e_{\sigma \ts_i} )\boxtimes\cl(0,\bw^{{ij}} ) + \heartsuit \tt \cl(0,\e_{\sigma \ts_i} )\boxtimes \cl(\bv^{{ij}},\bw^{{ij}} ).
\end{align*}

Note that $\cm(\bv^{i},\bw^{{ij}})=\cm(\bv^{{ij}},\bw^{{ij}})$ by the natural inclusion. So we obtain
$\pi(\bv^{i},\bw^{{ij}})= \pi(\bv^{{ij}},\bw^{{ij}})$. By \eqref{eqn:comultiplication} and \eqref{eqn:d 4}--\eqref{eqn:d 6}, similarly,
\begin{align*}
 \Res_{ \e_{\sigma \ts_i}, \bw^{{ij}} }^{\bw} \big(\pi(\bv^{i}, \bw)\big)
&= \heartsuit \pi(0,\e_{\sigma \ts_i} )\boxtimes\pi( \bv^{i},\bw^{{ij}} )  +\heartsuit \tt \pi(\bv^{i},\e_{\sigma \ts_i} )\boxtimes\pi( 0,\bw^{{ij}} ) \\
&= \heartsuit \pi(0,\e_{\sigma \ts_i} )\boxtimes\pi(\bv^{{ij}},\bw^{{ij}})+\heartsuit \tt \pi(0,\e_{\sigma \ts_i} )\boxtimes\pi( 0,\bw^{{ij}} )\\
&= \heartsuit \cl(0,\e_{\sigma \ts_i} )\boxtimes\cl(\bv^{{ij}},\bw^{{ij}})+\heartsuit \tt \cl(0,\e_{\sigma \ts_i} )\boxtimes \cl( 0,\bw^{{ij}} ) .
\end{align*}

Assume that $\pi(\bv^{i}, \bw)= \cl(\bv^{i},\bw)+ a_{\bv^{i},\bv^{{ij}}, w}(\tt) \cl(\bv^{{ij}},\bw)+a_{\bv^{i},0, w}(\tt) \cl(0,\bw)$. Again we must have $a_{\bv^{i},\bv^{{ij}}, w}(\tt)=0=a_{\bv^{i},0, w}(\tt)$.
Then
\begin{align*}
 \Res_{ \e_{\sigma \ts_i}, \bw^{{ij}} }^{\bw} \big(\cl(\bv^{i}, \bw)\big)
= \heartsuit \cl(0,\e_{\sigma \ts_i} )\boxtimes\cl(\bv^{{ij}},\bw^{{ij}}) +\heartsuit \tt \cl(0,\e_{\sigma \ts_i} )\boxtimes\cl( 0,\bw^{{ij}} ).
\end{align*}
Therefore, we have obtained
\begin{align}
 L(0,\e_{\sigma \ts_i})    L(0,\bw^{{ij}}) &= \heartsuit L(\bv^{{ij}},\bw) + \heartsuit L(0,\bw)+\heartsuit \tt L(\bv^{i},\bw),\\
 L(0,\e_{\sigma \ts_i})   L(\bv^{{ij}},\bw^{{ij}}) &= \heartsuit L(\bv^{i}, \bw)+\heartsuit \tt L(\bv^{{ij}},\bw).
\end{align}

In an entirely similar fashion, we obtain
\begin{align}
\label{eqn: sisj multiplication2}
 L(0,\e_{\sigma \ts_j})   L(0,\e_{\sigma \ts_i}) &= \heartsuit^{-1}L(0,\bw^{{ij}})+\heartsuit^{-1}\tt^{-1} L(\bv^{{ij}},\bw^{{ij}}),\\
 L(0,\bw^{{ij}})   L(0,\e_{\sigma \ts_i}) &= \heartsuit^{-1}L(0,\bw)+\heartsuit^{-1} L(\bv^{{ij}},\bw)+\heartsuit^{-1}\tt^{-1}L(\bv^{i},\bw),\\
 L(\bv^{{ij}},\bw^{{ij}})   L(0,\e_{\sigma \ts_i}) &= \heartsuit^{-1}\tt^{-1}L(\bv^{{ij}},\bw)+ \heartsuit^{-1} L(\bv^{i}, \bw).
 \label{eq:LL3}
\end{align}
Now the desired formula \eqref{split serre 1} follows by a direct computation from Lemma \ref{lem:cartan and positive} and the identities \eqref{eqn: sisj multiplication}--\eqref{eq:LL3}.

\vspace{2mm}

(2)
\underline{Proof of \eqref{split serre 2}}. It is entirely similar to the above arguments for \eqref{split serre 1}. We shall only record the intermediate steps and formulas.

Denote $\bw'=\bw^{ijj}=\e_{\sigma \ts_i}+2\e_{\sigma \ts_j}$. By Proposition~ \ref{prop:dominant for gij}(ii), the only $l$-dominant pairs $(\bv,\bw')$ are given by $\bv \in \{ {\bf 0}, \bv^{{ij}}, \bv^{j} \}$. By Proposition \ref{prop:bilinear form d}, we have
\begin{align*}
\Res_{ \e_{\sigma \ts_j}, \bw^{{ij}} }^{\bw'} \big(\cl(0, \bw')\big)
&=\Res_{ \e_{\sigma \ts_j},\bw^{{ij}} }^{\bw'} \big(\pi(0, \bw')\big)\\
&=\heartsuit^{-1}\pi(0,\e_{\sigma \ts_j} )\boxtimes\pi(0,\bw^{{ij}} )  =\heartsuit^{-1}\cl(0,\e_{\sigma \ts_j} )\boxtimes\cl(0,\bw^{{ij}});
\\
\Res_{ \e_{\sigma \ts_j}, \bw^{{ij}}}^{\bw'} \big(\cl(\bv^{{ij}}, \bw')\big)
&=\Res_{ \e_{\sigma \ts_j}, \bw^{{ij}}}^{\bw'} \big(\pi(\bv^{{ij}}, \bw')\big)\\
&=\heartsuit^{-1}\tt^{-1}\pi(0,\e_{\sigma \ts_j} )\boxtimes\pi(\bv^{{ij}},\bw^{{ij}} )
= \heartsuit^{-1}\tt^{-1}\cl(0,\e_{\sigma \ts_j} )\boxtimes\cl(\bv^{{ij}},\bw^{{ij}} );
\\
\Res_{ \e_{\sigma \ts_j}, \bw^{{ij}}}^{\bw'} \big(\cl(\bv^{j}, \bw')\big)
&=\Res_{ \e_{\sigma \ts_j}, \bw^{{ij}}}^{\bw'} \big(\pi(\bv^{j}, \bw')\big)\\
&=\heartsuit^{-1}\pi(\bv^{j}-\bv^{{ij}},\e_{\sigma \ts_j} )\boxtimes\pi(\bv^{{ij}},\bw^{{ij}} )
+ \heartsuit^{-1} \tt^{-1} \pi(\bv^{j},\e_{\sigma \ts_j})\boxtimes \pi(0,\bw^{{ij}})\\
&= \heartsuit^{-1}\cl(0,\e_{\sigma \ts_j} )\boxtimes\cl(\bv^{{ij}},\bw^{{ij}} )
+ \heartsuit^{-1} \tt^{-1} \pi(0, \e_{\sigma \ts_j}) \boxtimes \cl(0,\bw^{{ij}}).
\end{align*}
Therefore, we have
\begin{align}
L(0,\e_{\sigma \ts_j})    L(0,\bw^{{ij}})
 &= \heartsuit^{-1}\tt^{-1}L(\bv^{j},\bw' ) + \heartsuit^{-1}L(0,\bw'),
 \label{eq:LL1} \\
L(0,\bw^{{ij}})   L(0,\e_{\sigma \ts_j})
 &=\heartsuit \tt L(\bv^{j},\bw' ) + \heartsuit L(0,\bw'),\\
L(0,\e_{\sigma \ts_j})    L(\bv^{{ij}},\bw^{{ij}})
 &= \heartsuit^{-1}\tt^{-1}L(\bv^{{ij}},\bw' ) +\heartsuit^{-1}L(\bv^{j},\bw' ),\\
L(\bv^{{ij}},\bw^{{ij}})   L(0,\e_{\sigma \ts_j})
 &= \heartsuit \tt L(\bv^{{ij}},\bw' )+\heartsuit L(\bv^{j},\bw' ).
 \label{eq:LL4}
 \end{align}
Now the desired formula \eqref{split serre 2} follows by a direct computation from Lemma \ref{lem:cartan and positive}, the identities \eqref{eqn: sisj multiplication}, \eqref{eqn: sisj multiplication2}, and \eqref{eq:LL1}--\eqref{eq:LL4}.
\end{proof}

\subsection{Computation for rank 2 $\imath$quivers, III}

Let $(Q,\btau)$ be a Dynkin $\imath$quiver.
\begin{proposition}  \label{proposition relaiton for A3}
If $1\xrightarrow{\alpha} 2\xleftarrow{\beta}3$ or $1\xleftarrow{\alpha} 2\xrightarrow{\beta}3$ is a subquiver of $Q$ such that $\btau(1)=3$, and $\btau(2)=2$, then
the following identities hold in $\tRi$, for $i=1,3$:
\begin{align}\label{eqn:serre 1}
L(0,\e_{\sigma \ts_i})^{2}   L(0,\e_{\sigma \ts_2})
& -(\tt+\tt^{-1}) L(0, \e_{\sigma \ts_i}) L(0,\e_{\sigma \ts_2})  L(0,\e_{\sigma \ts_i})
+L(0,\e_{\sigma \ts_2}) L(0,\e_{\sigma \ts_i})^{2} \\
\notag &=0,
\\
\label{eqn:serre 2}
L(0,\e_{\sigma \ts_2})^{2} L(0,\e_{\sigma \ts_i})
&-(\tt+\tt^{-1}) L(0, \e_{\sigma \ts_2}) L(0,\e_{\sigma \ts_i}) L(0,\e_{\sigma \ts_2})
+L(0,\e_{\sigma \ts_i}) L(0,\e_{\sigma \ts_2})^{2} \\
\notag
&=-(\tt-\tt^{-1})^2L(\bv^{2}, \bw^2)   L(0,\e_{\sigma \ts_i}).
\end{align}
\end{proposition}

\begin{proof}
We only prove the statements for the case $1\xleftarrow{\alpha} 2\xrightarrow{\beta}3$, as the other case is similar.
\vspace{2mm}

(1)
\underline{Proof of \eqref{eqn:serre 1}}. Without loss of generality, we assume $i=1$.

Recall that
$\bw^{{12}}=\e_{\sigma \ts_1} +\e_{\sigma \ts_2}$, and $\bw^{{112}}=2\e_{\sigma \ts_1}+\e_{\sigma \ts_2}$.  Note $\tt^{-(1/2) \langle (0,\e_{\sigma \ts_1}), (0,\e_{\sigma \ts_2})\rangle_a}=\tt^{-1/2}.$ In this proof we shall  denote by $\bw=\bw^{{112}}$ and
$\heartsuit= \tt^{-1/2}.$

We have $\pi(0,\e_{\sigma \ts_1}) =\cl(0,\e_{\sigma \ts_1})$, $\pi(0,\e_{\sigma \ts_2})=\cl(0,\e_{\sigma \ts_2})$ and $\pi(0,\bw^{{12}})=\cl(0,\bw^{{12}})$. From Proposition \ref{prop:dominant pairs for gij}, the only $l$-dominant pairs $(\bv,\bw^{{12}})$ are given by
$\bv=0$ and $\bv=\bv^{{12}}$.
\begin{align}
\label{eqn: rank 2 III1}
\Res_{ \e_{\sigma \ts_1}, \e_{\sigma \ts_2} }^{\bw^{{12}}} \big(\cl(0, \bw^{{12}})\big)
&=\Res_{ \e_{\sigma \ts_1}, \e_{\sigma \ts_2} }^{\bw^{{12}}} \big(\pi(0, \bw^{{12}})\big)\\
\notag
&=\heartsuit\pi(0,\e_{\sigma \ts_1} )\boxtimes\pi(0,\e_{\sigma \ts_2} )=\heartsuit\cl(0,\e_{\sigma \ts_1} )\boxtimes\cl(0,\e_{\sigma \ts_2} ).
\end{align}

Note $\cm(\bv, \e_{\sigma \ts_2})\neq \emptyset$ with $\bv\leq \bv^{{12}}$ if and only if $\bv=0$. From \eqref{eqn:comultiplication} and (\ref{eqn:d 1}), we have
\begin{align*}
\Res_{ \e_{\sigma \ts_1}, \e_{\sigma \ts_2} }^{\bw^{{12}}} \big(\pi(\bv^{{12}}, \bw^{{12}})\big)
=&\heartsuit t\pi(\bv^{{12}},\e_{\sigma \ts_1} )\boxtimes\pi(0,\e_{\sigma \ts_2} )
=\heartsuit \tt\cl(0,\e_{\sigma \ts_1} )\boxtimes\cl(0,\e_{\sigma \ts_2} ),
\end{align*}
by noting that $\cm(\bv^{{12}},\e_{\sigma \ts_1} ) = \cm_0(\e_{\sigma \ts_1})=\cm(0,\e_{\sigma \ts_1})$ is just a point. Assume that $\pi(\bv^{{12}}, \bw^{{12}})=\cl(\bv^{{12}}, \bw^{{12}})+ a_{\bv^{{12}},0,\bw^{{12}} }(\tt)\cl(0,\bw^{{12}})$ by \eqref{eqn:decomposition theorem}. Then we see that $a_{\bv^{{12}},0,\bw^{{12}} }(\tt)=0$ since it is bar-invariant. So we have
\begin{align}
\label{eqn: rank 2 III2}
\Res_{ \e_{\sigma \ts_1}, \e_{\sigma \ts_2} }^{\bw^{{12}}} \big(\cl(\bv^{{12}}, \bw^{{12}})\big)=\heartsuit \tt\cl(0,\e_{\sigma \ts_1} )\boxtimes\cl(0,\e_{\sigma \ts_2} ).
\end{align}
Combining \eqref{eqn: rank 2 III1}--\eqref{eqn: rank 2 III2}, we have
\begin{align}
 \label{eq:L00L}
L(0,\e_{\sigma \ts_1})   L(0,\e_{\sigma \ts_2})= \heartsuit L(0,\bw^{{12}})+\heartsuit \tt L(\bv^{{12}}, \bw^{{12}}).
\end{align}

Similarly,
\begin{align}
L(0,\e_{\sigma \ts_2})   L(0,\e_{\sigma \ts_1})&= \heartsuit^{-1}L(0,\bw^{{12}})+\heartsuit^{-1}\tt^{-1}L(\bv^{{12}}, \bw^{{12}}).
\end{align}

The only $l$-dominant pairs $(\bv,\bw)$ are given by
$\bv=0$ and $\bv=\bv^{{12}}$ by Proposition \ref{prop:dominant for gij}(i). Then
\begin{align*}
\Res_{ \e_{\sigma \ts_1}, \bw^{{12}} }^{\bw} \big(\cl(0, \bw)\big)
& =\Res_{ \e_{\sigma \ts_1},\bw^{{12}} }^{\bw} \big(\pi(0, \bw)\big)\\
&=\heartsuit\pi(0,\e_{\sigma \ts_1} )\boxtimes\pi(0,\bw^{{12}} ) =\heartsuit\cl(0,\e_{\sigma \ts_1} )\boxtimes\cl(0,\bw^{{12}}).
\end{align*}
Similarly, by \eqref{eqn:comultiplication} and \eqref{eqn:d 2}--\eqref{eqn:d 3}, we have
\begin{align*}
\Res_{ \e_{\sigma \ts_1}, \bw^{{12}}}^{\bw} \big(\pi(\bv^{{12}}, \bw)\big)
&= \heartsuit  \pi(\bv^{{12}},\e_{\sigma \ts_1} )\boxtimes\pi(0,\bw^{{12}} ) + \heartsuit \tt\pi(0,\e_{\sigma \ts_1} )\boxtimes\pi(\bv^{{12}},\bw^{{12}} )\\
&= \heartsuit \cl(0,\e_{\sigma \ts_1} )\boxtimes\cl(0,\bw^{{12}} ) + \heartsuit \tt\cl(0,\e_{\sigma \ts_1} )\boxtimes\cl(\bv^{{12}},\bw^{{12}} ).
\end{align*}
As $\pi(\bv^{{12}},\bw)=\cl(\bv^{{12}},\bw)+a_{\bv^{{12}},0,w}(\tt)\cl(0,\bw)$,
we get that $a_{\e_{\ts_i},0,w}(\tt)=0$ by studying the fiber of $\cm(\bv^{{12}},\bw)$ over the origin of $\cm_0(\bw)$.
Then
\begin{align*}
\Res_{ \e_{\sigma \ts_1}, \bw^{{12}}}^{\bw} \big(\cl(\bv^{{12}}, \bw)\big)
= \heartsuit \cl(0,\e_{\sigma \ts_1} )\boxtimes\cl(0,\bw^{{12}} ) + \heartsuit \tt\cl(0,\e_{\sigma \ts_1} )\boxtimes\cl(\bv^{{12}},\bw^{{12}} ).
\end{align*}

By definition, it follows that
\begin{align}
L(0,\e_{\sigma \ts_1})   L(0,\bw^{{12}})
&= \heartsuit L(0,\bw)+\heartsuit L(\bv^{{12}},\bw),\\
L(0,\bw^{{12}})   L(0,\e_{\sigma \ts_1})
&= \heartsuit^{-1} L(0,\bw)+\heartsuit^{-1}L(\bv^{{12}},\bw),\\
L(0,\e_{\sigma \ts_1})   L(\bv^{{12}},\bw^{{12}})
&=\heartsuit \tt L(\bv^{{12}},\bw),\\
L(\bv^{{12}},\bw^{{12}})   L(0,\e_{\sigma \ts_1})
&=\heartsuit^{-1}\tt^{-1}L(\bv^{{12}},\bw).
\label{eq:L12L}
\end{align}
Then the desired formula \eqref{eqn:serre 1} follows from \eqref{eq:L00L}--\eqref{eq:L12L} by a direct calculation.

\vspace{2mm}

(2)
\underline{Proof of  \eqref{eqn:serre 2}}. Again we can assume $i=1$. Let us denote $\bw' =\bw^{122} =\e_{\sigma \ts_1}+2\e_{\sigma \ts_2}$. By Proposition~ \ref{prop:dominant for gij}(ii), the only $l$-dominant pairs $(\bv,\bw')$ are given by $\bv \in \{ {\bf 0}, \bv^{{12}}, \bv^{2} \}$. We compute
\begin{align*}
\Res_{ \e_{\sigma \ts_2}, \bw^{{12}} }^{\bw'} \big(\cl(0, \bw')\big)
&=\Res_{ \e_{\sigma \ts_2},\bw^{{12}} }^{\bw'} \big(\pi(0, \bw')\big)\\
&=\heartsuit^{-1}\pi(0,\e_{\sigma \ts_2} )\boxtimes\pi(0,\bw^{{12}} )  =\heartsuit^{-1}\cl(0,\e_{\sigma \ts_2} )\boxtimes\cl(0,\bw^{{12}}).
\end{align*}
From \eqref{eqn:comultiplication} and (\ref{eqn:d 7}), we see that
\begin{align*}
\Res_{ \e_{\sigma \ts_2}, \bw^{{12}}}^{\bw'} \big(\pi(\bv^{{12}}, \bw')\big)
&= \heartsuit^{-1} \tt^{-1} \pi(0,\e_{\sigma \ts_2} )\boxtimes\pi(\bv^{{12}},\bw^{{12}} )\\
&=\heartsuit^{-1}\tt^{-1}\cl(0,\e_{\sigma \ts_2} )\boxtimes\cl(\bv^{{12}},\bw^{{12}} ).
\end{align*}
As $\pi(\bv^{{12}},\bw')=\cl(\bv^{{12}},\bw')+a_{\bv^{{12}},0,\bw'}(\tt)\cl(0,\bw')$,
we have $a_{\bv^{{12}},0,\bw'}(\tt)=0$ by noting that it is bar-invariant.
Then
\begin{align*}
\Res_ {\e_{\sigma \ts_2}, \bw^{{12}}}^{\bw'} \big(\cl(\bv^{{12}}, \bw')\big)
=\heartsuit^{-1}\tt^{-1}\cl(0,\e_{\sigma \ts_2} )\boxtimes \cl(\bv^{{12}},\bw^{{12}} ).
\end{align*}

It follows from \eqref{eqn:comultiplication} and \eqref{eqn:d 8}--\eqref{eqn:d 10} that
\begin{align*}
&\Res_{ \e_{\sigma \ts_2}, \bw^{{12}} }^{\bw'} \big(\pi(\bv^{2}, \bw')\big)\\
&=\heartsuit^{-1}\tt^{-1}\pi(\bv^{2},\e_{\sigma \ts_2} )\boxtimes\pi(0,\bw^{{12}} ) +\heartsuit^{-1} \pi(\bv^{2}-g^{12},\e_{\sigma \ts_2} )\boxtimes\pi(\bv^{{12}},\bw^{{12}} )  \\
&= \heartsuit^{-1}\tt^{-1}\pi(0,\e_{\sigma \ts_2} )\boxtimes\pi(0,\bw^{{12}} ) +\heartsuit^{-1} \pi(0,\e_{\sigma \ts_2} )\boxtimes\pi(\bv^{{12}},\bw^{{12}} ) \\
 &=\heartsuit^{-1}\tt^{-1}\cl(0,\e_{\sigma \ts_2} )\boxtimes\cl(0,\bw^{{12}} ) +\heartsuit^{-1} \cl(0,\e_{\sigma \ts_2} )\boxtimes\cl(\bv^{{12}},\bw^{{12}} ).
\end{align*}

Assume that $\pi(\bv^{2}, \bw')= \cl(\bv^{2},\bw')+ a_{\bv^{2},\bv^{{12}}, \bw'}(\tt) \cl(\bv^{{12}},\bw')+a_{\bv^{2},0, \bw'}(\tt) \cl(0,\bw')$.
Similarly, we can get that $a_{\bv^{2},\bv^{{12}}, \bw'}(\tt) =0= a_{\bv^{2},0, \bw'}(\tt)$.
Then
\begin{align*}
\Res_{ \e_{\sigma \ts_2}, \bw^{{12}} }^{\bw'} \big(\cl(\bv^{2}, \bw')\big)
=\heartsuit^{-1}\tt^{-1}\cl(0,\e_{\sigma \ts_2} )\boxtimes\cl(0,\bw^{{12}} )
+\heartsuit^{-1} \cl(0,\e_{\sigma \ts_2} )\boxtimes\cl(\bv^{{12}},\bw^{{12}} ).
\end{align*}
Therefore, we obtain
\begin{align}
L(0,\e_{\sigma \ts_2})   L(0,\e_{\sigma \ts_1})
&= \heartsuit^{-1}L(0,\bw^{{12}})+\heartsuit^{-1}\tt^{-1}L(\bv^{{12}}, \bw^{{12}})),
  \label{eq:L00L-2}
\\
L(0,\e_{\sigma \ts_1})   L(0,\e_{\sigma \ts_2})
&= \heartsuit L(0,\bw^{{12}})+\heartsuit \tt L(\bv^{{12}}, \bw^{{12}})),
\\
L(0,\e_{\sigma \ts_2})   L(0,\bw^{{12}})
&= \heartsuit^{-1} L(0,\bw')+\heartsuit^{-1}\tt^{-1}L(\bv^{2},\bw),
\\
 L(0,\bw^{{12}})   L(0,\e_{\sigma \ts_2})
 &= \heartsuit L(0,\bw')+\heartsuit \tt L(\bv^{2},\bw),
 \\
L(0,\e_{\sigma \ts_2})   L(\bv^{{12}},\bw^{{12}})
&=\heartsuit^{-1}\tt^{-1}L(\bv^{{12}},\bw')+\heartsuit^{-1}L(\bv^{2},\bw'),
\\
L(\bv^{{12}},\bw^{{12}})   L(0,\e_{\sigma \ts_2})
&=\heartsuit \tt L(\bv^{{12}},\bw')+\heartsuit L(\bv^{2},\bw').
\label{eq:L12L-2}
\end{align}
The desired formula \eqref{eqn:serre 2} follows from \eqref{eq:L00L-2}--\eqref{eq:L12L-2}  by a direct calculation. The proposition is proved.
\end{proof}

\begin{lemma}[cf. \cite{HL15,Qin}]
\label{lem:serre relations for one edge}
For any $i,j\in Q_0$ such that $i,j$ is connected by an arrow and $\btau i\neq i$, $\btau(j)\neq j$, the following identity holds in $\tRi$:
\begin{align*}
L(0,\e_{\sigma \ts_i})^{2}  L(0,\e_{\sigma \ts_j})
-(\tt+\tt^{-1}) L(0, \e_{\sigma \ts_i})   L(0,\e_{\sigma \ts_j}) L(0,\e_{\sigma \ts_i})
+L(0,\e_{\sigma \ts_j})  L(0,\e_{\sigma \ts_i})^{2} =0.
\end{align*}
\end{lemma}

\begin{proof}
The proof is the same as for (\ref{eqn:serre 1}) and will be omitted here.
\end{proof}

%%%%%%%%
%%%%%%%%
\section{Geometric realization of $\imath$quantum groups}
   \label{sec:main}

In this section, we give a quick review of quantum groups $\tU$, $\U$, and $\imath$quantum groups $\tUi$, $\Ui$, associated to the Dynkin diagrams of type ADE. We then develop a new basis and an algebra filtration for the quantum Grothendieck ring $\tRi$. Finally we establish the algebra isomorphism between $\tUi$ and $\tRi$.

As the algebra $\tRi$ is defined over the field $\Q(\tt^{1/2})$, we shall consider the quantum groups and variants over the same field $\Q(\tt^{1/2})$ instead of $\Q(\tt)$.

\subsection{Quantum groups}   \label{subsection Quantum groups}

Let $Q$ be a Dynkin quiver with vertex set $Q_0= \I$.
Let $n_{ij}$ be the number of edges connecting vertex $i$ and $j$. Let $C=(c_{ij})_{i,j \in \I}$ be the Cartan matrix of the underlying graph of $Q$ with $c_{ij}=2\delta_{ij}-n_{ij}.$ Let $\fg$ be the corresponding simple Lie algebra.
%Let $\alpha_i$ ($i\in\I $) be the simple roots of $\fg$, and denote the root lattice by $\Z^{\I}:=\Z\alpha_1\oplus\cdots\oplus\Z\alpha_n$. The {\em simple reflection} $s_i:\Z^{\I}\rightarrow\Z^{\I}$ is defined to be $s_i(\alpha_j)=\alpha_j-c_{ij}\alpha_i$, for $i,j\in \I$.
%Denote the Weyl group by $W =\langle s_i\mid i\in \I\rangle$.

%Let $\btau$ be an involution of $Q$. We assume that $c_{i,\btau i}=0$ for all $i$, which always hold  for the {\em Dynkin} $\imath$quivers. We denote by $\bs_{i}$ the following element of order 2 in the Weyl group $W$, i.e.,
%\begin{align}
%\label{def:simple reflection}
%\bs_i= \left\{
%\begin{array}{ll}
%s_{i}, & \text{ if } \btau i=i;
%\\
%s_is_{\btau i}, & \text{ if } \btau i\neq i.
%\end{array}
%\right.
%\end{align}
%It is well known (cf., e.g., \cite{KP11}) that the {\rm restricted Weyl group} associated to the quasi-split symmetric pair $(\fg, \fg^\theta)$ can be identified with the following subgroup $W_\btau$ of $W$:
%\begin{align}
%  \label{eq:Wtau}
%W_{\btau} =\{w\in W\mid \btau w =w \btau\}
%\end{align}
%where $\btau$ is regarded as an automorphism of $\Aut(C)$. Moreover, it admits the following property.

Let $\tt$ be an indeterminate. Write $[A, B]=AB-BA$. Denote, for $r,m \in \N$,
\[
 [r]=\frac{\tt^r-\tt^{-r}}{\tt-\tt^{-1}},
 \quad
 [r]!=\prod_{i=1}^r [i], \quad \qbinom{m}{r} =\frac{[m][m-1]\ldots [m-r+1]}{[r]!}.
\]
Then the Drinfeld double $\tU := \tU_\tt(\fg)$ is defined to be the $\Q(\tt^{1/2})$-algebra generated by $E_i,F_i, \tK_i,\tK_i'$, $i\in \I$, %\red{where $\tK_i, \tK_i'$ are invertible,?}
 subject to the following relations:  for $i, j \in \I$,
\begin{align}
[E_i,F_j]= \delta_{ij} \frac{\tK_i-\tK_i'}{\tt-\tt^{-1}},  &\qquad [\tK_i,\tK_j]=[\tK_i,\tK_j']  =[\tK_i',\tK_j']=0,
\label{eq:KK}
\\
\tK_i E_j=\tt^{c_{ij}} E_j \tK_i, & \qquad \tK_i F_j=\tt^{-c_{ij}} F_j \tK_i,
\label{eq:EK}
\\
\tK_i' E_j=\tt^{-c_{ij}} E_j \tK_i', & \qquad \tK_i' F_j=\tt^{c_{ij}} F_j \tK_i',
 \label{eq:K2}
\end{align}
 and for $i\neq j \in \I$,
\begin{align}
& \sum_{r=0}^{1-c_{ij}} (-1)^r \left[ \begin{array}{c} 1-c_{ij} \\r \end{array} \right]  E_i^r E_j  E_i^{1-c_{ij}-r}=0,
  \label{eq:serre1} \\
& \sum_{r=0}^{1-c_{ij}} (-1)^r \left[ \begin{array}{c} 1-c_{ij} \\r \end{array} \right]  F_i^r F_j  F_i^{1-c_{ij}-r}=0.
  \label{eq:serre2}
\end{align}
Note that $\tK_i \tK_i'$ are central in $\tU$ for all $i$. %\brown{Consider the necessity for $K_i,K_i'$ to be invertible.}
The comultiplication $\Delta: \widetilde{\U} \rightarrow \widetilde{\U} \otimes \widetilde{\U}$ is given by
\begin{align}  \label{eq:Delta}
\begin{split}
\Delta(E_i)  = E_i \otimes 1 + \tK_i \otimes E_i, & \quad \Delta(F_i) = 1 \otimes F_i + F_i \otimes \tK_{i}', \\
 \Delta(\tK_{i}) = \tK_{i} \otimes \tK_{i}, & \quad \Delta(\tK_{i}') = \tK_{i}' \otimes \tK_{i}'.
 \end{split}
\end{align}

Analogously as for $\tU$, the quantum group $\bU$ is defined to be the $\Q(\tt^{1/2})$-algebra generated by $E_i,F_i, K_i, K_i^{-1}$, $i\in \I$, subject to the  relations modified from \eqref{eq:KK}--\eqref{eq:serre2} with $\tK_i$ and $\tK_i'$ replaced by $K_i$ and $K_i^{-1}$, respectively. The comultiplication $\Delta$ is obtained by modifying \eqref{eq:Delta} with $\tK_i$ and $\tK_i'$ replaced by $K_i$ and $K_i^{-1}$, respectively.

%\blue{Let $\bvs=(\vs_i)\in  (\Q(\tt^{1/2})^\times)^{\I}$. Up to a base change to the field $\Q(\tt^{1/2})(\sqvs_i \mid i\in \I)$, the algebra $\U$ is isomorphic to a quotient algebra of $\tU$ by the ideal $( \tK_i \tK_i'- \vs_i \mid \forall i\in \I )$, by sending $F_i \mapsto F_i, E_i \mapsto \sqvs_i^{-1} E_i, K_i \mapsto \sqvs_i^{-1} \tK_i, K_i^{-1} \mapsto \sqvs_i^{-1} K_i'$. }
%\green{No need?}

Let $\widetilde{\bU}^+$ be the subalgebra of $\widetilde{\bU}$ generated by $E_i$ $(i\in \I)$, $\widetilde{\bU}^0$ be the subalgebra of $\widetilde{\bU}$ generated by $\tK_i, \tK_i'$ $(i\in \I)$, and $\widetilde{\bU}^-$ be the subalgebra of $\widetilde{\bU}$ generated by $F_i$ $(i\in \I)$, respectively.
The subalgebras $\bU^+$, $\bU^0$ and $\bU^-$ of $\bU$ are defined similarly. Then both $\widetilde{\bU}$ and $\bU$ have triangular decompositions:
\begin{align*}
\widetilde{\bU} =\widetilde{\bU}^+\otimes \widetilde{\bU}^0\otimes\widetilde{\bU}^-,
\qquad
\bU &=\bU^+\otimes \bU^0\otimes\bU^-.
\end{align*}
Clearly, ${\bU}^+\cong\widetilde{\bU}^+$, ${\bU}^-\cong \widetilde{\bU}^-$, and ${\bU}^0 \cong \widetilde{\bU}^0/(\tK_i \tK_i' -1 \mid   i\in \I)$.

\subsection{Theorems of Hernandez-Leclerc and Qin}  %Realization of quantum groups}
  \label{subsec:Gring}

Let $\mcr, \cs$ and $\cp$ be the regular, singular, and preprojective NKS categories defined in Definition \ref{def:RS for QG} and  Proposition \ref{prop:2-complexes}. In this case, $\Ind \cp=\Ind \mod(kQ)\bigsqcup \Ind \Sigma\mod(kQ)$.

For any dimension vector $\bf{u}:\mcr_0\rightarrow\N$, the dimension vector $\Sigma^*\bf u$ is such that
\[
(\Sigma^*{\bf u}) (z)={\bf u}(\Sigma z), \quad \forall z\in\mcr_0.
\]
Recalling $V^+, W^+$ from \eqref{eq:VW+}, we define
\begin{align}
&{W}^- =\bigoplus_{x\in\{\ts_i,i\in Q_0\}}\N \e_{\sigma \Sigma x},
&&{V}^- =\bigoplus_{x\in\Ind \mod(kQ),\, x\text{ is not injective}} \N \e_{\Sigma x}.
\end{align}
(Informally, we have ${W}^-=\Sigma^*{W}^+, {V}^-=\Sigma^* {V}^+$.)
Recall $\bv^i, \bw^i$ in \eqref{eq:vfi} used in $\imath$NKS categories $\mcr^\imath, \cs^\imath$. For use in the NKS categories $\mcr, \cs$ defined in \S\ref{subsec:NKS}, we shall denote, for $i\in Q_0$,
\begin{align}
\bar{\bw}^{i}=\e_{\sigma \ts_i}+\e_{\sigma \Sigma \ts_i},
&\qquad
\bar{\bv}^{i}=\sum_{z\in\mcr_0-\cs_0} \dim\Hom_{\cp}(\ts_i, z)\e_z,
\qquad \bar{\bv}^{\Sigma i}=\Sigma^* \bar{\bv}^{i}.
%\\
%\bar{W}^0=\bigoplus_{i\in Q_0} \N\bar{\bw}^{i},
%&\qquad
% \bar{V}^0=\bigoplus_{i\in Q_0} (\N \bar{\bv}^{i}\oplus \N \bar{\bv}^{\Sigma i}).
\end{align}

The quantum Grothendieck ring $K^{gr}(\mod(\cs))$ in \eqref{eq:Kgr} in the current setting reads as follows:
\[
K^{gr}(\mod(\cs))=\bigoplus_{\bw\in{W}^+\oplus {W}^-} K_\bw(\mod(\cs)).
\]
Recalling the Euler form $\langle \cdot, \cdot \rangle$ on $K_0(\mod (kQ))$, define a bilinear form $\langle \cdot, \cdot \rangle_{\bar{a}}$ on $K_0(\mod (kQ))\oplus K_0(\Sigma \mod (kQ))$ as follows:
for $x=(x_1,x_2)$, $y=(y_1,y_2)\in K_0(\mod (kQ))\oplus K_0(\Sigma \mod (kQ))$, we let
\begin{align}
&\langle x,y\rangle_{\bar{a}}=\langle x_1,y_1\rangle-\langle y_1,x_1\rangle+\langle x_2,y_2\rangle-\langle y_2,x_2\rangle.\label{eqn: antisymmetric bilinear form 2}
%&(x,y)=\langle x_1,y_1\rangle+\langle y_1,x_1\rangle+\langle x_2,y_2\rangle+\langle y_2,x_2\rangle.
\end{align}

Define the twisted comultiplication $\Res^{\bw}_{\bw^1,\bw^2}: = t^{-\frac12\langle \bw^1,\bw^2\rangle_{\bar{a}}}  \widetilde{\Res}^{\bw}_{\bw^1,\bw^2}$. Then $\{ \Delta^\bw_{\bw^1,\bw^2} \}$
gives rise to a comultiplication $\Delta$ on $K^{gr}(\mod(\cs))$, which induces an algebra structure on the $\Z[\tt^{1/2}, \tt^{-1/2}]$-graded dual $\tRZ :=K^{gr}(\mod(\cs))$; cf. \eqref{eq:Kgr2} and compare \eqref{eq:tRi}. That is,
\begin{align}
\tRZ=\bigoplus_{\bw\in W^{+}}  \tR_{\Z,\bw},
\quad \text{where } \tR_{\Z,\bw}  &=   \Hom_{\Z[\tt, \tt^{-1}]} \big(K_\bw(\mod(\cs)),\Z[\tt^{1/2}, \tt^{-1/2}]\big),
\label{eq:tR2}
\end{align}
%in practice, the product sign $\cdot$ is often omitted.
The algebra $\tRZ$ has the following 3 distinguished subalgebras:
\vspace{2mm}

$\triangleright$ $\tRZ^+=\bigoplus_{\bw\in {W}^+}  \tR_{\Z,\bw}$;

$\triangleright$ $\tRZ^-=\bigoplus_{\bw\in {W}^-}  \tR_{\Z,\bw}$;

$\triangleright$ the subalgebra $\tRZ^0$ generated by $\{ L(\bar{\bv}^{i},\bar{\bw}^{i}), L(\bar{\bv}^{\Sigma i},\bar{\bw}^{i})\mid i\in \I \}.$
\vspace{2mm}

\noindent We further denote
\[
\tR = \Q(\tt^{1/2}) \otimes \tRZ,
\qquad \text{ and }\quad
 \tR^{\star} = \Q(\tt^{1/2}) \otimes \tRZ^\star, \text{ for } \star \in \{+,-,0\}.
\]

%From \cite{Qin}, we can obtain that $\tR^0$ is a right Ore, right reversible and multiplicatively closed subset of $\tR$.
%Let
%\begin{align}
%\label{localization quantum}
%\LtR:=\tR[L(\bar{\bv}^{i},\bar{\bw}^{i})^{-1}, L(\bar{\bv}^{\Sigma i},\bar{\bw}^{i})^{-1},i\in\I ]
%\end{align}
%by the localization of $\tR$ with respect to $\tR^0$.

%We shall use $(\cdot )_{\tt^{1/2}}$ and $(\cdot )_{\Q(\tt^{1/2})}$ to denote the base changes $(\cdot )\otimes \Z(\tt^{1/2})$ and $(\cdot )\otimes\Q(\tt^{1/2})$.

The theorem below summarizes the main result in \cite{Qin}, which is a generalization of a fundamental result from \cite{HL15}.
\begin{theorem}    \label{thm:iso HL-Q}
{\quad}
\begin{enumerate}
\item \cite{HL15}
There exists a $\Q(\tt^{1/2})$-algebra isomorphism $\kappa^+: \bU^+ \longrightarrow \tR^+$ which sends
$E_i \mapsto \frac{\tt}{1-\tt^2} L(0,\e_{\sigma \ts_i})$ for each $i\in Q_0$.

\item \cite{Qin}
There exists a $\Q(\tt^{1/2})$-algebra isomorphism
$\widetilde{\kappa}: {\tU}\longrightarrow  \tR$ which extends $\kappa^+$ above and sends
\begin{align}
E_i&\mapsto \frac{\tt}{1-\tt^2} L(0,\e_{\sigma \ts_i}),&
F_i& \mapsto \frac{\tt}{\tt^2-1}L(0,\e_{\sigma \Sigma \ts_i}) ,\\
\tK_i&\mapsto L(\bar{\bv}^{\Sigma i},\bar{\bw}^{i}),&\tK'_i&\mapsto L(\bar{\bv}^{i},\bar{\bw}^{i}).
\end{align}
\end{enumerate}
\end{theorem}

Let  ${\bf R}$ denote the quotient algebra of $\tR$ by the ideal generated by
$L(\bar{\bv}^{i}+\bar{\bv}^{\Sigma i},2\bar{\bw}^{i})-1$, for all $i\in\I.$
Then $\widetilde{\kappa}$ in Theorem~\ref{thm:iso HL-Q}(2) induces a $\Q(\tt^{1/2})$-algebra isomorphism $\bU\cong \bf R$.

\subsection{The $\imath$quantum groups}

For the Cartan matrix $C=(c_{ij})$, let $\Aut(C)$ be the group of all permutations $\btau$ of the set $\I$ such that $c_{ij}=c_{\btau i,\btau j}$. An element $\btau\in\Aut(C)$ is called an \emph{involution} if $\btau^2=\Id$. We define the {\em universal $\imath$quantum group} $\widetilde{\bU}^\imath$ to be the $\Q(\tt^{1/2})$-subalgebra of $\tU$ generated by
\begin{equation}
  \label{eq:Bi}
B_i= F_i +  E_{\btau i} \tK_i',
\qquad \tk_i = \tK_i \tK_{\btau i}', \quad \forall i \in \I.
\end{equation}
Let $\tU^{\imath 0}$ be the $\Q(\tt^{1/2})$-subalgebra of $\tUi$ generated by $\tk_i$, for $i\in \I$.
By \cite[Lemma 6.1]{LW19}, the elements $\tk_i$ (for $i= \btau i$) and $\tk_i \tk_{\btau i}$  (for $i\neq \btau i$) are central in $\tUi$.

Let $\bvs=(\vs_i)\in  (\Q(\bv)^\times)^{\I}$ be such that $\vs_i=\vs_{\btau i}$ for each $i\in \I$ which satisfies $c_{i, \btau i}=0$.
Let $\Ui:=\Ui_{\bvs}$ be the $\Q(\tt^{1/2})$-subalgebra of $\bU$ generated by
\[
B_i= F_i+\vs_i E_{\btau i}K_i^{-1},
\quad
k_j= K_jK_{\btau j}^{-1},
\qquad  \forall i \in \I, j \in \ci.
\]
It is known \cite{Let99, Ko14} that $\bU^\imath$ is a right coideal subalgebra of $\bU$ in the sense that $\Delta: \Ui \rightarrow \Ui\otimes \U$; and $(\bU,\Ui)$ is called a \emph{quantum symmetric pair}. %, as they specialize at $\tt=1$ to $(U(\fg), U(\fg^\theta))$, where $\theta=\omega \circ \btau$, and $\theta$ is understood here as an automorphism of $\fg$.

The algebras $\Ui_{\bvs}$, for $\bvs\in  (\Q(\tt)^\times)^{\I}$, are obtained from $\tUi$ by central reductions.

\begin{proposition} [\text{\cite[Proposition 6.2]{LW19}}]
  \label{prop:QSP12}
(1) The algebra $\Ui$ is isomorphic to the quotient of $\tUi$ by the ideal generated by
\begin{align}   \label{eq:parameters}
\tk_i - \vs_i \; (\text{for } i =\btau i),
\qquad  \tk_i \tk_{\btau i} - \vs_i \vs_{\btau i}  \;(\text{for } i \neq \btau i).
\end{align}
The isomorphism is given by sending $B_i \mapsto B_i, k_j \mapsto \vs_{\btau j}^{-1} \tk_j, k_j^{-1} \mapsto \vs_j^{-1} \tk_{\btau j},  \forall i\in \I, j\in \I\backslash\ci$.

(2) The algebra $\widetilde{\bU}^\imath$ is a right coideal subalgebra of $\widetilde{\bU}$; that is, $(\widetilde{\bU}, \widetilde{\bU}^\imath)$ forms a quantum symmetric pair.
\end{proposition}

We shall refer to $\tUi$ and $\Ui$ as {\em (quasi-split) $\imath${}quantum groups}; they are called {\em split} if $\btau =\Id$.

\begin{remark}
Following \cite{Qin}, we do not require $K_i, K_i'$ to be invertible in the definition of the Drinfeld double $\tU$ in this paper, in contrast to the conventions adopted in \cite{Br13, LW19}. Similarly, the above definition of $\tUi$ does not require $\tk_i$ to be invertible, in contrast to \cite{LW19}.
The definition of semi-derived Ringel-Hall algebras can be suitably revised so the resulting $\imath$Hall algebras are isomorphic to the current versions of $\tU$ and $\tUi$; this will be explained in detail elsewhere.
\end{remark}

Let us choose the subset $\ci$ of representatives of $\btau$-orbits on $\I$ as follows:
\begin{align}
\label{eqn:representative}
&\ci :=\left\{ \begin{array}{cl} \I, & \text{ if }\btau=\Id,\\
\left.\begin{array}{cl}
%\{1,\dots,r\}, & \text{ if } \Delta \text{ is of type }A_{2r},\\
\{0,1,\dots,r\}, & \text{ if } \Delta \text{ is of type }A_{2r+1},\\
\{1,\dots,n-1\}, & \text{ if } \Delta \text{ is of type }D_n,\\
\{1,2,3,4\},  & \text{ if } \Delta \text{ is of type }E_6,\end{array} \right\}
& \text{ if }\btau\neq\Id.\end{array}\right.
\end{align}

The following quivers are the bound quivers used to describe $\Lambda^\imath$ or equivalently $\cs^\imath$ associated to non-split $\imath$quivers; cf. \cite{LW19}. The Dynkin diagrams for the non-split Dynkin $\imath$quivers can be recovered from diagrams below by removing the loops and 2-cycles.

%%%A
\begin{center}\setlength{\unitlength}{0.7mm}
\vspace{-1.5cm}
\begin{equation}
\label{diag: A}
\begin{picture}(100,40)(0,20)
\put(0,10){$\circ$}
\put(0,30){$\circ$}
\put(50,10){$\circ$}
\put(50,30){$\circ$}
\put(72,10){$\circ$}
\put(72,30){$\circ$}
\put(92,20){$\circ$}
\put(0,6){$r$}
\put(-2,34){${-r}$}
\put(50,6){\small $2$}
\put(48,34){\small ${-2}$}
\put(72,6){\small $1$}
\put(70,34){\small ${-1}$}
\put(92,16){\small $0$}

\put(3,11.5){\vector(1,0){16}}
\put(3,31.5){\vector(1,0){16}}
\put(23,10){$\cdots$}
\put(23,30){$\cdots$}
\put(33.5,11.5){\vector(1,0){16}}
\put(33.5,31.5){\vector(1,0){16}}
\put(53,11.5){\vector(1,0){18.5}}
\put(53,31.5){\vector(1,0){18.5}}
\put(75,12){\vector(2,1){17}}
\put(75,31){\vector(2,-1){17}}
%
%\color{purple}
\put(0,13){\vector(0,1){17}}
\put(2,29.5){\vector(0,-1){17}}
\put(50,13){\vector(0,1){17}}
\put(52,29.5){\vector(0,-1){17}}
\put(72,13){\vector(0,1){17}}
\put(74,29.5){\vector(0,-1){17}}

\put(-5,20){$\varepsilon_r$}
\put(3,20){$\varepsilon_{-r}$}
\put(45,20){\small $\varepsilon_2$}
\put(53,20){\small $\varepsilon_{-2}$}
\put(67,20){\small $\varepsilon_1$}
\put(75,20){\small $\varepsilon_{-1}$}
\put(92,30){\small $\varepsilon_0$}

\qbezier(93,23)(90.5,25)(92,27)
\qbezier(92,27)(94,30)(97,27)
\qbezier(97,27)(98,24)(95.5,22.6)
\put(95.6,23){\vector(-1,-1){0.3}}
\end{picture}
\end{equation}
\vspace{-0.2cm}
\end{center}

%%%D
\begin{center}\setlength{\unitlength}{0.8mm}
\begin{equation}
\label{eqn: Dn}
 \begin{picture}(100,25)(-5,0)
\put(0,-1){$\circ$}
\put(0,-5){\small$1$}
\put(20,-1){$\circ$}
\put(20,-5){\small$2$}
\put(64,-1){$\circ$}
\put(84,-10){$\circ$}
\put(80,-13){\small${n-1}$}
\put(84,9.5){$\circ$}
\put(84,12.5){\small${n}$}

\put(19.5,0){\vector(-1,0){16.8}}
\put(38,0){\vector(-1,0){15.5}}
\put(64,0){\vector(-1,0){15}}

\put(40,-1){$\cdots$}
\put(83.5,9.5){\vector(-2,-1){16}}
\put(83.5,-8.5){\vector(-2,1){16}}
%
%\color{purple}
\put(86,-7){\vector(0,1){16.5}}
\put(84,9){\vector(0,-1){16.5}}

\qbezier(63,1)(60.5,3)(62,5.5)
\qbezier(62,5.5)(64.5,9)(67.5,5.5)
\qbezier(67.5,5.5)(68.5,3)(66.4,1)
\put(66.5,1.4){\vector(-1,-1){0.3}}
\qbezier(-1,1)(-3,3)(-2,5.5)
\qbezier(-2,5.5)(1,9)(4,5.5)
\qbezier(4,5.5)(5,3)(3,1)
\put(3.1,1.4){\vector(-1,-1){0.3}}
\qbezier(19,1)(17,3)(18,5.5)
\qbezier(18,5.5)(21,9)(24,5.5)
\qbezier(24,5.5)(25,3)(23,1)
\put(23.1,1.4){\vector(-1,-1){0.3}}

\put(-1,9.5){$\varepsilon_1$}
\put(19,9.5){$\varepsilon_2$}
\put(59,9.5){$\varepsilon_{n-2}$}
\put(79,-1){$\varepsilon_{n}$}
\put(87,-1){$\varepsilon_{n-1}$}
\end{picture}
\end{equation}
\vspace{.8cm}
\end{center}

%%%%E
\begin{center}\setlength{\unitlength}{0.8mm}
\vspace{-2.5cm}
\begin{equation}
\label{eqn: E6}
 \begin{picture}(100,40)(0,20)
\put(10,6){\small${6}$}
\put(10,10){$\circ$}
\put(32,6){\small${5}$}
\put(32,10){$\circ$}

\put(10,30){$\circ$}
\put(10,33){\small${1}$}
\put(32,30){$\circ$}
\put(32,33){\small${2}$}

\put(31.5,11){\vector(-1,0){19}}
\put(31.5,31){\vector(-1,0){19}}

\put(52,22){\vector(-2,1){17.5}}
\put(52,20){\vector(-2,-1){17.5}}

\put(54.7,21.2){\vector(1,0){19}}

\put(52,20){$\circ$}
\put(52,16.5){\small$3$}
\put(74,20){$\circ$}
\put(74,16.5){\small$4$}
%
%\color{purple}
\put(10,12.5){\vector(0,1){17}}
\put(12,29.5){\vector(0,-1){17}}
\put(32,12.5){\vector(0,1){17}}
\put(34,29.5){\vector(0,-1){17}}

\qbezier(52,22.5)(50,24)(51,26.5)
\qbezier(51,26.5)(53,29)(56,26.5)
\qbezier(56,26.5)(57.5,24)(55,22)
\put(55.1,22.4){\vector(-1,-1){0.3}}
\qbezier(74,22.5)(72,24)(73,26.5)
\qbezier(73,26.5)(75,29)(78,26.5)
\qbezier(78,26.5)(79,24)(77,22)
\put(77.1,22.4){\vector(-1,-1){0.3}}

\put(35,20){$\varepsilon_2$}
\put(27,20){$\varepsilon_5$}
\put(13,20){$\varepsilon_1$}
\put(5,20){$\varepsilon_6$}
\put(52,30){$\varepsilon_3$}
\put(73,30){$\varepsilon_4$}
\end{picture}
\end{equation}
\vspace{1cm}
\end{center}

\begin{proposition}  [\text{\cite[Proposition~ 6.4]{LW19}}]
  \label{prop:Serre}
Let $(Q, \btau)$ be a Dynkin $\imath$quiver. The $\Q(\tt^{1/2})$-algebra $\tUi$ has a presentation with generators $B_i, \tk_i$ $(i\in \I)$, where $\tk_i$ are invertible, subject to the relations \eqref{relation1}--\eqref{relation2}: for $l \in \I$, and $i\neq j \in \I$,
\begin{align}
\tk_i \tk_l =\tk_l \tk_i,
\quad
\tk_l B_i & = \tt^{c_{\btau (l) i} -c_{l i}} B_i \tk_l,
%K_\mu B_i-q_i^{-\langle \mu,\alpha_i\rangle} B_iK_\mu & =0,
   \label{relation1}
\\
B_{\btau i}B_i -B_i B_{\btau i}& =   \frac{\tk_i -\tk_{\btau i}}{\tt-\tt^{-1}},
%\sum_{s=0}^{1-c_{i,\DTr i}} (-1)^s B_i^{(s)}B_{\DTr i}B_i^{(1-c_{i,\DTr i}-s)}& =\frac{1}{q_i-q_i^{-1}}
%  \\
% \cdot \left(q_i^{c_{i,\DTr i}} (q_i^{-2};q_i^{-2})_{-c_{i,\DTr i}} \vs_{\DTr i}B_i^{(-c_{i,\DTr i})} \tK_i \tK_{\DTr i}^{-1} \right.  & \left. -(q_i^{2};q_i^{2})_{-c_{i,\tau i}}\vs_{i}B_i^{(-c_{i,\tau i})} \tK_{\tau i} \tK_i^{-1} \right),
\quad \text{ if } \btau i \neq i,
\label{relation5}   %implies c_{i, \tau i}=0 for Dynkin iquiver
\\
B_iB_{j}-B_jB_i &=0, \quad \text{ if } c_{ij} =0 \text{ and }\btau i\neq j,
 \label{relationBB}
\\
\sum_{s=0}^{1-c_{ij}} (-1)^s \qbinom{1-c_{ij}}{s} B_i^{s}B_jB_i^{1-c_{ij}-s} &=0, \quad \text{ if } j \neq \btau i\neq i,\label{relation3}
\\
B_i^2B_{j} - [2] B_iB_{j}B_i +B_{j}B_i^2 &= \tt \tk_i B_{j},
%\sum_{s=0}^{1-c_{ij}} (-1)^s  B_{i,\overline{c_{ij}}+\overline{p_i}}^{(s)}B_j B_{i,\overline{p}_i}^{(1-c_{ij}-s)} &=0,\quad
%
\quad \text{ if }  c_{ij} = -1 \text{ and }\btau i=i.
\label{relation2}
\end{align}
\end{proposition}

\subsection{Decomposition of $\tRi$}

Recall $I_i$ denotes the injective $kQ$-module for each $i\in Q$.

\begin{lemma}  [Connecting triangle]\cite[Page 51]{Ha2}
\label{lem: connecting triangle}
For any indecomposable projective $kQ$-module $P_i\in\Ind \proj(kQ)\subseteq \cd_Q$, the AR triangle ending to $P_i$ is
of the form
$$\Sigma^{-1}I_i\longrightarrow \Sigma^{-1} I \oplus P\longrightarrow P_i\longrightarrow I_i,$$
for some injective module $I$ and some projective module $P$.
\end{lemma}

Similarly, for any indecomposable injective $kQ$-module $I_i$, the AR triangle starting from $I_i$ is of the form
$$I_i\longrightarrow  I \oplus \Sigma P\longrightarrow \Sigma P_i\longrightarrow \Sigma I_i,$$
where $I$ is some injective module and $P$ is some projective module.

Moreover, for any $M\in \Ind\mod(kQ)\subseteq \cd_Q$, $M$ is not projective if and only if $\tau M\in\mod(kQ)$. In this case, the AR triangle ending to $M$ is constructed from the almost split sequence ending to $M$ in $\mod(kQ)$, see \cite[Page 50]{Ha2}.

%For any $\bw\in {W}^+$, $\Phi(\bw)$ can be viewed as a $\C Q$-module. We define $\deg \Phi(\bw)$ to be the total dimension of $\Phi(\bw)$ and the bilinear form
%\begin{align}
%N(\Phi(\bw))=(\Phi(\bw),\Phi(\bw))-\deg\Phi(\bw).
%\end{align}

%Let $B_L^*=\{B_L^*(\bw)\mid w\in {W}^S\}$ be the dual canonical basis of $\bU_t(\fn^+)$ with respect to Lusztig's bilinear form $(\cdot,\cdot)_L$. The rescaled dual canonical basis is defined to be $\widetilde{B}_L^*=\{\widetilde{B}_L^*(\bw)\mid w\in {W}^S\}$ such that
%\begin{align}
%\widetilde{B}^*_L(\bw)=t^{(1/2)N(\Phi(\bw))-\deg \Phi(\bw)} B^*_L(\bw).
%\end{align}

Recall $\mcr$ and $\cs$ denote the regular and singular NKS categories defined in Definition \ref{def:RS for QG}, and $\mcr^\imath$ and $\cs^\imath$ denote the regular and singular $\imath$NKS categories. Recall $V^{+}, W^{+}, V^0, W^0$ from \eqref{eq:VW+}. A pair in $(V^{+},W^{+})$ can be naturally viewed as a dimension vector for both $\mcr$ and $\mcr^\imath$. When we need to separate the notations for $\mcr$ and $\mcr^\imath$, we adopt the following convention:
\begin{align}
\label{eq:vwvw}
& \text{For any dimension vector }(\bv, \bw) \in (V^{+}, W^{+}) \text{ for }\mcr^\imath,
 \\
& \qquad\qquad  (\bar{\bv}, \bar{\bw}) \text{ denotes the corresponding dimension vector for $\mcr$. }
  \notag
\end{align}

\begin{lemma}
   \label{lem:sameDP}
A dimension vector $(\bv,\bw)\in(V^{+},W^{+})$ is an $l$-dominant pair for $\mcr^\imath$ if and only if it is an $l$-dominant pair for $\mcr$.
\end{lemma}

\begin{proof}
Below the map defined in \eqref{def:Cq} for $\mcr$ is denoted by $\ov{\mathcal C}_q$, and for $\mcr^\imath$ is denoted by ${\mathcal C}_q$.

First, we prove the ``only if'' part. For $x\in\Ind \mod (kQ)\bigsqcup \Sigma\Ind \mod(kQ)\subseteq \Ind\cd_Q$, let
$$\tau x\longrightarrow N_x\longrightarrow x\longrightarrow \Sigma \tau x$$
be the AR triangle in $\cp$. we proceed by separating into four cases (i)-(iv) below.

(i) \underline{$x,\tau x\in\mod(kQ)$}. Then we have $(\sigma^*\bar{\bw}-\ov{\mathcal C}_q \bar{\bv})(x)=(\sigma^*\bw-{\mathcal C}_q \bv)(x)\geq0$.

(ii) \underline{$x\in\mod(kQ)$, $\tau x\notin \mod(kQ)$}. Then $x$ is projective. It follows by Lemma \ref{lem: connecting triangle} that
\begin{align*}
(\sigma^*\bar{\bw}-\ov{\mathcal C}_q \bar{\bv})(x)
& =\sigma^*\bar{\bw}(x)+ \bar{\bv}(N_x)-\bar{\bv}(x)\\
&= \sigma^*\bw(x)+\bv(P_{N_x}) -\bv(x)
=(\sigma^*\bw-{\mathcal C}_q\bv)(x)\geq 0.
\end{align*}
Here $P_{N_x}$ is the maximal direct summand of $N_x$ belonging to $\mod(kQ)$.

(iii) \underline{$\tau x\in\mod (kQ),x\notin \mod(kQ)$}. Then $\tau x$ is injective. So $(\sigma^*\bar{\bw}-\ov{\mathcal C}_q\bar{\bv})(x)= (\sigma^* \bw)(x)=0$ thanks to $x\in\Sigma\Ind\mod (kQ)$.

(iv) \underline{$x,\tau x\notin \mod (kQ)$}. Then $(\sigma^*\bar{\bw}-\ov{\mathcal C}_q \bar{\bv})(x)=0$ by definition.

\noindent Therefore $(\bar{\bv},\bar{\bw})$ is an $l$-dominant pair for $\mcr$.

Now we prove the ``if'' part. For any $x\in\Ind \mod (kQ)$, let
$$\tau x\longrightarrow N_x\longrightarrow x\longrightarrow \Sigma \tau x$$
be the AR triangle in $\cp^\imath$.
If $x$ is not projective, then $(\sigma^*\bw-{\mathcal C}_q \bv)(x)=(\sigma^*\bar{\bw}-\ov{\mathcal C}_q \bar{\bv})(x)\geq0$.
If $x$ is projective, then $\tau x$ is injective in $\cp^\imath$, and $\tau x\in\Sigma^{-1}\Ind \inj(kQ)$ when considered in $\cp$. It follows that
\begin{align*}
(\sigma^*\bw-{\mathcal C}_q \bv)(x)
& =\sigma^*\bar{\bw}(x)-\bv(x)+\bv(N_x)
\\
& =\sigma^*\bar{\bw}(x)-\bv(x)+\bv(P_{N_x})=\sigma^*\bar{\bw}(x)-\ov{\mathcal C}_q \bar{\bv}(x)\geq0.
\end{align*}
Here $P_{N_x}$ is the maximal direct summand of $N_x$ belonging to $\mod(kQ)$.

Therefore $(\bv,\bw)$ is an $l$-dominant pair for $\mcr^\imath$.
\end{proof}

Recall $\cv^*: \mod(\mcr^\imath)\rightarrow \mod(\mcr)$ and $\cv_*:\mod(\mcr)\rightarrow \mod(\mcr^\imath)$ in \S\ref{subsec: covering}. For any $M\in \mod(\mcr^\imath)$ with dimension vector $(\bv,\bw)\in (V^{+},W^{+})$, we have an $\mcr$-module $M_\mcr$ of the same dimension vector such that $\cv_*(M_\mcr)=M$. Then $\cv^*(M)=M_\mcr\oplus \Sigma\widehat{\btau}(M_\mcr)$.

Similarly, any $X \in \mod(kQ)$ can be viewed as a $\cs$-module (respectively,  $\cs^\imath$-module), which will be denoted by $X_{\cs}$ (respectively,   $X_{\cs^\imath}$).

Recall the intermediate extension defined in \S\ref{subsec: strat}. To distinguish the notations, we shall denote by $K_{LR}:\mod(\cs)\rightarrow \mod(\mcr)$, and by $K_{LR}^\imath:\mod(\cs^\imath) \rightarrow \mod(\mcr^\imath)$; similarly $K_L,\res,K_R$ are defined for $\cs,\mcr$, and $K_L^\imath,\res^\imath,K_R^\imath$ are defined for $\cs^\imath,\mcr^\imath$.

\begin{lemma}
\label{lem: coincide of KLR}
We have $K_{LR}\circ \cv^*= \cv^*\circ K_{LR}^\imath$.
Moreover, for any $X\in\mod(kQ)$, we have $\dimv K_{LR}(X_\cs)=\dimv K_{LR}^\imath(X_{\cs^\imath})$.
\end{lemma}

\begin{proof}
Clearly, we have $\cv_*\circ \res=\res^\imath\circ \cv_*$. By comparing their left and right adjoint pairs, it follows that
$K_L\circ \cv^*=\cv^*\circ K_L^\imath$ and $K_R\circ \cv^*=\cv^*\circ K_R^\imath$. By the definition of the intermediate extension, we have
 $K_{LR}\circ \cv^*= \cv^*\circ K_{LR}^\imath$.

For any $X\in\mod(kQ)$, we have
$$K_{LR}\circ \cv^*(X_{\cs^\imath})\cong K_{LR} (X_{\cs}\oplus \Sigma \widehat{\btau} (X_{\cs}))=K_{LR} (X_\cs) \oplus  \Sigma \widehat{\btau} K_{LR} (X_\cs).$$
Similarly, $\cv^*\circ K_{LR}^\imath(X_{\cs^\imath})= K_{LR}^\imath (X_{\cs^\imath}) \oplus \Sigma \widehat{\btau} K_{LR}^\imath (X_{\cs^\imath})$.
Note that
\begin{align*}
&\dimv K_{LR} (X_\cs) \in ({V}^+,{W}^+),\quad  \dimv \Sigma \widehat{\btau} K_{LR} (X_\cs) \in ({V}^-,{W}^-); \\
&\dimv K_{LR}^\imath (X_{\cs^\imath}) \in ({V}^+,{W}^+), \quad \dimv \Sigma \widehat{\btau} K_{LR}^\imath (X_{\cs^\imath})\in ({V}^-,{W}^-).
\end{align*}
It follows from
$$ K_{LR} (X_\cs) \oplus  \Sigma \widehat{\btau} K_{LR} (X_\cs)=   K_{LR}^\imath (X_{\cs^\imath}) \oplus \Sigma \widehat{\btau} K_{LR}^\imath (X_{\cs^\imath}) $$
that $K_{LR}(X_\cs)= K_{LR}^\imath (X_{\cs^\imath})$, and then $\dimv K_{LR}^\imath(X_{\cs^\imath})=\dimv K_{LR}(X_\cs)$.
\end{proof}

\begin{proposition}
\label{prop:dominant++}
A dimension vector $(\bv,\bw)\in(V^{+},W^{+})$ is a strongly $l$-dominant pair for $\mcr^\imath$ if and only if it is a strongly $l$-dominant pair for $\mcr$. Moreover, any $l$-dominant pair for $\mcr^\imath$ in $(V^+,W^+)$ is strongly $l$-dominant.
\end{proposition}
\begin{proof}
The first assertion follows from Lemma \ref{lem:bistable}(ii), Lemma \ref{lem:sameDP} and Lemma \ref{lem: coincide of KLR}.

Assume now $(\bv,\bw)\in(V^{+},W^{+})$ is an $l$-dominant pair for $\mcr^\imath$. Then it is an $l$-dominant pair for $\mcr$. Similarly, we can prove that it is an $l$-dominant pair for $\mcr^{\gr}$. It follows by \cite[Theorem 3.14]{LeP13} that it is a strongly $l$-dominant pair for $\mcr^{\gr}$, and the same claim holds for $\mcr$ and $\mcr^\imath$.
\end{proof}

%Denote the extension $()\otimes_\N\Z$ by $()_\Z$.
Denote $\N^{\cp^\imath} =\bigoplus_{x\in\mcr^\imath_0-\cs^\imath_0} \N \e_x$. Recalling $V^0, V^+, W^+$ from \eqref{eq:VW+}, we denote
\begin{align*}
 %\label{eq:VW2}
W^{+}_\Z&=\bigoplus_{x\in\{\ts_i,i\in Q_0\}} \Z \e_{\sigma x},&&
V^{+}_\Z=\bigoplus_{
x\in\Ind \mod(kQ),\, x\text{ is not injective}} \Z \e_x.
\end{align*}
Similar to \cite[Lemma 4.3.1]{Qin}, we observe that $\N^{\cp^\imath}$ is a subset of $V^{+}_\Z\oplus V^0$. Indeed, for any $i,j\in Q_0$, we have $\bv^{i}(I_j)=\delta_{ij}$; see \eqref{eq:vfi} for $\bv^i$.  For any $\bv\in \N^{\cp^\imath}$, define
\begin{align}
\label{eqn: decomposition}
\bv^0 &=\sum_{i\in \I} \bv(I_i)\bv^{i},
\qquad
\bv^+ =\bv-\bv^0.
\end{align}
Then $\bv=\bv^++\bv^0 \in V^{+}_\Z\oplus V^0$ is the desired decomposition.

\begin{lemma}
\label{lem: dominant pair 1}
Let $\bw\in W_\Z^{+}$, $\bv\in V_\Z^+$. If $\sigma^*\bw-{\mathcal C}_q \bv\geq0$, then $\bv\in V^{+}$, $\bw\in W^{+}$.
\end{lemma}

\begin{proof}
As in the proof of Lemma \ref{lem:sameDP}, the map defined in \eqref{def:Cq} for $\mcr$ is denoted by $\ov{\mathcal C}_q$. From the proof of the ``only if'' part of Lemma \ref{lem:sameDP}, we have $\sigma^*{\bw}-\ov{\mathcal C}_q {\bv}\geq0$. It follows by \cite[Proposition~ 4.3.2]{Qin} that $\bv\in V^{+}$, $\bw\in W^{+}$.
\end{proof}

\begin{lemma} [$l$-dominant pair decomposition]
   \label{lem:DP3}
If $(\bv,\bw)$ is an $l$-dominant pair for $\mcr^\imath$, then we have a unique decomposition of $(\bv,\bw)$ into $l$-dominant pairs $(\bv^+,\bw^+) \in (V^+, W^+)$, $(\bv^0,\bw^0) \in (V^0, W^0)$ such that
\[
(\bv^+, \bw^+) + (\bv^0, \bw^0)= (\bv, \bw),
\qquad
%
%\bv^++\bv^0=\bv,\quad \bw^++\bw^0= \bw, \quad
\sigma^*\bw^0-{\mathcal C}_q \bv^0=0.
\]
\end{lemma}

\begin{proof}
The proof is the same as for \cite[Proposition 4.3.2]{Qin}, now with help of Lemma \ref{lem: dominant pair 1}. We omit the detail.
\end{proof}

Introduce the following $\Z[\tt^{1/2}, \tt^{-1/2}]$-submodules of $\tRZ$:
\begin{align}
\tR^{\imath+}_\Z &= \bigoplus_{\stackrel{(\bv^+,\bw^+)\in (V^{+},W^{+})}{\sigma^*\bw^+ -{\mathcal C}_q \bv^+\geq0}} \Z[\tt^{1/2}, \tt^{-1/2}] L(\bv^+,\bw^+),\\
\tR^{\imath0}_\Z &=\bigoplus_{\stackrel{(\bv^0,\bw^0)\in (V^0,W^0)}{\sigma^*\bw^0-{\mathcal C}_q \bv^0=0}}  \Z[\tt^{1/2}, \tt^{-1/2}] L(\bv^0,\bw^0).
\label{eq:Ri0Z}
\end{align}
We further set $\tR^{\imath+} =\Q(\tt^{1/2}) \otimes \tR^{\imath+}_\Z$ and $\tR^{\imath0} =\Q(\tt^{1/2}) \otimes \tR^{\imath0}_\Z$.

\begin{theorem}
\label{thm:RRR}
The quantum Grothendieck ring $(\tRiZ,\cdot)$ has a basis given by
\begin{align*}
&\Big \{L(\bv^+,\bw^+)  L(\bv^0,\bw^0)
\mid  (\bv^+,\bw^+)\in(V^{+},W^{+}) \text{ are $l$-dominant},\\
&\qquad\qquad\qquad\qquad\qquad (\bv^0,\bw^0)\in(V^0,W^0) \text{ such that }\sigma^*\bw^0-{\mathcal C}_q \bv^0=0 \Big\}.
\end{align*}
\end{theorem}

\begin{proof}
Since the multiplication $\cdot$ is a twisted version of $\ast $ (up to some powers of $\tt^{1/2}$), the statement is equivalent to its counterpart in $(\tRiZ, \ast)$, and we shall prove this version. Let $\mu: \tR^{\imath+}_\Z \otimes \tR^{\imath0}_\Z \rightarrow \tRiZ$ be given by the multiplication $\ast $. It suffices to prove that $\mu$ is an isomorphism.

Recall from \eqref{eq:tRi} that $\tRiZ$ is a $\N \I$-graded (i.e., $W^+$-graded) algebra, whose graded subspaces $\tR^{\imath}_{\Z,\bw}$ are of finite rank.
For a given $\bw \in W^+$, consider the following two finite subsets in $\tR^{\imath}_{\Z ,\bw}$, of the same cardinality thanks to Lemma~\ref{lem:DP3}:
\begin{align}
& \{ L(\bv^+,\bw^+)\ast  L(\bv^0,\bw^0)
\mid  \bw^+ +\bw^0 =\bw,
(\bv^+,\bw^+)\in(V^{+},W^{+}) \text{ are $l$-dominant},
\label{set1} \\
&\qquad\qquad\qquad\qquad\qquad (\bv^0,\bw^0)\in(V^0,W^0) \text{ such that }\sigma^*\bw^0-{\mathcal C}_q \bv^0=0 \Big\};
\notag
\\
& \{ L(\bv,\bw)  \mid (\bv,\bw)\ \text{ are strongly $l$-dominant} \}.
\label{set2}
\end{align}
Note \eqref{set2} is a basis for $\tR^{\imath}_{\Z,\bw}$, and we have used ``$l$-dominant" instead of ``strongly $l$-diminant" in \eqref{set1} thanks to Proposition~\ref{prop:dominant++}.

For any pairs $(\bv^+,\bw^+), (\bv^0,\bw^0)$ which satisfy the conditions in \eqref{set1}, we have
$L(\bv^+,\bw^+)\in \tR^{\imath+}_\Z$, and $L(\bv^0,\bw^0)\in \tR^{\imath0}_\Z$, and by
\eqref{equation multiplication},
\begin{align}
  \label{equation multiplication2}
L(\bv^+,\bw^+)\ast  L(\bv^0,\bw^0)
\in v^\Z L(\bv^++\bv^0,\bw)
 +\sum_{{}'\bv > \bv^++\bv^0}  \Z[\tt, \tt^{-1}] L({}'\bv,\bw).
\end{align}
We observe from \eqref{equation multiplication2} that the transition matrix $T$ from \eqref{set2} to \eqref{set1} is a triangular matrix with entries in $\N[\tt, \tt^{-1}]$ and diagonals in $v^\Z$, and hence $T$ is invertible with $T^{-1}$ having entries in $\Z[\tt, \tt^{-1}]$. This implies $\mu$ is an isomorphism.
\end{proof}

\begin{remark}
Similarly, $(\tRiZ,\cdot)$ has a basis given by
\begin{align*}
&\Big \{ L(\bv^0,\bw^0) L(\bv^+,\bw^+)
\mid  (\bv^+,\bw^+)\in(V^{+},W^{+}) \text{ are $l$-dominant},\\
&\qquad\qquad\qquad\qquad\qquad (\bv^0,\bw^0)\in(V^0,W^0) \text{ such that }\sigma^*\bw^0-{\mathcal C}_q \bv^0=0 \Big\}.
\end{align*}
\end{remark}

\subsection{Filtered algebra $\tRi$}

\begin{lemma} \label{lem:R0 subalgebra}
Let $\bv\in V^0$, $\bw\in W^0$ be such that $\sigma^*\bw-{\mathcal C}_q \bv=0$. Then we have $(\bv,\bw)\in\bigoplus_{i\in \I}\N(\bv^{i},\bw^{i})$.
\end{lemma}

\begin{proof}
Recall $\I_\btau$ from \eqref{eqn:representative}. By definition, assume that $\bv=\sum_{i\in \I} a_i \bv^{i}$, $\bw=\sum_{i\in \I_\btau} b_i\bw^{i}$ with $a_i,b_i\in\N$.
As $\sigma^*\bw-{\mathcal C}_q \bv=0$, using Lemma \ref{lem:dimCartan}(iv) we have
\begin{align*}
 \begin{cases}
\,\; b_i =\sigma^*\bw(\ts_i)={\mathcal C}_q\bv(\ts_i)=\sum_{j\in \I} a_j {\mathcal C}_q\bv^{j}(\ts_i)=\sum_{j\in \I} a_j \sigma^*\bw^{j}(\ts_i)=a_i+a_{\btau i}, & \;\text{ if }\btau i\neq i;
\\
2b_i = \sigma^*\bw(\ts_i)={\mathcal C}_q\bv(\ts_i)=\sum_{j\in \I} a_j {\mathcal C}_q\bv^{j}(\ts_i)=\sum_{j\in \I} a_j \sigma^*\bw^{j}(\ts_i)=2a_i, & \; \text{ if } \btau i=i.
\end{cases}
\end{align*}
It follows from $\bw^{i}=\bw^{{\btau i}}$
that $(a_i\bv^{i}+a_{\btau i}\bv^{{\btau i}}, b_i \bw^{i})=a_i(\bv^{i},\bw^{i})+a_{\btau i}(\bv^{\btau i},\bw^{i})$ for any $i\in \I_\btau$ with $\btau i\neq i$, and
$(a_i\bv^{i}, b_i \bw^{i})=a_i(\bv^{i},\bw^{i})$ for any $i\in \I_\btau$ with $\btau i=i$. Therefore,
$(\bv,\bw)=\sum_{i\in \I} a_i (\bv^{i},\bw^{i}).$
\end{proof}

This following is a somewhat enhanced reformulation of Theorem \ref{thm:RRR}.

\begin{proposition}  \label{prop:poly}
$\tR^{\imath0}_\Z$ is a polynomial subring of $\tRiZ$ in generators $\{L(\bv^{i},\bw^{i})\mid i\in \I \}$.
As a left/right $\tR^{\imath0}_\Z$-module, $\tRiZ$ is free with basis $\{L(\bv^+,\bw^+)\mid (\bv^+,\bw^+)\in (V^{+},W^{+})\text{ $l$-dominant} \}.$
\end{proposition}

\begin{proof}
The first assertion follows from \eqref{eqn: cartan mult 2} and Lemma~\ref{lem:R0 subalgebra}.
The second assertion follows from Theorem~ \ref{thm:RRR}.
\end{proof}

For any $\gamma \in \N^{\cs^\imath_0} =\N \I$, we denote
\vspace{2mm}

$\triangleright$ $\tRZ^{\imath,\leq \gamma}$ $=$ $\Z[\tt^{1/2}, \tt^{-1/2}]$-submodule of $\tRiZ$ spanned by the basis elements $L(\bv^+,\bw^+)  L(\bv^0,\bw^0)$ in Theorem~\ref{thm:RRR} with $\bw^+\leq \gamma$.

$\triangleright$ $\tRZ^{\imath,<\gamma}$ $=$ $\Z[\tt^{1/2}, \tt^{-1/2}]$-submodule of $\tRiZ$ spanned by the basis elements $L(\bv^+,\bw^+)  L(\bv^0,\bw^0)$ in Theorem~\ref{thm:RRR} with $\bw^+ <\gamma$.

\begin{lemma}
\label{lem: filtration strict}
(1) Retain the notation for the $l$-dominant pair decomposition of $(\bv, \bw)$ as in Lemma~\ref{lem:DP3}.
%Let $(\bv^+,\bw^+)\in(V^+,W^+)$, $(\bv^0,\bw^0)\in (V^0,W^0)$ be $l$-dominant pairs with $(\bv^+, \bw^+) + (\bv^0, \bw^0)= (\bv, \bw)$ and $\sigma^*\bw^0-\cc_q \bv^0=0$.
Then
\[
L(\bv, \bw) \in v^\Z L(\bv^+,\bw^+)\ast  L(\bv^0,\bw^0) + \sum_{\stackrel{{}'\bw^+\leq \bw^+,}{ {}'\bv^+ +{}'\bv^0 > \bv}} \Z[\tt, \tt^{-1}]  L('\bv^+,{}'\bw^+)\ast  L('\bv^0,{}'\bw^0).
\]
In particular, we have $L(\bv,\bw) \in \tR^{\imath, \leq \bw^+}$.

(2) Let $(\bv,\bw), ({}'\bv,\bw)$ be strongly $l$-dominant pairs such that ${}'\bv >\bv$. If $L(\bv,\bw) \in \tR^{\imath, \leq \gamma}$ for some $\gamma$, then $L('\bv,\bw) \in \tR^{\imath, \leq \gamma}$.
%
%Then for any strongly $l$-dominant pair $(\bv, \bw^++\bw^0)$ with $\bv \geq \bv^++\bv^0$, we have $L(\bv,\bw^++\bw^0)\in \tR^{\imath,\leq\bw{^+}}$.
%
%For any $l$-dominant pair $(\bv^+,\bw^+)\in(V^+,W^+)$, $(\bv^0,\bw^0)\in (V^0,W^0)$ such that $\sigma^*\bw^0-\cc_q \bv^0=0$, we have $L(\bv^++\bv^0,\bw^++\bw^0)\in \tR^{\imath,\leq\bw{^+}}$.
\end{lemma}

\begin{proof}
(1)
For any $l$-dominant pair $(\bv',\bw)$, recall its unique $l$-dominant pair decomposition
of $(\bv',\bw)$ as in Lemma~\ref{lem:DP3}: $ (\bv', \bw) = ('\bv^+, {}'\bw^+) + ('\bv^0, {}'\bw^0)$.
%into $('\bv^+,'\bw^+) \in (V^+, W^+)$, $('\bv^0,'\bw^0) \in (V^0, W^0)$.
%such that
%\[%
%('\bv^+, {}'\bw^+) + ('\bv^0, {}'\bw^0)= (\bv', \bw),
%\qquad
%\sigma^*{}'\bw^0-{\mathcal C}_q {}'\bv^0=0.
%\]

{\bf Claim}. If $\bv'>\bv$, then $'\bw^+\leq \bw^+$.

Indeed, assume  $\bv'>\bv$. It follows by \eqref{eqn: decomposition} that $'\bv^0\geq\bv^0$. Then $'\bw^0\geq \bw^0$ by Lemma~ \ref{lem:R0 subalgebra}, which implies that $'\bw^+\leq \bw^+$. The Claim is proved.

Hence we can refine the partial ordering in \eqref{equation multiplication2} as
\begin{align*}
 % \label{equation multiplication3}
L(\bv^+,\bw^+)\ast  L(\bv^0,\bw^0)
\in v^\Z L(\bv,\bw) +\sum_{\stackrel{ {}'\bw^+ \leq \bw^+,}{ {}'\bv > \bv}}  \N[\tt, \tt^{-1}] L({}'\bv,\bw).
\end{align*}
The invertible transition matrix $T$ from a basis \eqref{set2} to another basis \eqref{set1} of $\tRZ^{\imath,\bw}$ (see the proof of Theorem~\ref{thm:RRR}) is a matrix with diagonals in $\tt^\Z$ which is trianglular with respect to this refined partial ordering.
% ${}'\bw^+ \leq \bw^+,  {}'\bv > \bv$.
Inverting the matrix $T$ proves (1).

Retaining the notation above, we only need to prove (2) by taking $\gamma =\bw^+$.
Thanks to ${}'\bw^+ \leq \bw^+$ above, we obtain by (1) that $L('\bv,\bw) \in \tR^{\imath, \leq \, {}'\bw^+} \subseteq \tR^{\imath, \leq\bw^+}$. This proves (2).
\end{proof}

\begin{proposition}
  \label{prop:filtered}
We have $\tRZ^{\imath,\leq\alpha} \ast \tRZ^{\imath,\leq\beta}\subseteq \tRZ^{\imath, \leq \alpha+\beta}$, and
$\tRZ^{\imath,\leq\alpha} \cdot \tRZ^{\imath,\leq\beta}\subseteq \tRZ^{\imath, \leq \alpha+\beta}$, for any $\alpha,\beta \in \N^{\cs^\imath_0}$. Hence $(\tRiZ, \ast)$ and $(\tRiZ,\cdot)$ are filtered rings.
\end{proposition}

\begin{proof}
As the two versions are equivalent, it suffices for us to prove for $(\tRiZ, \ast)$.

For two basis elements $L(\bv^+_1,\bw^+_1) \ast L(\bv^0_1,\bw^0_1) \in \tR^{\imath,\leq \bw^+_1}$, $L(\bv^+_2,\bw^+_2) \ast L(\bv^0_2,\bw^0_2)\in \tR^{\imath,\leq \bw^+_2}$ from Theorem~\ref{thm:RRR} (with corresponding conditions attached), it follows by \eqref{equation multiplication} that
\begin{align}
L(\bv^+_1,\bw^+_1) & \ast L(\bv^0_1,\bw^0_1)\ast L(\bv^+_2,\bw^+_2) \ast L(\bv^0_2,\bw^0_2)
 \notag \\
&\in \sum_{\bv\geq \bv^+_1+\bv^+_2 + \bv^0_1+\bv^0_2}\Z[\tt, \tt^{-1}] L\big(\bv,(\bw^+_1+\bw^+_2)+(\bw^0_1+\bw^0_2) \big)
  \label{eq:belong}
\end{align}
We have $L\big((\bv^+_1+\bv^+_2)+(\bv^0_1+\bv^0_2),(\bw^+_1+\bw^+_2)+(\bw^0_1+\bw^0_2) \big) \in \tR^{\imath,\leq \bw^+_1+\bw^+_2}$, by Lemma \ref{lem: filtration strict}(1) (whose assumption is readily verified).
Then by Lemma \ref{lem: filtration strict}(2), each term on the RHS of \eqref{eq:belong} lies in $\tR^{\imath,\leq \bw^+_1+\bw^+_2}$, and so does the LHS of \eqref{eq:belong}. The proposition is proved.
\end{proof}

Denote
\[
\tR^{\imath,gr}=\bigoplus_{\gamma \in \N \I} \tR^{\imath,gr}_{\gamma},
\qquad \text{ where }\tR^{\imath,gr}_{\gamma}=\tR^{\imath,\leq \gamma}/\tR^{\imath,<\gamma}.
\]
Then $(\tRZ^{\imath,gr},\ast _{gr})$ (respectively, $(\tRZ^{\imath,gr},\cdot _{gr})$) is the graded ring associated to the filtered ring $(\tRiZ, \ast)$ (respectively, $(\tRiZ,\cdot)$) by Proposition~\ref{prop:filtered}. By Proposition~ \ref{prop:poly}, it is natural to view $\tR^{\imath0}$ as a subalgebra of $\tR^{\imath,gr}$.

For any $(\bv^+,\bw^+)\in (V^+,W^+)$,  we denote $\bar{L}(\bv^+,\bw^+) =L(\bv^+,\bw^+)+ \tR^{\imath,<\bw^+}\in \tR^{\imath,gr}$. The following statement is similar to \cite[Lemma 5.4]{LW19} and can be derived from Theorem~\ref{thm:RRR}.

\begin{corollary}
  \label{cor:basisGr}
The associated graded ring $(\tR^{\imath,gr},\ast _{gr})$ has a basis given by
\begin{align*}
& \{\bar L(\bv^+,\bw^+)\ast _{gr} L(\bv^0,\bw^0)
\mid  (\bv^+,\bw^+)\in(V^{+},W^{+}) \text{ are $l$-dominant},\\
&\qquad\qquad\qquad\qquad\qquad
(\bv^0,\bw^0)\in(V^0,W^0) \text{ such that }\sigma^*\bw^0-{\mathcal C}_q \bv^0=0 \Big\}.
\notag
\end{align*}
\end{corollary}

There is a similar basis as in Corollary~\ref{cor:basisGr} for the associated graded ring $(\tR^{\imath,gr},\cdot_{gr})$.

\subsection{Generators for $\tRi$}

%In general, we do not have $\cm_0(\bw,\mcr^\imath)\cong \cm_0(\bw,\mcr)$ for $\bw\in W^+$, cf. \cite[Proposition~ 4.2.1]{Qin}. However, we have the following lemma. \red{remove}

Any $(\bv,\bw)\in(V^+,W^+)$ can be viewed naturally as a dimension vector for $\mcr^{\gr}$ by the canonical embedding $\mod(kQ)\subseteq \mcr^{\gr}$. Recall the notation \eqref{eq:vwvw} on $(\bar{\bv},\bar{\bw})$.

\begin{lemma}
\label{lem: iso of varieties}
For any $(\bv,\bw)\in(V^+,W^+)$, the following isomorphisms of varieties hold:
\begin{align}
\label{eqn: iso of var 1}
\cm(\bv,\bw,\mcr^\imath) & \cong \cm(\bv,\bw,\mcr^{\gr})\cong \cm(\bar{\bv},\bar{\bw},\mcr);\\
\label{eqn: iso of var 2}
\cm_0(\bv,\bw,\mcr^\imath) & \cong \cm_0(\bv,\bw,\mcr^{\gr})\cong \cm_0(\bar{\bv},\bar{\bw},\mcr).
\end{align}
\end{lemma}

\begin{proof}
For \eqref{eqn: iso of var 1} and \eqref{eqn: iso of var 2}, we shall only prove the first isomorphisms; the second isomorphism in \eqref{eqn: iso of var 1} or \eqref{eqn: iso of var 2} can be proved similarly or follows from \cite[Proposition~ 4.2.1]{Qin}.

Let $\ov{\mcr}^\imath:=\mcr^\imath/( x | x\in\Ind\mod(kQ)\text{ is injective}  )$. Let $\ov{\mcr}^{\gr,+}$ be the subcategory of $\mcr^{\gr}$ generated by $x\in\Ind\mod(kQ)$ with $x$ not injective. Then $\ov{\mcr}^{\gr,+}$ is also a quotient category of $\mcr^{\gr}$, i.e., $\ov{\mcr}^{\gr,+}=\mcr^{\gr}/( x\mid x\notin \big(\Ind\mod(kQ) -\inj(kQ)\big))$.
The definition of $(V^+,W^+)$ implies that $(\bv,\bw)$ can be viewed as dimension vectors for $\ov{\mcr}^\imath$ and $\ov{\mcr}^{\gr,+}$. In this way, we have $\rep(\bv,\bw,\mcr^\imath)\cong \rep(\bv,\bw,\ov{\mcr}^\imath)$ and $\rep(\bv,\bw,\mcr^{\gr})\cong\rep(\bv,\bw,\ov{\mcr}^{\gr,+})$ as varieties.

{\bf Claim.} $\ov{\mcr}^\imath\cong \ov{\mcr}^{\gr,+}$.

Assume for now the Claim holds. Then we have $\rep(\bv,\bw,\ov{\mcr}^\imath)\cong \rep(\bv,\bw,\ov{\mcr}^{\gr,+})$, and thus $\rep(\bv,\bw,\mcr^\imath)\cong \rep(\bv,\bw,\mcr^{\gr})$. The first isomorphisms in \eqref{eqn: iso of var 1} and \eqref{eqn: iso of var 2} follow by definition.

It remains to prove the Claim. There exists a natural functor $\Omega:\ov{\mcr}^{\gr,+}\rightarrow \mcr^\imath$, which is faithful. Combining with the natural projection $\mcr^\imath\rightarrow \ov{\mcr}^\imath$, we obtain a functor
\[
\ov{\Omega}: \ov{\mcr}^{\gr,+} \longrightarrow \ov{\mcr}^\imath,
\]
which is dense.
The Claim follows by showing that $\ov{\Omega}$ is an isomorphism. To that end, it remains to prove that $\ov{\Omega}$ is fully faithful.

\underline{Faithfulness of $\ov{\Omega}$.} For any morphism $f$ in $\ov{\mcr}^{\gr,+}$, if $\ov{\Omega}(f)=0$, then $\Omega(f)$ factors through some injective module $I$. Therefore $f$ factors through some object in $\Omega^{-1}(I)$. Note that $\Omega^{-1}(I)\subseteq \add\{ (\Sigma\widehat{\varrho})^i(I_j)=\Sigma^i I_{\varrho^i(j)} \mid i\in\Z ,j\in\I \}$.
%Since $f$ is a morphism in $\ov{\mcr}^{\gr,+}$, if $i\neq0$, then $f=0$ by definition. Note that $\widehat{\varrho}$ preserves injective modules.
By definition of $\ov{\mcr}^{\gr,+}$, we have $f=0$. So $\ov{\Omega}$ is faithful.

\underline{Fullness of $\ov{\Omega}$.} Denote by $F$ the autoequivalence of $\mcr^\gr$ induced by $\Sigma\widehat{\varrho}$. Then $\mcr^\imath=\mcr^{\gr}/F$. Note that $\Hom_{\mcr^\imath}(x,y)=\bigoplus_{i\in\Z}\Hom_{\mcr^{\gr}}(x,F^iy)$ for any $x,y\in\Ind\mod(kQ)\sqcup\{\sigma(\ts_i)\mid i\in\I\}$. For any nonzero morphism $g:x\rightarrow y$ in $\ov{\mcr}^{\imath}$, without loss of generality, we assume that $g\in\Hom_{\mcr^{\gr}}(x,F^iy)$ for some $i\neq0$, and furthermore, we may assume $g$ is a composition of irreducible morphisms $x=x_0\xrightarrow{g_0}x_1\xrightarrow{g_1}\cdots \xrightarrow{g_n}x_n=F^iy$ in $\mcr^{\gr}$. If $i=0$, then $g\in\Hom_{\mcr^{\gr}}(x,y)$, and we see that $g$ is the image of $\ov{\Omega}$. Otherwise,
since $\mcr^{\gr}$ is a directed category (see, e.g., \cite{Ha2} or \cite[\S9.1]{LW19} for definition), we have $i>0$. We shall complete the proof by showing that it is not allowed to have $i>0$.

Note that $F^iy\in\Sigma^i\mod(kQ)\sqcup\{\sigma(\Sigma^i\ts_j)\mid j\in\I\}$.
So there exists the minimal integer $j$ such that $x_j\in\mod(kQ)$ or $x_j=\sigma(\ts_k)$ for some $k\in\I$, but $x_{j+1}\in\Sigma^t\mod(kQ)$ or $x_{j+1}=\sigma(\Sigma^t\ts_l)$ for some $t>0$, $l\in\I$. We proceed by separating in the following three cases (i)--(iii).

(i) \underline{$x_j=\sigma(\ts_k)$}. Since there is only one irreducible morphism starting from $\sigma(\ts_k)$, i.e., $\sigma(\ts_k)\rightarrow \ts_k$, this contradicts with the assumption on $x_{j+1}$.

(ii) \underline{$x_{j+1}=\sigma(\Sigma^t\ts_l)$}. Since there is only one irreducible morphism ending at $\sigma(\Sigma^t\ts_l)$: $\Sigma^t\tau \ts_l\rightarrow \sigma(\Sigma^t\ts_l)$,  we have $x_j=\Sigma^t\tau \ts_l\in\mod(kQ)$. So $t=1$, and $\ts_l$ is projective. Then $x_j=\Sigma \tau \ts_l$ is injective. It follows that $g:x\rightarrow y$ factors through some injective module, which contradicts with $0\neq g\in\ov{\mcr}^{\imath}$.

(iii) \underline{$x_j\in\mod(kQ)$ and $x_{j+1}\in \Sigma^t\mod(kQ)$ for some $t>0$}. Since $g_j:x_j\rightarrow x_{j+1}$ is irreducible, we see $x_j$ is injective. Then $g:x\rightarrow y$ factors through some injective module, which is again a contradiction.

This proves that $\ov{\Omega}:\ov{\mcr}^{\gr,+}\rightarrow \ov{\mcr}^\imath$ is an isomorphism. The lemma is proved.
\end{proof}

Recalling the basis of $\tRZ^{\imath,gr}$ in Corollary~\ref{cor:basisGr}, we introduce a $\Z[\tt^{1/2}, \tt^{-1/2}]$-linear map
\begin{align}
\label{def: varphi}
\varphi: & \tRZ^+  \longrightarrow \tRZ^{\imath,gr},
\\
& L_\cs (\bar{\bv},\bar{\bw})  \mapsto \bar L(\bv,\bw),
\quad
 \text{ for all $l$-dominant pairs } (\bar{\bv},\bar{\bw}) \in (V^+, W^+).
 \notag
\end{align}
%Here we recall the notation \eqref{eq:vwvw} on $(\bar{\bv},\bar{\bw})$.

\begin{proposition}
\label{prop:R+Rigr}
The linear map $\varphi: \tRZ^+\rightarrow \tRZ^{\imath,gr}$ is an injective algebra homomorphism. (Here $\tRZ^+$ and $\tRZ^{\imath,gr}$ are either both endowed with the standard multiplications or both with the twisted multiplications.)
\end{proposition}

\begin{proof}
We will consider the (untwisted) rings $(\tRZ^+, \ast)$ and $(\tRZ^{\imath,gr}, \ast_{gr})$ until near the end. The injectivity of $\varphi$ follows by definition and Proposition~\ref{prop:dominant++}.

For any dimension vector $(\bv,\bw)\in(V^+,W^+)$, by Lemma \ref{lem:sameDP}, we have $(\bv,\bw)$ is (strongly) $l$-dominant if and only if $(\bar{\bv},\bar{\bw})$ is. Moreover, note that $\tau^*(\e_{x})=\e_{\tau^{-1}x}$. Then we read off from the proof of Lemma \ref{lem:sameDP} that
\begin{align}
\label{eqn: lem embedding 1}
\big(\sigma^*\bar{\bw}-\ov{\mathcal C}_q \bar{\bv}\big)(\tau^{-1}x)=\big(\sigma^*\bw-{\mathcal C}_q \bv\big)(\tau^{-1}x)
\end{align}
for any $x\in\mod(kQ)$, $x\notin\inj(kQ)$. Hence, for any  $(\bv_1,\bw_1)$, $(\bv_2,\bw_2)\in (V^+,W^+)$, we have by \eqref{definition:d} and \eqref{eqn: lem embedding 1}
\begin{align}
  \label{eq:dd}
d\big((\bv_1,\bw_1),(\bv_2,\bw_2) \big)=d\big((\bar{\bv}_1,\bar{\bw}_1),(\bar{\bv}_2,\bar{\bw}_2) \big).
\end{align}

%Then \eqref{eqn:comultiplication} implies that $\widetilde{\Delta}^{\bw}_{\bw_1,\bw_2} \big(\pi(\bv,\bw,\mcr^\imath) \big)$ coincides with $\widetilde{\Delta}^{\bar{\bw}}_{\bar{\bw}_1,\bar{\bw}_2} \big(\pi(\bar{\bv},\bar{\bw},\mcr) \big)$ when we forget the framework $\mcr^\imath$ and $\mcr$.

We shall denote \eqref{eqn:decomposition theorem} in our setting as
\begin{align*}
\pi(\bv,\bw,\mcr^\imath)
&= \sum_{\bv':\sigma^*\bw-{\mathcal C}_q\bv'\geq0,\bv'\leq \bv} a^\imath_{\bv,\bv';\bw}(\tt)\cl(\bv',\bw,\mcr^\imath),\\
\pi(\bar{\bv},\bar{\bw},\mcr)
&= \sum_{\bv':\sigma^*\bar{\bw}-\bar{{\mathcal C}}_q\bv'\geq0,\bv'\leq \bar{\bv}} a_{\bar{\bv},\bv';\bar{\bw}}(\tt)\cl(\bv',\bar{\bw},\mcr).
\end{align*}
Since $\bv'\leq \bv$ (or equivalently, $\bv'\leq \bar{\bv}$), we have $\bv'\in V^+$, and then by Lemma \ref{lem: iso of varieties} we obtain
\begin{align}
  \label{eq:aa}
a^\imath_{\bv,\bv';\bw}(\tt)=a_{\bar{\bv},\bv';\bar{\bw}}(\tt).
\end{align}
It follows by \eqref{eqn:comultiplication} and \eqref{eq:dd}--\eqref{eq:aa} that the structure constants for
$\widetilde{\Delta}^{\bw}_{\bw_1,\bw_2} \big(\cl(\bv,\bw,\mcr^\imath) \big)$ and $\widetilde{\Delta}^{\bar{\bw}}_{\bar{\bw}_1,\bar{\bw}_2} \big(\cl(\bar{\bv},\bar{\bw},\mcr) \big)$ in the corresponding $\cl$-bases coincide, where $\bw_1+\bw_2 =\bw$.

Fix $\bw_1, \bw_2$ with $\bw_1+\bw_2 =\bw$. Dualizing $\widetilde{\Delta}^{\bw}_{\bw_1,\bw_2} \big(\cl(\bv,\bw,\mcr^\imath) \big)$ and $\widetilde{\Delta}^{\bar{\bw}}_{\bar{\bw}_1,\bar{\bw}_2} \big(\cl(\bar{\bv},\bar{\bw},\mcr) \big)$, we conclude that $c_{\bv_1,\bv_2}^\bv(\tt)= \bar{c}_{\bar{\bv}_1,\bar{\bv}_2}^\bv(\tt)$ for any $\bv\in V^+$, where we have denoted \eqref{equation multiplication} in our setting as
\begin{align*}
L(\bv_1,\bw_1,\mcr^\imath)\ast  L(\bv_2,\bw_2,\mcr^\imath)
&= \sum_{\bv \geq \bv_1+\bv_2} c_{\bv_1,\bv_2}^\bv(\tt) L(\bv,\bw,\mcr^\imath),\\
L(\bar{\bv}_1,\bar{\bw}_1,\mcr)\ast  L(\bar{\bv}_2,\bar{\bw}_2,\mcr)
&= \sum_{\bv \geq \bar{\bv}_1+\bar{\bv}_2} \bar{c}_{\bar{\bv}_1,\bar{\bv}_2}^\bv(\tt) L(\bv,\bar{\bw},\mcr).
\end{align*}
In addition, we note that $L(\bv,\bw,\mcr^\imath)\in\tRZ^{\imath,<\bw}$ by Lemma \ref{lem: filtration strict}(1), for any $\bv\notin V^+$. Therefore, we have shown that $\varphi$ is a homomorphism.

Comparing the relevant bilinear forms in \eqref{eqn: antisymmetric bilinear form} and \eqref{eqn: antisymmetric bilinear form 2} used in twisting comultiplications, we see $\varphi$ is also a homomorphism with respect to the twisted multiplications.
\end{proof}

\begin{proposition}
 \label{prop:generator}
The $\Q(\tt^{1/2})$-algebra $\tRi$ is generated by $\{L(0, \e_{\sigma \ts_i}), L(\bv^{i},\bw^{i}) \mid i\in Q_0\}$.
\end{proposition}

\begin{proof}
By Theorem~\ref{thm:iso HL-Q}(1) the algebra $\tR$ is generated by $L_\cs (0, \e_{\sigma \ts_i})$, for $i\in Q_0$. Therefore, by Corollary~\ref{cor:basisGr} and Proposition~ \ref{prop:R+Rigr}, the associated graded algebra $\tR^{\imath,gr}$ is generated by the elements as listed in the statement. The proposition now follows from this by a standard filtered algebra argument.
\end{proof}

\begin{corollary}
\label{lem:center elements}
The following elements in $(\tRiZ,\cdot)$ are central:
\begin{align*}
\begin{cases}
L(\bv^{i}, \bw^{i}) L(\bv^{{\btau i}}, \bw^{i}), &  \text{ for } i \in Q_0\text{ with  } \btau i\neq i;
\\
L(\bv^{i},\bw^{i}),  &  \text{ for } i \in Q_0 \text{ with } \btau i=i.
\end{cases}
\end{align*}
\end{corollary}

\begin{proof}
Follows from Lemmas \ref{lem:cartan and positive}--\ref{lem:cartan part} and Proposition \ref{prop:generator}.
\end{proof}

\begin{definition}
  \label{def:reducedHall}
Let $\bvs=(\vs_i)\in  (\Q(\tt)^\times)^{\I}$ be such that $\vs_i=\vs_{\btau i}$ for each $i\in \I$. The \emph{reduced quantum Grothendieck ring} associated to $(Q,\btau)$, denoted by $\bf R^\imath$, is defined to be the quotient $\Q(\tt^{1/2})$-algebra of $\tRi$ by the ideal generated by the central elements
\begin{align}
\label{eqn: reduce 1}
L(\bv^{i},\bw^{i})+\tt\vs_i \; (\forall i\in \I \text{ with } \btau i=i), \quad L(\bv^{i}, \bw^{i}) L(\bv^{{\btau i}}, \bw^{i}) -\vs_i^2\; (\forall i\in \I \text{ with }\btau i\neq i).
\end{align}
\end{definition}

\subsection{Realization of $\imath$quantum groups}
  \label{subsec:mor}

Let $\cw=\cw_\I$ be the set of words in the alphabet $\I=Q_0$, i.e., $\cw$ consists of sequences $i_1\cdots i_m$ of elements in $\I$. For any $w=i_1\cdots i_m\in\cw$, define
\[
F_w=F_{i_1}\cdots F_{i_m} \in \tU^-, \qquad
B_w=B_{i_1}\cdots B_{i_m} \in \tUi.
\]
Let $\cj$ be a fixed subset of $\cw$ such that
$\{F_w\mid w\in\cj\}$ is a (monomial) basis of $ \tU^-$; such a monomial basis exists, according to Lusztig \cite{Lus90}.

\begin{lemma}  [cf. \cite{Let99, Ko14, LW19}]   \label{lem:kob}
Retain the notation above. Then $\{B_w\mid w\in\cj\}$ is a basis of $\tUi$ as the left (or right) $\tU^{\imath 0}$-module. %(This basis is called {\rm a monomial basis}.)
\end{lemma}

\begin{lemma}
There exists a $\Q(\tt^{1/2})$-algebra isomorphism
\begin{align*}
\widetilde{\kappa}: \tU^{\imath 0} &\longrightarrow  \tR^{\imath0},
\qquad
%\end{align*}
%which sends
%\begin{align*}
\tk_i\mapsto
\begin{cases}
- \tt^{-1} L(\bv^{i},\bw^{i}), & \text{ if }\btau i=i,
\\
L(\bv^{i},\bw^{i}), & \text{ if }\btau i\neq i.
\end{cases}
\end{align*}
\end{lemma}

\begin{proof}
Since both $\tU^{\imath 0}$ and $\tR^{\imath0}$ are polynomial algebras (see Proposition~\ref{prop:poly}) in the same number of generators, and $\widetilde{\kappa}:\tU^{\imath 0} \rightarrow  \tR^{\imath0}$ matches the generators, it must be an algebra isomorphism.
\end{proof}

Recall that $\ci$ is the subset of $\I$ defined in \eqref{eqn:representative}. Now we can state the main theorem of this paper (under Hypothesis~\ref{hypothesis}). 

\begin{theorem} \label{thm:UiRi}
Let $(Q, \btau)$ be a Dynkin $\imath$quiver. Then there exists an isomorphism of $\Q({\tt^{1/2}})$-algebras
 $%\begin{align*}
\widetilde{\kappa}: \tUi \stackrel{\simeq}{\longrightarrow}  \tRi,
$ %\end{align*}
which sends
\begin{align}
  \label{eqn:morphism 1}
B_i\mapsto  \frac{\tt}{1-\tt^2} L(0,\e_{\sigma \ts_i}),\text{ if }i\in\I_\btau ,
&\qquad B_i\mapsto \frac{\tt}{\tt^2-1}L(0,\e_{\sigma \ts_i}),\text{ if } i\in \I \backslash \I_\btau,,\\
\label{eqn:morphism 2}
\tk_j\mapsto L(\bv^{j}, \bw^{j}),\text{ if } \btau j\neq j,&\qquad
\tk_j\mapsto-\frac{1}{\tt} L(\bv^{j},\bw^{j}),\text{ if } \btau j=j.
\end{align}
Moreover, it induces an algebra isomorphism $\Ui \stackrel{\simeq}{\longrightarrow} {\bf R}^\imath$, which sends $B_i$ as in \eqref{eqn:morphism 1} for $i\in\I$, and  $k_j \mapsto  \frac{1}{\vs_j}L(\bv^{j},\bw^{j}) \text{ for } j \in\I\backslash \ci$.
\end{theorem}

\begin{proof}
We first show that $\widetilde{\kappa}$ is a homomorphism, that is, $\widetilde{\kappa}$ preserves the defining relations \eqref{relation1}--\eqref{relation2} for $\tUi$.
Indeed, the relation \eqref{relation1} follows from Lemmas \ref{lem:cartan and positive}--\ref{lem:cartan part};
the relation \eqref{relation5} follows from  Lemma \ref{eqn:cartan};
the relation \eqref{relationBB} follows from Lemma \ref{lem:relations of no edges};
the relation \eqref{relation3} follows from Lemma \ref{lem:serre relations for one edge} and Proposition \ref{proposition relaiton for A3}; see \eqref{eqn:serre 1};
the relation \eqref{relation2} follows from Proposition~\ref{prop:iserre split} and Proposition \ref{proposition relaiton for A3}; see \eqref{eqn:serre 2}. Therefore $\widetilde{\kappa}$ is a homomorphism.

By Proposition \ref{prop:generator} on generators for $\tRi$, $\widetilde{\kappa}$ is surjective. It remains to verify that $\widetilde{\kappa}$ is injective.
Recall the regular NKS category $\mcr$ defined in Definition \ref{def:RS for QG}. For any  $u=i_1\cdots i_m\in\cw$, define
$$
L_u:=L(0,\sigma \ts_{i_1})\cdots L(0,\sigma \ts_{i_m})\in \tRi,\quad
L^\cs_u:=L_\cs(0,\sigma \ts_{i_1})\cdots L_\cs(0,\sigma \ts_{i_m})\in \tR^+
$$
and  $\bw_u:=\sum_{j=1}^m \e_{\sigma \ts_{i_j}}$.

Let $\cj$ be a  subset of $\cw$ such that
$\{F_u\mid u\in\cj\}$ is a basis of $ \bU^-$; such a basis exists by a result of Lusztig \cite{Lus90}. Then $\{L^\cs_u\mid u\in\cj\}$ is a basis of $\tR^+$ by Theorem~ \ref{thm:iso HL-Q}(1). It follows by the morphism $\varphi$ defined in \eqref{def: varphi} and Proposition~ \ref{prop:R+Rigr} that $\{\varphi(L^\cs_u)\mid u\in\cj\}$ is linearly independent in $\tR^{\imath,gr}$. %Here we consider the twisted multiplication $\cdot$ in $\tRi$ and $\tR^{\imath,gr}$.
It follows by a standard filtered algebra argument that $\{L_u\mid u\in\cj\}$ is linearly independent in $\tRi$. Then $\widetilde{\kappa}:\tUi \rightarrow  \tRi$ is injective.
Therefore, $\widetilde{\kappa}$ is an isomorphism.

Since the isomorphism $\widetilde{\kappa}: \tUi\rightarrow  \tRi$ sends the ideal of $\tUi$ generated by elements in \eqref{eq:parameters} onto the ideal of $\tRi$ generated by elements in \eqref{eqn: reduce 1}, it induces an algebra isomorphism $\Ui\stackrel{\simeq}{\longrightarrow}  R^\imath$ as stated.
\end{proof}

By construction, the algebra $\tRi$ admits a distinguished basis $\{L_{\cs^\imath}(\bv,\bw)\}$, the basis dual to the semisimple perverse sheaf basis of $K^{gr}(\mod(\cs^\imath))$, cf. \eqref{eq:bases dual}; compare \cite{Lus90, Ka91}. It enjoys a positivity property as stated in Proposition~\ref{prop:D} in Introduction. 

%\begin{remark}
%A variant of Theorem~\ref{thm:UiRi} for $\imath$quivers of diagonal type (cf. \cite[Example~ 2.10]{LW19}) simply leads to a reformulation of Qin's theorem, i.e., Theorem~\ref{thm:iso HL-Q}(2). This explains the similarities between the two theorems.
%\end{remark}

%%%%%%%
%%%%%%%

\end{document}